\documentclass{amsart}	%%%%%%%%%%%%%%%%	[11pt]	/	{article}
\usepackage[margin=2.9cm]{geometry}	%%%%	\thispagestyle{empty}								%%%%
\usepackage{amsthm}	%%%%%%%%%%%%%%%%%%%%														%%%%
\usepackage{amsmath}	%%%%%%%%%%%%%%%%														%%%%
\usepackage{amsfonts}	%%%%%%%%%%%%%%%%														%%
\usepackage[scr=boondoxo]{mathalfa}	%%%%														%%%%
\usepackage{bbm}	%%%%%%%%%%%%%%%%%%%%														%%
\usepackage{amssymb}	%%%%%%%%%%%%%%%%														%%%%
\usepackage{mathtools}	%%%%%%%%%%%%%%%%														%%%%
\usepackage{enumitem}	%%%%%%%%%%%%%%%%														%%%%
\usepackage{tabularx}	%%%%%%%%%%%%%%%%														%%%%
\usepackage{booktabs}	%%%%%%%%%%%%%%%%														%%
\usepackage{xcolor}	%%%%%%%%%%%%%%%%%%%%	load before tikz									%%%%
\usepackage{comment}	%%%%%%%%%%%%%%%%														%%%%

\usepackage[LGR,T1]{fontenc}	%%%%%%%%														%%%%
\usepackage[utf8]{inputenc}	%%%%%%%%%%%%														%%%%
\usepackage{lmodern}	%%%%%%%%%%%%%%%%														%%%%
\usepackage[greek,english]{babel}	%%%%	The latter language is the default.					%%%%
\usepackage{alphabeta}	%%%%%%%%%%%%%%%%														%%%%
\usepackage{csquotes}	%%%%%%%%%%%%%%%%	[style=greek]										%%%%
\usepackage[doi=true,isbn=true,style=alphabetic,sorting=nyvt,giveninits=true,maxbibnames=99,backend=biber]
	{biblatex}	%%%%%%%%%%%%%%%%%%%%%%%%														%%	
\usepackage{hyperref}	%%%%%%%%%%%%%%%%														%%%%
\usepackage{xurl}	%%%%%%%%%%%%%%%%%%%%														%%%%
\usepackage{cleveref}	%%%%%%%%%%%%%%%%														%%%%
\hypersetup{colorlinks=true,linktoc=section,urlcolor={blue!80!black},citecolor={green!60!black}, breaklinks=true}%%%%
\makeatletter	%%%%%%%%%%%%%%%%%%%%%%%%	cleveref compatibility fix		--DO NOT REMOVE--	%%%%
\AtBeginDocument{%					%%%%									--DO NOT REMOVE--	%%%%
\patchcmd\label@noarg{\edef\@tempb}{\protected@edef\@tempb}{}{}}	%%		--DO NOT REMOVE--	%%%%
\makeatother	%%%%%%%%%%%%%%%%%%%%%%%%									--DO NOT REMOVE--	%%%%
\theoremstyle{plain}	%%%%%%%%%%%%%%%%														%%%%
\newtheorem{thm}{Theorem}[section]	%%%%	[<counter>]=[chapter] when dc=book					%%%%
\newtheorem{lem}[thm]{Lemma}	%%%%%%%%														%%%%
\newtheorem{prop}[thm]{Proposition}	%%%%														%%%%
\newtheorem{cor}[thm]{Corollary}	%%%%														%%%%
\theoremstyle{definition}	%%%%%%%%%%%%														%%%%
\newtheorem{dfn}[thm]{Definition}	%%%%														%%%%
\theoremstyle{remark}	%%%%%%%%%%%%%%%%														%%%%
\newtheorem{rem}[thm]{Remark}	%%%%%%%%														%%%%
\theoremstyle{plain}	%%%%%%%%%%%%%%%%														%%%%
\newtheorem*{thm*}{Theorem}	%%%%%%%%%%%%														%%%%
\newtheorem*{lem*}{Lemma}	%%%%%%%%%%%%														%%%%
\newtheorem*{prop*}{Proposition}	%%%%														%%%%
\newtheorem*{cor*}{Corollary}	%%%%%%%%														%%%%
\theoremstyle{definition}	%%%%%%%%%%%%														%%%%
\newtheorem*{dfn*}{Definition}	%%%%%%%%														%%%%
\newtheorem*{eg*}{Example}	%%%%%%%%%%%%														%%%%
\newtheorem*{not*}{Notations}	%%%%%%%%														%%%%
\theoremstyle{remark}	%%%%%%%%%%%%%%%%														%%%%
\newtheorem*{rem*}{Remark}	%%%%%%%%%%%%														%%%%
\newtheorem*{que*}{Question}	%%%%%%%%														%%%%
\theoremstyle{plain}	%%%%%%%%%%%%%%%%														%%%%
\newtheorem{letteredtheorem}{Theorem}	%%														%%%%
\newtheorem{lemabc}[letteredtheorem]{Lemma}	%%													%%%%
\numberwithin{equation}{section}	%%%%	<counter> resets on section							%%%%

\def \C{{\mathbb{C}}}				%%%%
\def \D{{\mathbb{D}}}		\def \R{{\mathbb{R}}}		%%%%
			\def \Z{{\mathbb{Z}}}	%%%%
	\def \N{{\mathbb{N}}}								%%%%
\def \cA{{\mathcal{A}}}			\def \cV{{\mathcal{V}}}	%%%%
			\def \cW{{\mathcal{W}}}	%%%%
		\def \cQ{{\mathcal{Q}}}		%%%%
\def \cD{{\mathcal{D}}}	\def \cK{{\mathcal{K}}}	\def \cR{{\mathcal{R}}}		%%%%
\def \cE{{\mathcal{E}}}				%%%%
\def \cF{{\mathcal{F}}}	\def \cM{{\mathcal{M}}}	\def \cT{{\mathcal{T}}}							%%%%
\def \cG{{\mathcal{G}}}		\def \cU{{\mathcal{U}}}							%%%%
\NewCommandCopy{\oldin}{\in}	%%%%%%%%	no break after \in									%%%%
\DeclareRobustCommand{\in}{\oldin\nolinebreak[4]}												%%%%
\makeatletter	%					%%%%	the \hrulefill										%%%%
\def\hrulefill{\noindent\leavevmode\leaders\hrule\hfill\kern\z@}								%%%%
\makeatother	%%%%%%%%%%%%%%%%%%%%%%%%														%%%%
\def \var{\,\cdot\,}	%%%%%%%%%%%%%%%%	blank function variable								%%%%
\def \such{\,:\,}	%%%%%%%%%%%%%%%%%%%%	set conditional										%%%%
\DeclareMathOperator{\tr}{tr}	%%%%%%%%	support												%%%%
\DeclareMathOperator{\dist}{dist}	%%%%	distance											%%%%
\DeclareMathOperator{\diam}{diam}	%%%%	diameter											%%%%
\DeclareMathOperator*{\esssup}{ess\,sup}	%%	essential supremum								%%	
\newcommand{\1}{\mathbbm{1}}	%%%%%%%%	characteristic function								%%	
\newcommand{\dd}{\,\mathrm{d}}	%%%%%%%%														%%%%
\DeclareMathOperator{\id}{Id}	%%%%%%%%														%%%%
\newcommand{\inline}[1]{\quad\text{#1}\quad}	%%	include text inbetween math					%%%%
\newcommand{\afterline}[1]{\qquad\text{#1~}}	%%	include text after math						%%%%
\newcommand{\inmath}[1]{\text{ #1~}}			%%	include text inside math					%%%%
\def\Xint#1{\mathchoice	%%%%%%%%%%%%%%%%	average integral symbol								%%	
   {\XXint\displaystyle\textstyle{#1}}
   {\XXint\textstyle\scriptstyle{#1}}
   {\XXint\scriptstyle\scriptscriptstyle{#1}}
   {\XXint\scriptscriptstyle\scriptscriptstyle{#1}}
   \!\int}	%%						%%%%
\def\XXint#1#2#3{{\setbox0=\hbox{$#1{#2#3}{\int}$}	
     \vcenter{\hbox{$#2#3$}}\kern-.5\wd0}}	%%
\def\dashint{\Xint-}	%%%%%%%%%%%%%%%%	average integral									%%	
\def\XLp#1{\mathchoice	%%%%%%%%%%%%%%%%	normalised L^p space								%%%%
   {\XXLp\displaystyle\textstyle{#1}}
   {\XXLp\textstyle\scriptstyle{#1}}
   {\XXLp\scriptstyle\scriptscriptstyle{#1}}
   {\XXLp\scriptscriptstyle\scriptscriptstyle{#1}}
   \!L}	%%							%%%%
\def\XXLp#1#2#3{{\setbox0=\hbox{$#1{#2#3}{L}$}	
     \vcenter{\hbox{$#2#3$}}\kern-.3\wd0}}	%%
\renewcommand{\L}{\XLp-}	%%%%%%%%%%%%	average integral									%%%%
\DeclarePairedDelimiterX{\inner}[2]{\langle}{\rangle}{#1,#2}	%%	inner product				%%%%
\DeclarePairedDelimiterX{\norm}[1]{\lVert}{\rVert}{#1}	%%	norm								%%%%
\DeclarePairedDelimiterX{\aver}[1]{\langle}{\rangle}{#1}	%%	average							%%%%
\makeatletter	%%%%%%%%%%%%%%%%%%%%%%%%							<---double angle delimeter--%%%%
\DeclareFontFamily{OMX}{MnSymbolE}{}
\DeclareSymbolFont{MnLargeSymbols}{OMX}{MnSymbolE}{m}{n}
\SetSymbolFont{MnLargeSymbols}{bold}{OMX}{MnSymbolE}{b}{n}
	\DeclareFontShape{OMX}{MnSymbolE}{m}{n}{%													%%%%
    <-6>  MnSymbolE5
   <6-7>  MnSymbolE6
   <7-8>  MnSymbolE7
   <8-9>  MnSymbolE8
   <9-10> MnSymbolE9
  <10-12> MnSymbolE10
  <12->   MnSymbolE12
}{}	%%%%%%%%%%%%%%%%%%%%%%%%%%%%%%%%%%%%														%%%%
\DeclareFontShape{OMX}{MnSymbolE}{b}{n}{%														%%%%
    <-6>  MnSymbolE-Bold5
   <6-7>  MnSymbolE-Bold6
   <7-8>  MnSymbolE-Bold7
   <8-9>  MnSymbolE-Bold8
   <9-10> MnSymbolE-Bold9
  <10-12> MnSymbolE-Bold10
  <12->   MnSymbolE-Bold12
}{}	%%%%%%%%%%%%%%%%%%%%%%%%%%%%%%%%%%%%														%%%%
\let\llangle\@undefined	%%			%%%%														%%%%
\let\rrangle\@undefined	%%			%%%%														%%%%
\DeclareMathDelimiter{\llangle}{\mathopen}%
                     {MnLargeSymbols}{'164}{MnLargeSymbols}{'164}
\DeclareMathDelimiter{\rrangle}{\mathclose}%
                     {MnLargeSymbols}{'171}{MnLargeSymbols}{'171}
\makeatother	%%%%%%%%%%%%%%%%%%%%%%%%														%%%%
\DeclarePairedDelimiterX{\convex}[1]{\llangle}{\rrangle}{#1}	%%	convex body					%%%%
\DeclareRobustCommand{\SkipTocEntry}[5]{}	%%	removes entries from toc						%%%%
\newcommand{\note}[1]{\textcolor{red}{\textbf{#1}}}	%%											%%%%
\newcommand{\bnote}[1]{\textcolor{blue}{\textbf{#1}}}	%%										%%%%
\newcommand{\cald}{\mathscr{D}}	%%%%%%%%	dyadic cubes										%%%%
\newcommand{\cals}{\mathscr{S}}	%%%%%%%%	sparse collection									%%%%
\newcommand{\calu}{\mathcal{U}}	%%%%%%%%														%%%%
\newcommand{\calm}{\mathcal{M}}	%%%%%%%%	maximal function									%%%%
\newcommand{\loc}{\textnormal{loc}}	%%%%														%%%%
\newcommand{\bcs}{\textnormal{bcs}}	%%%%														%%%%
\newcommand{\cs}{\textnormal{cs}}	%%%%														%%%%
\DeclareMathOperator{\ch}{ch}	%%%%%%%%														%%%%
\DeclareMathOperator{\RH}{RH}	%%%%%%%%														%%%%
\newcommand{\dgrid}{\mathscr{D}}	%%%%	dyadic sets											%%%%
\title{Matrix-weighted estimates beyond Calder{\'o}n--Zygmund theory}	%%						%%%%
\author	%%							%%%%														%%%%
	[S.~Kakaroumpas \and T. H. Nguyen \and D.~Vardakis]											%%
	{Spyridon Kakaroumpas \and Thu Hien Nguyen \and Dimitris Vardakis}	%%						%%
\address{Julius-Maximilians-Universit{\"a}t W{\"u}rzburg}	%%									%%%%	
\email{spyridon.kakaroumpas@uni-wuerzburg.de}	%%												%%%%
\email{thu-hien.nguyen@uni-wuerzburg.de}	%%													%%%%
\email{dimitrios.vardakis@uni-wuerzburg.de}	%%													%%%%

\begin{document}

\begin{abstract}
	We investigate matrix-weighted bounds for the sublinear non-kernel operators considered by F.~Bernicot, D.~Frey, and S.~Petermichl. We extend their result to sublinear operators acting upon vector-valued functions. First, we dominate these operators by bilinear convex body sparse forms, adapting a recent general principle due to T.~Hyt{\"o}nen. Then we use this domination to derive matrix-weighted bounds, adapting arguments of F.~Nazarov, S.~Petermichl, S.~Treil, and A.~Volberg. Our requirements on the weight are formulated in terms of two-exponent matrix Muckenhoupt conditions, which surprisingly exhibit a rich structure that is absent in the scalar case. Consequently, we deduce that our matrix-weighted bounds improve the ones that were recently obtained by A.~Laukkarinen. 
	
	The methods we use are flexible, which allows us to complement our results with a limited range extrapolation theorem for matrix weights, extending the results of P.~Auscher and J.~M.~Martell, as well as M.~Bownik and D.~Cruz-Uribe.
\end{abstract}

\maketitle

\tableofcontents

\section*{Notation}

For the reader's convenience, we provide a list of notations.

\begin{table}[h]
%\caption{Notation}
\begin{tabularx}{\textwidth}{p{0.1\textwidth}X}	%%	{@{}p{0.1\textwidth}X}	%%	%{@{}XX@{}}
\toprule
	
	$p'$ & H{\"o}lder conjugate to an exponent $p> 1$, $1/p + 1/p' = 1$ (also $1':=\infty$ and $\infty':=1$);\\
	
	$\1_{E}$ & characteristic function of a set $E$;\\
	
	$B(x,r)$	&	a ball centered at $x$ with radius $r$;	\\
	
	$r(B)$	&	the radius of the ball $B$;	\\
	
	$\cald$ & a dyadic system;\\
	
	$\dgrid(Q)$ & the dyadic subcubes of $Q\subseteq\dgrid$, $\dgrid(Q):=\{P\in\dgrid\such P\subseteq Q\}$;	\\
	
	$\ell(Q)$	&	the radius of the ball containing the cube $Q\subseteq\dgrid$;	\\
	
	$|v|$ & Euclidean norm of a vector $v\in\C^d$;\\
	
	$\norm{\cdot}_{X}$ & norm in a normed space $X$;\\
	
	$L^p(X,\mu)$	&	the Lebesgue space with exponent $p$ and measure $\mu$ on $X$;	\\
	
	$\L^p(X)$	&	the Lebesgue space with exponent $p$ and normalised measure $\frac{\mu}{\mu(X)}$ on $X$;	\\

	$L^{p}_W(X,\mu)$	&	the Lebesgue space of exponent $p$ with respect to a scalar or matrix weight $W$ and the measure~$\mu$ equipped with the norm $\norm{\cdot}_{L^p_W(X,\mu)}$;	\\
	
	$\norm{f}_{L^{p}_W(X,\mu)}$	&	the norm $\norm{f}_{L^{p}_W(X,\mu)}:=\left(\int_X|W(x)^{1/p}f(x)|^{p}\dd\mu(x)\right)^{1/p}$ for $1<p<\infty$;	\\
	
	$\dashint_{E}f\dd\mu$   & average of the function $f$ over a set $E$, defined as $\dashint_{E}f\dd\mu:=\mu(E)^{-1} \int_E f(x) \dd\mu(x)$; \\
	
	$\aver{f}_{E}$   & average of the function $f$ over a set $E$, defined as $\aver{f}_{E}:=\mu(E)^{-1} \int_E f(x) \dd\mu(x)$; \\
	
	$\convex{f}_{X}$	&	convex body ``average'' of $f$ with respect to a normed space $X$;\\
	
	$\diam(E)$	&	diameter of a set $E$;\\
	
	$\mathrm{dist}(E_1,E_2)$	&	the distance between the subsets $E_1$ and $E_2$ of a metric space $(X,d)$: $\mathrm{dist}(E_1,E_2):=\inf\{d(e_1,e_2)\such e_1\in E_1,\ e_2\in E_2\}$, sometimes denoted by just $d(E_1,E_2)$;	\\
	
	$|E|$	&	the supremum of the norms of the elements of $E\subseteq\C^d$: $|E|:=\sup\{|e|\such e\in E\}$;	\\
	
	$\mathrm{M}_{d}(\C)$  &  $d\times d$ matrices with complex entries;\\
	
	$\mathrm{P}_{d}(\C)$   &   $d\times d$ positive-definite matrices with complex entries;\\
	
	$\tr(A)$ & trace of a matrix $A\in M_{d}(\C)$;\\
	
	$|A|$ & usual spectral norm (largest singular value) of a matrix $A\in M_{d}(\C)$;\\

    $\D$ & open unit disk in $\C$ centered at $0$, $\D := \{\lambda\in\C\such |\lambda|<1\}$;\\

    $\overline{\D}$ & closed unit disk in $\C$ centered at $0$, $\overline\D := \{\lambda\in\C\such |\lambda|\leq 1\}$;\\

    $\overline{D}(a,r)$ & closed disk in $\C$ of center $a\in\C$ and radius $r>0$, $\overline{D}(a,r) := \{z\in\C\such |z-a|\leq r\}$;\\

    $\mathrm{clos}(A)$ & closure of a set $A$ in the corresponding topological space;\\

    $\overline{\mathbf{B}}$ & standard (Euclidean) closed unit ball in $\C^d$;\\

    $\cK_{\mathrm{cs}}(\C^d)$ & set of closed, convex and complex-symmetric subsets of $\C^d$;\\

    $\cK_{\bcs}(\C^d)$ & set of closed, bounded, convex, and complex-symmetric subsets of $\C^d$.\\
\bottomrule
\end{tabularx}
\end{table}

%%%%%%%%%%%%%%%%%%%%
Additionally, the notation $u \lesssim v$ means that there exists a constant $C>0$ such that $u \leq Cv$. This constant $C$ is just an absolute constant or one depending on parameters that are specified each time or that are understood by the context. Finally, the notation $u \simeq v$ means $u \lesssim v$ and $v \lesssim u$.

\section{Introduction}
\label{section:introduction}

A \emph{scalar weight} $w$ on $\R^n$ is a locally integrable function $w:\R^n\to\R$ with $w(x)>0$ for a.e.~$x\in\R^n$. For any weight $w$ on $\R^n$ and $1<p<\infty$ we denote $L^{p}(w):=L^{p}(w(x)\dd x)$. The question of establishing bounds of the form
\begin{equation*}
    \norm{Tf}_{L^{p}(w)}\leq C\norm{f}_{L^{p}(w)},
\end{equation*}
for $1<p<\infty$ and classical operators $T$ like Calder{\'o}n--Zygmund operators and the Hardy--Littlewood maximal function (where $C$ is some constant depending on $p$, the weight $w$, the parameters of the ambient space and the operator but \emph{not} the function $f$) has a long history. The field began with the seminal paper of Helson--Szeg{\"o} \cite{Helson_Szego_1960} on stationary stochastic processes. Subsequent works \cite{Muckenhoupt_1972, Muckenhoupt_Wheeden_1976, Hunt_Muckenhoupt_Wheeden_1973, Coifman_Fefferman_1974} showed that the correct condition on the weight for such bounds to hold is
\begin{equation*}
    [w]_{A_p}:=\sup_{Q}\left(\dashint_{Q}w(x)\dd  x\right)\left(\dashint_{Q}w(x)^{-1/(p-1)}\dd x\right)^{p-1}<\infty,
\end{equation*}
where the supremum is taken over all cubes $Q\subseteq\R^n$ (with faces parallel to the coordinate hyperplanes). The exact dependence of the constant $C$ on $[w]_{A_p}$ turned out to be significantly harder and was much later completely described \cite{Wittwer_2000, Petermichl_H2007, Petermichl_R2007, Hytonen_2012}. A desire to give a simple proof of the exact dependence led to the invention of the method of \emph{sparse domination} by L.~Lerner \cite{Lerner_2013} which was soon further developed and refined \cite{Lerner_2016, Lacey_2017} and also applied to rough singular integral operators \cite{Conde_Alonso_2017, Di_Plinio_2021}. Of particular importance for the present article are the results by F.~Bernicot, D.~Frey, and S.~Petermichl \cite{BFP2016}, which give a bilinear sparse domination principle for very general non-kernel operators under mild regularity conditions. This, in turn, was used to derive sharp weighted bounds for such operators and doubling measure metric spaces in general. Here, we state the two main results from \cite{BFP2016}, and refer the reader to \Cref{section:preliminaries} below for the notation and precise definitions.

\begin{letteredtheorem}[\protect{\cite[Theorem 5.7]{BFP2016}}]
\label{BFP57}
    Let $(M,\mu)$ be a doubling measure metric space. Let $1\leq p_0<2<q_0\leq\infty$, and take $p\in (p_0,q_0)$. Also, let $T:L^{p}(M,\mu)\to L^{p}(M,\mu)$ be a bounded sublinear operator satisfying the assumptions in \Cref{section:preliminaries} with respect to $p_0$ and $q_0$. Let $\cald$ be a dyadic system in $M$ as in Subsection~\ref{subsec:dyadic}. Let also $0<\varepsilon<1$. Then, there exists a constant $C=C(T,\mu,p_0,q_0,\varepsilon)>0$ such that for all $f\in L^p(M,\mu)$ and $g\in L^{p'}(M,\mu)$ that are supported on $\lambda Q_0$ for some cube $Q_0\in \cald$ and sufficiently large (depending on $\dgrid$) $\lambda \geq 2$, there exists an $\varepsilon$-sparse collection $\cals\subseteq \cald$ (depending on $T,p_0,q_0,f$ and~$g$) with
	\[\left|\int_{Q_0} T f \cdot g \dd\mu\right|\leq C\sum_{P\in \cals}μ(P)\left(\dashint_{\lambda P}|f|^{p_0}\dd\mu\right)^{1/p_0}\left(\dashint_{\lambda P}|g|^{q_0'}\dd\mu\right)^{1/q_0'}.\]
\end{letteredtheorem}

\begin{letteredtheorem}[\protect{\cite[Theorem 6.4]{BFP2016}}]
\label{BFP64}
    Let $(M,\mu)$ be a doubling measure metric space. Let $1\leq p_0<q_0\leq\infty$ and $p\in (p_0, q_0)$. Let $\cald$ be a dyadic system in $M$ as in Subsection~\ref{subsec:dyadic} and let $\lambda \geq 2$. For $\varepsilon \in (0,1)$, let $\cals\subseteq\cald$ be an $\varepsilon$-sparse collection. Let $w$ be a weight on $M$. Then, there exists a constant $C=C(\mu,p,p_0,q_0,\varepsilon)$ such that the following estimate holds:
    \begin{align*}
        \sum_{P \in \cals} \mu(P)\left( \dashint_{\lambda P}|f|^{p_0} \dd\mu \right)^{\frac{1}{p_0}}\left(\dashint_{\lambda P}|g|^{q_0'} \dd\mu \right)^{\frac{1}{q_0'}}\leq C([w]_{A_{p/p_0}} [w]_{\RH_{(q_0/p)'}})^{\alpha}\norm{f}_{L^{p}_{w}(M,\mu)}\norm{g}_{L^{p'}_{w'}(M,\mu)},
    \end{align*}
    for all $f\in L^{p}_{w}(M,\mu)$ and $g\in L^{p'}_{w'}(M,\mu)$, where $w':=w^{-1/(p-1)}$ and
    \begin{align}
        \alpha := \max \left\lbrace\frac{1}{p-p_0}, \frac{q_0 - 1}{q_0 - p} \right\rbrace.
    \end{align}
\end{letteredtheorem}

A major tool for extending given or already proven $L^{p_0}$-weighted bounds for some fixed $p_0$ to $L^{p}$-weighted bounds for all $p$ in some appropriate range is the method of \emph{extrapolation}, introduced by Jos{\'e} L.~Rubio de Francia \cite{Rubio_de_Francia_1984}. Several works followed further developing this method, until J.~Duoandikoextea \cite{Duoandikoetxea_2011} proved a sharp version of it. We refer to the book \cite{Book_Extrapolation} for a history of the subject, as well as detailed proofs of various forms of extrapolation and related results. The original Rubio de Francia Extrapolation assumes an already known $L^{p_0}(w)$-bound, for all $A_{p_0}$ weights $w$, for some fixed $1<p_0<\infty$, and yields $L^{p}(w)$-bounds, for all $A_{p}$ weights $w$, for all $1<p<\infty$. Of particular importance for the present article is a form of \emph{limited range extrapolation} established by P.~Auscher and J.~M.~Martell \cite[Theorem 4.9]{Auscher_Martell_I_2007} (we refer to \Cref{section:matrix-weight} for the notation):

\begin{letteredtheorem}[\protect{\cite[Theorem 4.9]{Auscher_Martell_I_2007}}]
\label{thm:original_limited_range_extrapolation}
Let $\cF$ be a collection of pairs $(f,g)$ of measurable functions on $\R^n$ such that neither $f$ nor $g$ is a.e.~equal to $0$. Let $0<p_0<q_0\leq\infty$. Assume that there exists $p$ with $p_0\leq p\leq q_0$, and $p<\infty$ if $q_0=\infty$, such that for all $(f,g)\in\cF$ and for all weights $w\in A_{p/p_0}\cap RH_{(q_0/p)'}$ with $\norm{f}_{L^{p}(w)}<\infty$,
\[\norm{f}_{L^{p}(w)}\leq C\norm{g}_{L^{p}(w)}\] holds
for some constant $C$ depending on $n,p,p_0,q_0$ and $w$.

Then, for all $p_0<q<q_0$, for all $(f,g)\in\cF$ and for all weights $w\in A_{q/p_0}\cap RH_{(q_0/q)'}$ with $\norm{f}_{L^{q}(w)}<\infty$,
\[\norm{f}_{L^{q}(w)}\leq C\norm{g}_{L^{q}(w)}\] holds
for some constant $C$ depending on $n,p,q,p_0,q_0$ and $w$.
\end{letteredtheorem}

Questions regarding multivariate stationary stochastic processes lead naturally to the investigation of the boundedness of operators on \emph{matrix-weighted} Lebesgue spaces. More precisely, a $d\times d$ \emph{matrix weight} $W$ on $\R^n$ is a locally integrable function $W:\R^n\to M_d(\C)$ on $\R^n$ taking values in the set of $d\times d$ complex matrices such that $W(x)$ is a positive definite matrix for a.e.~$x\in\R^n$. Suppose $T$ is a \emph{linear} operator, which acts on a space of scalar-valued functions and goes into a space of scalar-valued functions. In that case, we can consider (abusing notation) the operator $T$ acting on a space of $d$-dimensional vector-valued functions by $T(f):=(Tf_1,\ldots, Tf_d)$. If we denote $\norm{f}_{L^{p}(W)}:=\left(\int_{\R^n}|W(x)^{1/p}f(x)|^{p}\dd x\right)^{1/p}$ for $1<p<\infty$, then the question of matrix-weighted boundedness reads as
\begin{equation*}
    \norm{Tf}_{L^{p}(W)}\leq C\norm{f}_{L^{p}(W)},
\end{equation*}
where $T$ is the vector-valued extension of a linear operator. The correct condition on the weight $W$ for the above bound to hold was established by F.~Nazarov, S.~Treil and A.~Volberg over several works \cite{Hunt_Bellman_1996, Angle_Past_Future_1997, Volberg_1997} and is the following:
\begin{equation*}
    [W]_{A_p}:=\sup_{Q}\dashint_{Q}\left(\dashint_{Q}|W(x)^{1/p}W(y)^{-1/p}|^{p'}\dd y\right)^{p/p'}\dd x<\infty,
\end{equation*}
where the supremum is taken over all cubes $Q\subseteq\R^n$ (with faces parallel to the coordinate hyperplanes). The class of such weights class was further studied by M.~Bownik \cite{Bownik_2001}, while M.~Christ and M.~Golderg \cite{Christ_Goldberg_2001, Goldberg_2003} studied a Hardy--Littlewood type maximal function for vector-valued functions in the presence of matrix weights.

The question of quantifying matrix-weighted bounds is much harder and deeper. A major breakthrough was achieved by F.~Nazarov, S.~Petermichl, S.~Treil and A.~Volberg \cite{NPTV2017}; they proved a \emph{pointwise convex body sparse domination} of Calder{\'o}n--Zygmund operators and established the bound
\begin{equation*}
    \norm{Tf}_{L^{2}(W)}\leq C[W]_{A_2}^{3/2}\norm{f}_{L^2(W)},
\end{equation*}
where the constant $C$ does not depend on the weight. These bounds were a bit later generalized to any exponent $1<p<\infty$ \cite{Cruz_Uribe_OFS_2018}. Very recently, K.~Domelevo, S.~Petermichl, S.~Treil, and A.~Volberg \cite{Matrix_A2} proved that when $T$ is the Hilbert transform on $\R$, the exponent $3/2$ in the above estimate cannot be improved.

Matrix-weight bounds for rough singular integral operators were established in \cite{Di_Plinio_2021}, while \cite{Muller_Rivera_Rios_2022} established a \emph{bilinear convex body sparse domination} for such operators. More recently, a quite general principle for deriving a bilinear convex body sparse domination from \emph{scalar} sparse domination was provided by Hyt{\"o}nen \cite{Hyt2024}. Matrix-weighted bounds for such bilinear convex body sparse forms were already present in special cases in \cite{Di_Plinio_2021, Muller_Rivera_Rios_2022} and were obtained in full generality very recently by Laukkarinen \cite{Laukkarinen_2023}.

An analog of the Rubio de Francia extrapolation for matrix-weighted estimates seemed out of reach for many years, until the breakthrough by M.~Bownik and D.~Cruz-Uribe \cite{Cruz_Uribe_Bownik_Extrapolation}, which developed a whole theory of \emph{convex body valued} operators. See further \cite{Cruz_Uribe_Extrapolation} for an exposition of the results in \cite{Cruz_Uribe_Bownik_Extrapolation}.

\subsection{Main results}

The main goal of this paper is to prove matrix-weighted bounds for the sublinear operators considered in \cite{BFP2016}. Our main result is the following:

\begin{thm}
\label{thm:main_theorem}
    Let $(M,\mu)$ be a doubling measure metric space. $1\leq p_0< 2< q_0\leq \infty$ and take $p\in(p_0,q_0)$. Let $\mathbf{T}:L^{p}(M,\mu;\C^d)\to L^{p}_{\cK}(M,\mu)$ be a sublinear operator compatible with a bounded sublinear operator $T:L^{p}(M,\mu)\to L^{p}(M,\mu)$ satisfying the assumptions in \Cref{section:preliminaries}  with respect to $p_0$ and $q_0$ (see Definitions \ref{dfn:sublinear_convex_set} and \ref{dfn:compatibility_scalar_convex_body}). Set
    \begin{equation*}
        t:=\frac{p}{p_0},\quad s:=\frac{q_0}{p},\quad a:=\frac{t'p}{t},\quad b:=s'p.
    \end{equation*}
    Let $W$ be a $d\times d$ matrix weight on $M$ such that
    \begin{equation*}
        [W^{s'}]_{A_{a',b}}:=\sup_{B}\dashint_{B}\left(\dashint_{B}|W(x)^{1/p}W(x)^{-1/p}|^{a}\dd\mu(y)\right)^{b/a'}\dd\mu(x) < \infty,
    \end{equation*}
    where the supremum ranges over all open balls $B\subseteq M$. Then, there exists a constant $K>0$ depending only on $T,\mu,p_0,q_0,p$ and $d$ such that
    \begin{equation*}
        \norm{\mathbf{T}f}_{L^{p}_{\cK,W}(M,\mu)}\leq K\cdot ([W^{s'}]_{A_{a',b}})^{\alpha}\norm{f}_{L^{p}_{W}(M,\mu)},
    \end{equation*}
    for all compactly supported functions $f\in L^{p}(M,\mu;\C^d)$, where $W':=W^{-1/(p-1)}$ and
    \begin{equation*}
        \alpha:=\frac{1}{s'(p-p_0)}+\frac{1}{p'}.
    \end{equation*}   
\end{thm}

We prove \Cref{thm:main_theorem} using a duality argument (explained in detail in \ref{subsec:estimate_operator}) in two steps. First, in \Cref{section:bilinear_convex_body_sparse_domination} we extend \Cref{BFP57} and prove a bilinear convex body sparse domination for the operators in question:

\begin{thm}
\label{thm:bilinear_convex_body_sparse_domination}
	$1\leq p_0< 2< q_0\leq \infty$, and fix $p\in(p_0,q_0)$. Let $\mathbf{T}:L^{p}(M,\mu;\C^d)\to L^{p}_{\cK}(M,\mu)$ be a sublinear operator compatible with a bounded sublinear operator $T:L^{p}(M,\mu)\to L^{p}(M,\mu)$ satisfying the assumptions in \Cref{section:preliminaries} (see Definitions \ref{dfn:sublinear_convex_set} and \ref{dfn:compatibility_scalar_convex_body}). Let $\cald$ be a dyadic system in $M$ as in \Cref{subsec:dyadic}. Let also $0<\varepsilon<1$. Then, there exists a constant $C=C(T,\mu,p_0,q_0,\varepsilon)>0$ such that for all $f\in L^p(M,\mu;\C^d)$ and $g\in L^{p'}(M,\mu;\C^d)$, whose coordinate functions are all supported on $\lambda Q_0$ for some cube $Q_0\in \cald$ and sufficiently large (depending on $\dgrid$) $\lambda\geq 2$, there exists an $\varepsilon$-sparse collection $\cals\subseteq \cald$ (depending on $T,p_0,q_0,d,f$ and $g$) with
    \[\int_{Q_0} |\mathbf{T}f \cdot g| \dd\mu\leq Cd^{3/2}\sum_{P\in \cals}μ(P)\Big|\convex{f}_{\L^{p_0}(\lambda P;\C^d)}\cdot\convex{g}_{\L^{q_0'}(\lambda P;\C^d)}\Big|.\]
\end{thm}

%\begin{thm}
%\label{thm:bilinear_convex_body_sparse_domination}
	%Let $(M,\mu)$ be a doubling measure metric space. Let $1\leq p_0<2<q_0\leq\infty$ and take $p\in(p_0,q_0)$. Consider a bounded linear operator $T:L^{p}(M,\mu)\to L^{p}(M,\mu)$ satisfying the assumptions in \Cref{section:preliminaries} with respect to $p_0$ and $q_0$. Let $\cald$ be a dyadic system in $M$ as in Subsection~\ref{subsec:dyadic}. Also, let $0<\varepsilon<1$. Then, there exists a constant $C=C(T,\mu,p_0,q_0,\varepsilon)>0$ such that for all $f\in L^p(M,\mu;\C^d)$ and $g\in L^{p'}(M,\mu;\C^d)$, whose coordinate functions are all supported on $\lambda Q_0$ for some cube $Q_0\in \cald$ and sufficiently large (depending on $\dgrid$) $\lambda \geq 2$, there exists an $\varepsilon$-sparse collection $\cals\subseteq \cald$ (depending on $T,p_0,q_0,d,f$ and $g$) with
	%\[\left|\int_{Q_0} T f \cdot g \dd\mu\right|\leq Cd^{3/2}\sum_{P\in \cals}μ(P)\big|\convex{f}_{\L^{p_0}(\lambda P;\C^d)}\cdot\convex{g}_{\L^{q_0'}(\lambda P;\C^d)}\big|.\]
%\end{thm}

Then, in \Cref{section:matrix_weighted_bounds} we prove matrix-weighted bounds for bilinear convex body sparse forms of the type appearing in \Cref{thm:bilinear_convex_body_sparse_domination}:

\begin{thm}
    \label{thm:matrix_weighted_bound}
    Let $1\leq p_0<q_0\leq\infty$ and $p\in(p_0,q_0)$. Let $\cald$ be a dyadic system in $M$ as in Subsection~\ref{subsec:dyadic} and let $\lambda\geq2$. Let $0<\varepsilon<1$ and let $\cals\subseteq\cald$ be an $\varepsilon$-sparse collection. Set
    \begin{equation*}
        t:=\frac{p}{p_0},\quad s:=\frac{q_0}{p},\quad a:=\frac{p}{t-1},\quad b:=s'p.
    \end{equation*}
    Let $W$ be a $d\times d$ matrix weight on $M$ such that $[W^{s'}]_{A_{a',b}}<\infty$. Then, there exists a constant $C=(\mu,d,p_0,q_0,p)$ such that for all $f\in L^{p}_{W}(M,\mu)$ and all $g\in L^{p'}_{W'}(M,\mu)$ the following estimate holds:
	\begin{equation*}
		\sum_{P\in\cals}\mu(P)\big|\convex{f}_{\L^{p_0}(\lambda P;\C^d)}\cdot\convex{g}_{\L^{q_0'}(\lambda P;\C^d)}\big|\leq\frac{C}{\varepsilon}([W^{s'}]_{A_{a',b}})^{\alpha}\norm{f}_{L^{p}_{W}(M,\mu)}\norm{g}_{L^{p'}_{W'}(M,\mu)},
	\end{equation*}
    where $W':=W^{-1/(p-1)}$ and
    \begin{equation*}
        \alpha:=\frac{1}{s'(p-p_0)}+\frac{1}{p'}.
    \end{equation*}
\end{thm}

In \Cref{section:matrix-weight}, we look more closely at conditions on the weights like the ones imposed in \Cref{thm:matrix_weighted_bound} in the case of $\R^n$. In contrast to the scalar case \cite{Johnson_Neugebauer_1991, Auscher_Martell_I_2007, BFP2016}, the relation between these different conditions in the matrix case turns out to be more complicated.

On the one hand, some of the inclusions that were true in the scalar case continue to hold for matrix weights (we refer to \Cref{section:matrix-weight} for precise definitions):
\begin{prop}
\label{prop:equivalence}
    Let $1<t,s<\infty$. Let $W$ be a $d\times d$ matrix weight on $\R^n$. Then
    \begin{equation*}
        \max([W]_{A_{t}},[W]_{RH_{t,s}})^{s}\lesssim_{d,t,s}[W^{s}]_{A_{t,st}}\lesssim_{d,t,s}([W]_{A_{t}}[W]_{RH_{t,s}})^{s}.
    \end{equation*}
\end{prop}

On the other hand, several two-exponent Muckenhoupt classes, which in the scalar case just coincide, turn out to be quite different for matrix weights:
\begin{prop}
\label{prop:counterexample_statement}
    Let $1<p_i\leq q_i<\infty$, $i=1,2$ with $p_1\leq p_2$ and $\frac{q_1}{p_1'}=\frac{q_2}{p_2'}$. Then, one has that
	\begin{equation*}
        [W]_{A_{p_1,q_1}}\leq[W]_{A_{p_2,q_2}}
    \end{equation*}
    for all $d\times d$ matrix weights $W$ on $\R^n$. In the case when $p_1<p_2$, one can find a $2\times 2$ matrix weight $W$ on $\R$ with $[W]_{A_{p_1,q_1}}<\infty$ and $[W]_{A_{p_2,q_2}}=\infty$.
\end{prop}

Our results are slightly stronger; see \Cref{prop:weightscounterexample} and \Cref{prop:weights_counterexample_2}. It is worth mentioning that our analysis in \Cref{section:matrix-weight} also covers several endpoint cases.

We note that \Cref{prop:equivalence} and \Cref{prop:counterexample_statement} show that our matrix-weighted bounds in \Cref{thm:matrix_weighted_bound} improve the ones obtained recently in \cite{Laukkarinen_2023}; we explain this in \Cref{subsection:negative_results_on_weights}.

Finally, in \Cref{section:extrapolation} we prove the following limited range extrapolation theorem for matrix weights, where we refer to the notation and precise definitions:

\begin{thm}
\label{thm:limited_range_extrapolation_matrix_weights}
Let $\cF$ be a collection of pairs $(F,G)$ of measurable $\cK_{\bcs}(\C^d)$ valued functions on $\R^n$ such that neither $F$ nor $G$ is a.e.~equal to $\lbrace0\rbrace$. Let $0<p_0<q_0\leq\infty$. Assume that for some fixed $p\in[p_0,q_0]$, there exists an increasing function $K_{p}$ such that one of the following holds:
\begin{enumerate}
	\item If $p<q_0$, then
	\begin{equation*}
		\norm{F}_{L^{p}_{\cK}(W)}\leq K_{p}([W^{s(p)'}]_{A_{r(p)}})\norm{G}_{L^{p}_{\cK}(W)},\quad \text{for all }(F,G)\in\cF\text{ with }\norm{F}_{L^{p}_{\cK}(W)}<\infty,
	\end{equation*}
	for all $d\times d$ matrix weights $W$ on $\R^n$ with $[W^{s(p)'}]_{A_{r(p)}}<\infty$.
	\item If $p=q_0<\infty$, then
	\begin{equation*}
		\norm{F}_{L^{p}_{\cK}(W)}\leq K_{p}([W^{t(p)'/t(p)}]_{A_{\infty}})\norm{G}_{L^{p}_{\cK}(W)},\quad \text{for all }(F,G)\in\cF\text{ with }\norm{F}_{L^{p}_{\cK}(W)}<\infty,
	\end{equation*}
		for all $d\times d$ matrix weights $W$ on $\R^n$ with $[W^{t(p)'/t(p)}]_{A_{\infty}}<\infty$.
		\item If $p=q_0=\infty$, then
	\begin{equation*}
		\norm{F}_{L^{p}_{\cK}(W)}\leq K_{p}([W^{1/p_0}]_{A_{\infty}})\norm{G}_{L^{p}_{\cK}(W)},\quad \text{for all }(F,G)\in\cF\text{ with }\norm{F}_{L^{p}_{\cK}(W)}<\infty,
	\end{equation*}
	for all $d\times d$ matrix weights $W$ on $\R^n$ with $[W^{1/p_0}]_{A_{\infty}}<\infty$.
	\end{enumerate}
	Then, for all $q\in(p_0,q_0)$,
	\begin{equation*}
		\norm{F}_{L^{q}_{\cK}(W)}\leq K_{q}([W^{s(q)'}]_{A_{r(q)}})\norm{G}_{L^{q}_{\cK}(W)},\quad \text{for all }(F,G)\in\cF\text{ with }\norm{F}_{L^{q}_{\cK}(W)}<\infty,
	\end{equation*}
	for all $d\times d$ matrix weights $W$ on $\R^n$ with $[W^{s(q)'}]_{A_{r(q)}}<\infty$, where
	\begin{equation*}
		K_{q}([W^{s(q)'}]_{A_{r(q)}})=c_1K_{p}(c_2[W^{s(q)'}]_{A_{r(q)}}^{\alpha(p,q)}),
	\end{equation*}
	the exponent $\alpha(p,q)$ is given by \eqref{eq:extrapolation_exponent}, and the constants $c_1,c_2>0$ depend only on $n,d,p_0,q_0,p$ and $q$.
\end{thm}

\subsection*{Acknowledgements} The authors are grateful to Stefanie Petermichl for suggesting the problems considered in this paper.

%%%%%%%%%%%%%%%%%%%%%%%%%%%%%%%%%%%%%%%%		section	%%%%%%%%%%%%%%%%%%%%%%%%%%%%%%%%%%%%%%%%
\section{Set up and context of the problem}
\label{section:preliminaries}

We repeat some notions from \cite{BFP2016} for clarity and the reader's convenience.

\smallskip

Consider a space $M$, and an exponent $p\in[1,\infty)$. Take also a subset $E\subseteq M$, and a Borel measure $m$ on $E$. Then, $L^p(E,m)$ is the Lebesgue space of ($m$-a.e.~equivalence classes of) measurable functions $f:E\to\C$ satisfying $\int_{E}|f(x)|^{p}\dd m(x)<\infty$; this space is equipped with the norm
\[\norm{f}_{L^p(E,m)} = \left(\int_E |f|^p \dd m \right)^{\frac{1}{p}}.\]
For $p=\infty$ and $m(E)\not=0$, $L^\infty(E,m)$ is the Lebesgue space of ($m$-a.e.~equivalence classes of) measurable functions $f:E\to\C$ satisfying $\esssup_{x\in E}|f(x)|<\infty$; this space is equipped with the norm
\[\norm{f}_{L^\infty(E,m)} = \esssup_{x\in E}|f(x)|.\]
Also, $L^{p}_{\loc}(M,m)$ (for $1\leq p<\infty$) is the space consisting of all ($m$-a.e.~equivalence classes of) measurable functions $f:M\to\C$ with $\int_{E}|f(x)|^{p}\dd m(x)<\infty$ for any bounded Borel subset $E$ of $M$.

\smallskip

For the present text, $p_0$, $q_0$, $p$, and $q$ stand for exponents. Also, $p'$ (and similarly for all other exponents that might appear) denotes the H{\"o}lder conjugate exponent of $p$, that is $\frac{1}{p}+\frac{1}{p'}=1$. We also set $1':=\infty$ and $\infty':=1$.

\subsection{Ambient doubling space} \label{subsec:ambient} Let $(M,d,\mu)$ be a locally compact separable metric space equipped with a non-zero Borel measure $\mu$ which is finite on compact sets. We mostly refer to this space as $(M,\mu)$.

For all $x\in M$ and all $r>0$, we denote by $B(x,r)\subseteq M$ the open ball for the metric $d$ with centre~$x$ and radius $r$. Whenever $B$ is an open ball in $M$ we might write $r(B)$ (instead of simply $r$) for its radius. If $B$ is an open ball and $\lambda>0$, we denote by $\lambda B$ the ball concentric with $B$ of radius $\lambda r$. In the sequel, by ``ball'' we always mean \emph{open} ball unless otherwise specified.

We assume that $(M, d, \mu)$ satisfies the \emph{volume-doubling property}, that is there exists  a constant $C>0$ such that
    \[\mu(B(x, 2r)) \leq C\mu(B(x,r)) \afterline{for all} x \in M \inmath{and} r>0.\]
We denote by $C_{\mu}$ the smallest such constant $C$. Obviously, $C_{\mu}\geq1$. An immediate consequence of the doubling property is that non-empty open sets have a non-zero $\mu$-measure.

From this assumption, we also deduce
\begin{equation}\label{prop:nudoubling}
    \mu(B(x,r)) \leq C_{\mu}\left(\frac{r}{s}\right)^{\nu} \mu(B(x,s)) \afterline{for all} x \in M \inmath{and} r \geq s > 0,
\end{equation}
where $\nu:=\log_2(C_{\mu})\geq0$, since $B(x,r)\subseteq B(x,2^{n+1}s)$ for $n:=\lfloor \log_{2}(r/s)\rfloor$. As a consequence, it also holds that
     \[\mu(B(x,r)) \leq C_{\mu}\left(\frac{d(x,y) + r}{s}\right)^{\nu} \mu(B(y,s)), \afterline{for all} x, y \in M \inmath{and} r \geq s > 0.\]
because $B(x,r)\subseteq B(y,d(x,y)+r)$. From the properties above, one can see that balls with a non-empty intersection and comparable radii have comparable measures. In particular, $C_{\mu}>1$, so $\nu>0$, unless $M$ consists of a single point, a case we exclude from our considerations. In the rest of the text, we ignore the dependence on the implied constant $C_{\mu}$, unless specified otherwise.

Another important implication of the volume-doubling properties is that any ball of radius $r$ can be covered by at most a fixed number (which depends only on the measure $\mu$) of balls with radius~$\frac{r}{2}$. This is a consequence of the Vitali covering lemma.

\subsection{Dyadic systems} \label{subsec:dyadic} Here, we set up the basis of the ``dyadic analysis'' which is very important later on. For more details, one can refer to \cite{HytKai2012} where the following construction first appeared. A nice exposition can also be found in \cite{Per2019}. We note that our setting has slightly stronger assumptions on the ambient space $(M,d,\mu)$ than the ones appearing in \cite{HytKai2012}.

\begin{thm} \label{thm:dyadic_system}
	Let $(M,d,μ)$ be as in \Cref{subsec:ambient}. Suppose the constants $0< c_0\leq C_0< \infty$ and $\delta\in (0,1)$ satisfy $12C_0\delta\leq c_0$. Then, for each $k\in \Z$ there exists a countable set of indices $\cA_k$, and for each $i\in \cA_k$ there exists an open set $Q_i^k\subseteq M$, with the following properties:
	\begin{enumerate}
		\item \emph{(nestedness)} If $l\geq k$, then either $Q^l_i\subseteq Q^k_j$ or $Q^l_i\cap Q^k_j= \emptyset$ (for all $i\in \cA_l$ and all $j\in \cA_k$).
		\item \emph{(covering)} For all $k\in \Z$, there exists a zero-measure set $N_k$ (i.e.~$\mu(N_k)=0$) such that
		\[M= \bigsqcup_{i\in \cA_k} Q^k_i\sqcup N_k.\]
		\item \emph{(inner/outer balls)} For every $k\in \Z$ and $i\in \cA_k$ there exist points $z^k_i\in M$ such that
		\[B(z^k_i,{\textstyle \frac{c_0}{3}}\delta^k)\subseteq Q^k_i\subseteq B(z^k_i,2C_0\delta^k).\]
		\item If $l\geq k$ and $Q^l_i\subseteq Q^k_j$ (for $i\in \cA_l$ and $j\in \cA_k$), then $B(z^l_i,2C_0\delta^l)\subseteq B(z^k_j,2C_0\delta^k)$.
	\end{enumerate}
\end{thm}

\begin{rem} \label{rem:point_distancing}
	Fix $k\in \Z$. The points $z^k_j$ mentioned in the above \Cref{thm:dyadic_system} satisfy two rather important properties:
	\begin{gather}
		\label{eq:points_are_not_too_close}
		d(z^k_i,z^k_j)\geq c_0\delta^k \afterline{whenever} i\not= j \text{ for } i,j\in \cA_k\\
	\intertext{and also}
		\label{eq:points_are_not_far_from_other_points}
		\min_{i\in \cA_k} d(x,z^k_i)< C_0\delta^k \afterline{for all} x\in M.
	\end{gather}
\end{rem}

The sets $Q^k_i$ are called \emph{dyadic sets} or \emph{dyadic cubes}, and sometimes ---when there is no risk of confusion--- we call them simply \emph{cubes}. For a fixed $k\in \Z$, the family $\dgrid_k:= \bigcup_{i\in \cA_k}\{Q^k_i\}$ is called a \emph{generation} (of dyadic sets), and the collection $\dgrid:= \bigcup_{k\in \Z} \dgrid_k$ of all such families is called a \emph{dyadic system} subject to the parameters $c_0$, $C_0$, and $\delta$.

The nestedness property implies that there is no partial overlap between generations, and as a consequence, each cube has unique ``ancestors'' in earlier (i.e. for smaller $k$) generations. In other words, if $Q\in \dgrid_k$, then there is a unique cube $\hat{Q}\in \dgrid_{k-1}$ such that $Q\subseteq \hat{Q}$; we call $\hat{Q}$ the \emph{father} or \emph{parent} of $Q$, and $Q$ the \emph{child} of $\hat{Q}$.
The doubling properties of the measure imply that each cube uniformly has finitely many children.
%%% Original text: "Moreover, any ball of radius $r$ can be covered by at most $A$ (a fixed constant) balls of radius $r/2$"
%%% Remark: We already mentioned this above at the end of subsection 2.1
%%%% Continuation of original text: ", sometimes referred to as the \emph{finite overlap property}."
%%% Remark: I do not think this is what people usually call "finite overlap".

\medskip

Moreover, any collection of dyadic sets $\cG \subseteq \dgrid$ is called a \emph{dyadic collection} or a \emph{dyadic family}.
If $\cF$ and $\cG$ are two non-empty dyadic collections, we say that $\cG$ \emph{covers} (or is a \emph{covering} of) $\cF$ whenever $\bigcup_{Q\in\cF} Q \subseteq \bigcup_{P\in\cG}P$ and no dyadic set in $\cG$ is strictly contained in any dyadic set in $\cF$. This definition is a variant by Hyt{\"o}nen (see \cite[Proposition 4.6]{Hyt2024}) of \cite[Definition 3.1]{NPTV2017}.

%%%%%%%%%%%%%%%%%%%%%%%%%%%%%%%%%%%%%%%%	Collections of Dyadic Systems	%%%%%%%%%%%%%%%%%%%%%%%%
%We also need \cite[Theorem 4.1]{HytKai2012}, which we adapt to our notation and ambient space:
%
%\begin{thm}
%	Let $(M,d,\mu)$ be as in \Cref{subsec:ambient}. Suppose the constants $0< c_0\leq C_0< \infty$ and $\delta\in (0,1)$ satisfy $12C_0\delta\leq c_0$, and also that $\delta$ is sufficiently (depending on $\mu$) small. Then, there exist a positive integer $K=K(c_0,C_0,\delta,\mu)$ and a finite collection $\dgrid^s$ ($s=1,\dots,K$) of dyadic systems (with parameters $c_0$, $C_0$, and $\delta$) satisfying the properties of \Cref{thm:dyadic_system}.
%	
%	Additionally, there exists $C_1=C_1(c_0,C_0,\delta)\in (0,\infty)$ such that for every ball $B(x,r)\subseteq M$ we can find $s=1,\dots,K$ and $Q\in \dgrid^s$ with
%	\begin{equation}
%		B(x,r)\subseteq Q \inline{and} \diam(Q)\leq C_1r.
%	\end{equation}
%\end{thm}
%
%We denote the above collection of dyadic systems by
%\[\widetilde{\dgrid}:= \bigcup_{s=1}^K \dgrid^s.\]
%
%\smallskip
%%%%%%%%%%%%%%%%%%%%%%%%%%%%%%%%%%%%%%%%	Collections of Dyadic Systems	%%%%%%%%%%%%%%%%%%%%%%%%

We denote by $\dgrid(Q)$ the collection of all cubes $Q'\in \dgrid$ contained in $Q$, that is
\[\textstyle \dgrid(Q)=\bigcup_{\substack{Q'\in \dgrid\\ Q'\subseteq Q}}\{Q'\}.\]
%	When $Q\in \widetilde{\dgrid}$, then $Q\in \dgrid^s$ for some $s=1,\dots,K$. In this case, we define $\widetilde{\dgrid}(Q):=\dgrid^s(Q)$.
Also, if $\cals\subseteq \dgrid$ is any dyadic collection, then by $\ch_\cals(Q)$ we denote the \emph{$\cals$-children} of $Q\in \dgrid$, that is the maximal cubes in $\cals$ strictly contained inside $Q$.

\smallskip

Lastly, we need a certain form of ``annuli''. For each $Q\in {\dgrid}$, there is a $k\in\Z$ for which $Q\in\dgrid_k$. Set $\ell(Q):= 2C_0\delta^k$. \Cref{thm:dyadic_system}, then, tells us that
\[2\ell(Q') = 2\cdot 2C_0\delta^{k+1}\leq \frac{c_0}{3}\delta^k < \diam{Q} \leq 2\cdot 2C_0\delta^k = 2\ell(Q)\]
for any $Q'\in\dgrid_{k+1}$. Next, for any $\lambda> 1$, we define the \emph{$\lambda$-neighbourhood} of $Q$ by
\begin{equation} \label{eq:dilation}
	\lambda Q:= \{x\in M \such d(x,Q)\leq (\lambda-1)\ell(Q)\}= \{x\in M \such d(x,Q)\leq 2C_0(\lambda-1)\delta^k\}.
\end{equation}
The annuli $S_k(Q)$ ``centred'' around $Q$ are the sets
\begin{equation} \label{eq:annuli}
	S_0(Q) := 2Q \inline{and} S_k(Q) := 2^{k+1}Q \setminus 2^kQ, \inmath{for} k\in \Z_{\geq 1}.
\end{equation}

Later, during the proof of \Cref{thm:bilinear_convex_body_sparse_domination}, we will see that $\lambda$ has to be at least $1+\frac{2}{\delta}$. However, a more precise bound would be $\lambda \geq 2^{\lceil\log_2(1+\frac{2}{\delta})-1\rceil}$. Whenever $\delta=\frac{2}{2^m-1}$ for some $m\in\Z_{\geq2}$, we get that $\lambda \geq \frac{1}{2} + \frac{1}{\delta}$. This gives a clearer picture of the geometry and the constants that appear later on than the setup in \cite{BFP2016}, where they define $\ell(Q)=\delta^k$ for the integer $k$ which satisfies $\delta^{k+1}\leq\diam Q<\delta^k$. In turn, the factor $\lambda$ is more precise than simply ``$5$'' in Bernicot's et al. Theorem 5.7 and onwards. One can retrieve $5$ for $\lambda=1+\frac{2}{\delta}$ and $\delta=\frac{1}{2}$.

The following short lemma is not hard to prove. Nevertheless, we include a proof here for the sake of completeness.

\begin{lem}
\label{lemma:covering_dilation_of_dyadic_set_with_ball}
	Let $P\in\cald$ be a dyadic set.
	\begin{enumerate}[label=(\roman*)]
		\item There exists an open ball $B\subseteq M$ such that $P\subseteq B$ and $\mu(B)\leq c\mu(P)$, where the constant $c>0$ depends only on the doubling constant of $\mu$ and the parameters of $\cald$.
		
		\item Let $\lambda>1$. Then, there exists an open balls $B,B'\subseteq M$ such that $B'\subseteq P\subseteq \lambda P\subseteq B$ and
		\[
			\mu(B)\leq C_{\mu}\big(\frac{\lambda}{\lambda-1}\big)^{\nu}\mu(\lambda P)
		\inline{and}
			\mu(\lambda P)\leq C_\mu\left(\frac{6C_0}{c_0}\right)^\nu\lambda^\nu\mu(B').
		\]
		In particular, $\mu(\lambda P)\leq c_\lambda \mu(P)$ for some $c_\lambda > 0$. Also, if $\lambda\geq 2$, then $\mu(B)\leq C_{\mu}^2\mu(\lambda P)$.
	\end{enumerate}
\end{lem}

\begin{proof}
	By the properties of $\cald$, we have that there exist $x\in M$ and $k\in\Z$ such that
	\[
		B\Big(x,\frac{c_0}{3}\delta^k\Big)\subseteq P\subseteq B(x,2C_0\delta^k).
	\]
\begin{enumerate}[label=(\roman*)]
\item
	Set $B:=B(x,2C_0\delta^k)$. Then, $P\subseteq B$ and we can directly estimate
	\[
		\mu(P)\geq\mu\Big(B\Big(x,\frac{c_0}{3}\delta^k\Big)\Big)\geq\frac{1}{C_{\mu}}\left(\frac{c_0}{6C_0}\right)^{\nu}\mu(B).
	\]
	
\item
	It is clear that when $\lambda>1$
	\[
		B(x,\frac{c_0}{3}\delta^k)\subseteq P\subseteq B(x,2C_0(\lambda-1)\delta^k)\subseteq \lambda P\subseteq B(x,2C_0\delta^k+2C_0(\lambda-1)\delta^k).
	\]
	Set $B:=B(x,2C_0\lambda\delta^k)$. Then, we have $\lambda P\subseteq B$ and
	\[
		\mu(\lambda P) \geq \mu(B(x,2C_0(\lambda-1)\delta^k)) \geq \frac{1}{C_{\mu}}\left(\frac{\lambda-1}{\lambda}\right)^{\nu}\mu(B).
	\]
	For $\lambda \geq 2$, notice that $\frac{\lambda}{\lambda-1}\leq2$ and recall that $\nu =\log_2(C_\mu)$.
	
	Finally, set $B':=B(x,\frac{c_0}{3}\delta^k)$. Then, we have that $B'\subseteq P$ and
	\[
		\mu(\lambda P)\leq\mu(B(x,2C_0\lambda\delta^k))\leq C_\mu\left(\frac{2C_0\lambda\delta^k}{\frac{c_0}{3}\delta^k}\right)^\nu\mu(B(x,\frac{c_0}{3}\delta^k))=C_\mu\left(\frac{6C_0}{c_0}\right)^\nu\lambda^\nu\mu(B').
	\]
\end{enumerate}
\end{proof}

\addtocontents{toc}{\SkipTocEntry}
\subsection*{The maximal operators \texorpdfstring{$\cM$}{M}}
%\label{subsec:maximal_operator}
We denote by $\calm$ the \emph{uncentered Hardy--Littlewood maximal operator}

\begin{equation*}
\calm f(x):=\sup_{B\text{ ball}}\aver{|f|}_{B}\1_{B}(x),\quad x\in M
\end{equation*}
acting on functions of $L^1_\loc(M,\mu)$.
For an exponent $p > 1$, we denote by $\calm_p$ the $p$-maximal operator, $\calm_p(f):= [\calm(|f|^p)]^{\frac{1}{p}}$ for functions $f\in L^p_\loc(M,\mu)$. We recall the well-known bound
    \begin{equation}
    \label{Hardy_Litlewood_bound}
        \norm{\calm}_{L^{q}(M,\mu)\to L^{q}(M,\mu)}\leq (C_{\mu}3^{\nu})^{1/q}q',\quad 1<q<\infty.
    \end{equation}
See, for example, \cite[Exercises 1.3.3 and 2.1.1]{Grafakos_2014} for a sketch of a proof in $\R^n$ that works without any changes in any doubling measure metric space.

We also denote by $\cM^\dgrid$ the corresponding \emph{dyadic maximal operator} by
\[\calm^\dgrid f(x):=\sup_{Q\in\dgrid}\aver{|f|}_{Q}\1_{Q}(x),\quad x\in M\]
acting on $L^1_\loc(M,\mathrm{d}\mu)$, as well as $\cM_p^\dgrid(f):= \cM^\dgrid(|f|^p)^\frac{1}{p}$. It turns out that for $\mu$-almost all $y\in M$
\[\cM^\dgrid f(y) \lesssim \cM f(y) \inline{and} \cM^\dgrid_p f(y) \lesssim \cM_p f(y)\]
where the implicit constants depend only on the doubling measure $\mu$ and the exponent $p$. Naturally, all the properties of $\cM$ are inherited by its dyadic version $\cM^\dgrid$.

Next, for $Q\in \dgrid$ and for $1 \leq p < 2$ we define the operator $\cM_{Q,p}^*$ as follows: when $x\in Q$
\begin{equation} \label{eq:sharp_maximal_operator}
	\cM^*_{Q,p}f(x):= \sup_{\substack{P\in \dgrid(Q) \\ P\ni x}} \inf_{y\in Q} \cM_{p} f(y),
\end{equation}
and zero when $x\not\in Q$. It turns out that $\cM^*_{Q,p}f(x)$ is of weak type $(p,p)$ and of strong type $(q,q)$ for all $q\in (p,\infty]$ (see \cite[Section 5A]{BFP2016}).

\subsection{Sparse collections}

Here, we introduce the notion of ``sparseness''. It must be noted that one can find many different definitions of what a ``sparse collection'' is in the general bibliography. See, for example, \cite{NPTV2017} for an exposition of the alternative definitions in $\R^n$. All such definitions are, in fact, equivalent: one can look into \cite{Hae2018} to see the proof of such equivalences (see also \cite{Lerner_Nazarov_Dyadic_Calculus2019} for the special case of dyadic systems). Our setting is the following:

\begin{dfn}[Sparse Collection] \label{dfn:sparse_collection}
	Let $\cals$ be a collection of Borel subsets of $M$ of positive and finite $\mu$ measure, and consider a number $\varepsilon\in (0,1)$. The collection $\cals$ is called \emph{$\varepsilon$-sparse} when for all $Q\in \cals$ there exist pairwise disjoint measurable sets $F_Q\subseteq Q$ such that
	\[\mu(F_Q)\geq \varepsilon\mu(Q) \afterline{for all} Q\in \cals.\]
\end{dfn}

Later in the text, we work with sparse collections $\cals \subseteq \dgrid$ whose elements are dyadic sets.

%\begin{dfn}[Sparse Collection]
%	$\cals:= (P)_{P \in \cals} \subset \cald$ is said to be \textit{sparse} if for each $P \in \cals$ the inequality 
%	\begin{align}
%		\label{sparse_measure}
%		\sum_{Q \in \ch_{\cals}(P)} \mu(Q) \leq \frac{1}{2}\mu(P)
%	\end{align}
%	holds, where $\ch_{\cals}(P)$ is the collection of $\cals$-children of $P$, namely the maximal elements of $\cals$ that are strictly contained in $P$.
%\end{dfn}

\subsection{The operators} \label{sec:operator}
In this section, we introduce the main operators of our study along with some important properties.

%\begin{dfn}
%	For $\theta \in [0, \pi/2)$, a linear operator $L$ with dense domain $\cD_2(L)$ in $L^2(M, \mu)$ is called \emph{$\theta$-accretive} if the spectrum $\sigma(L)$ of $L$ is contained in the closed sector $S_{\theta+}:= \{\zeta \in \C: |\arg \zeta| \leq \theta\} \cup \{0 \}$, and $Lg \cdot g \in S_{\theta+}$ for all $g \in \cD_2(L)$.
%\end{dfn}

\addtocontents{toc}{\SkipTocEntry}
\subsection*{The unbounded operator \texorpdfstring{$L$}{L}} For $\theta \in [0, \pi/2)$, a linear operator $L$ with dense domain $\cD_2(L)$ in $L^2(M, \mu)$ is called \emph{$\theta$-accretive} if the spectrum $\sigma(L)$ of $L$ is contained in the closed sector $S_{\theta+}:= \{\zeta \in \C: |\arg \zeta| \leq \theta\} \cup \{0 \}$, and $Lg \cdot g \in S_{\theta+}$ for all $g \in \cD_2(L)$.

We assume the existence of an unbounded operator $L$ on $L^2(M, \mu)$ which satisfies the following properties: For any $\theta \in [0, \pi/2)$, $L$ is an injective $\theta$-accretive operator with dense domain $\cD_2(L)$ in $L^2(M, \mu)$. Also, suppose there exist two exponents $p_0$ and $q_0$ with $1 \leq p_0 < 2 < q_0 \leq \infty$ such that for all balls $B_1$ and $B_2$ of radius $\sqrt{t}$ ($t > 0$)
\begin{equation}\label{eq:maximal_accretive}
	\norm{e^{-tL}}_{L^{p_0}(B_1) \to L^{q_0}(B_2)} \lesssim |B_1|^{-1/p_0} |B_2|^{1/q_0} e^{-cd(B_1, B_2)^2/t},
\end{equation}
where $d(B_1, B_2)$ is the usual distance between $B_1$ and $B_2$; $d(B_1, B_2) := \inf\{d(x, y)\such {x \in B_1,\ y \in B_2}\}$.

From these assumptions, $L$ is what is called a \emph{maximal accretive operator} on $L^2(M, \mu)$, which implies that $L$ has a bounded $H^\infty$ functional calculus on $L^2(M, \mu)$. Note that if $L$ is unbounded, $e^{-tL}$ cannot be represented in the form of a power series. Since $\theta < \frac{\pi}{2}$, $-L$ is a generator of the analytic semigroup $(e^{-tL})_{t >0}$ in $L^2(M, \mu)$, see \cite{Albrecht_Duong_McIntosh_1996, Kato_1966}. Therefore, this semigroup satisfies the so-called \emph{$L^{p_0}$-$L^{q_0}$ off-diagonal estimates}, a generalization of the $L^2$-$L^2$ Davies--Gaffney estimates. This kind of off-diagonal estimate is useful in the case when pointwise heat kernel bounds fail.

\addtocontents{toc}{\SkipTocEntry}
\subsection*{The sublinear operator \texorpdfstring{$T$}{T}} Throughout this text, we consider an operator $T$ which satisfies the following properties:

\begin{enumerate}[label=(T\arabic*)]
	\item
	\label{def:sublinearity_definition}
	$T$ is a well-defined bounded sublinear operator on $L^2(M,\mu)$. Here, by ``sublinearity'' of $T$ we understand the following conditions:
	\begin{gather*}
		|T(f+g)(x)|\leq |Tf(x)|+|Tg(x)|,\\
		|Tf(x)-Tg(x)|\leq |T(f-g)(x)|,\\
	\intertext{and}
		|T(af)(x)|=|a|\cdot |Tf(x)|,
	\end{gather*}
	for $\mu$-a.e.~$x\in M$, for all $f,g\in L^{2}(M,\mu)$, and for all $a\in\C$.
    The second condition ensures that if $T$ is bounded on the Lebesgue space $L^{p}(M,\mu)$ as an operator $T:L^{p}(M,\mu)\to L^{p}(M,\mu)$ for some $1<p<\infty$, then $T$ is actually \emph{continuous} as a map $T:L^{p}(M,\mu)\to L^{p}(M,\mu)$.
    
	\item There exists $N_0 \in \N$ such that for all integers $N \geq N_0$ and all balls $B_1$ and $B_2$ of radius $\sqrt{t}$
	\[
		\norm{T(tL)^N e^{-tL}}_{L^{p_0}(B_1,\mu) \to L^{q_0}(B_2,\mu)} \lesssim |B_1|^{-1/p_0} |B_2|^{1/q_0} \left(1 + \frac{d(B_1, ,B_2)^2}{t} \right)^{- \frac{\nu + 1}{2}}.
	\]
	\item There exists an exponent $p_1 \in [p_0, 2)$ such that for all $x \in M$ and $r>0$
	\[
		\left( \dashint_{B(x,r)} |Te^{-r^2L}f|^{q_0} \dd\mu \right)^{\frac{1}{q_0}} \lesssim \underset{y \in B(x,r)}{\inf} \cM_{p_1}(Tf)(y) +  \underset{y \in B(x,r)}{\inf} \cM_{p_1}f(y).
	\]
\end{enumerate}

The assumptions (T2) and (T3) serve as a substitute for the notion of Calder{\'o}n--Zygmund operators, see for instance \cite{Auscher_2007} and \cite{Auscher_Pascal_Coulhon_2004}. Namely, condition (T2) provides the cancellation property for the operator $T$ which corresponds to the cancellation of the considered semigroup; this is known as \emph{$L^{p_0}$-$L^{q_0}$ off-diagonal estimates}. Condition (T3) is a Cotlar-type inequality, which allows one to get off-diagonal estimates for the low-frequency part of the operator $T$.

In the sequel, we fix an integer $N > \max\left\{\frac{3}{2} \nu + 1, N_0 \right\}$, where $\nu=\log_2(C_\mu)$, and $N_0$ comes from the assumptions on $T$. Unless specified otherwise, we don't emphasize the dependence on the implicit constant $N$.

\addtocontents{toc}{\SkipTocEntry}
\subsection*{The approximation operators \texorpdfstring{$P_t$}{P\_t} and \texorpdfstring{$Q_t$}{Q\_t}} Further, we introduce two classes of elementary operators;
\begin{itemize}[leftmargin=\parindent]
	\item the family $(P_t)_{t>0}$, which corresponds to an approximation of the identity at scale $\sqrt{t}$, and commutes with the heat semigroup, and
	\item the family $(Q_t)_{t>0}$, which satisfies some extra cancellation condition with respect to $L$.
\end{itemize}

\begin{dfn}
	Set $c_N:= \int_0^{\infty} s^N e^{-s} \dd s/s$. For $t>0$, we define
	\begin{gather*}
		Q_t^{(N)}:= c_N^{-1} (tL)^N e^{-tL}
	\intertext{and}	
		P_t^{(N)}:= \int_1^{\infty} Q_{st}^{(N)} \,\frac{\mathrm{d}s}{s} = \Phi_N(tL),
	\end{gather*}
	where $\Phi_N(x):= c_N^{-1} \int_x^{+\infty} s^{N} e^{-s} \frac{\mathrm{d}s}{s}$ for $x \geq 0$.
\end{dfn}

The operators above are derived from the semigroup $(e^{-tL})_{t>0}$ and substitute the Littlewood--Paley operators which are more typical for our circle of questions. We refer to \cite{Hofmann_Lu_Mitrea_Mitrea_Yan_2011, Auscher_Martell_II_2007} for a detailed discussion of their properties. We mention only some of them below.

%\begin{rem} \label{rem:operator_estimates}
%	Let $p \in [p_0, q_0]$ be a finite exponent and $N>0$.
	
%	\begin{enumerate}[label=(\roman*)]
%		\item For $N = 1$, $P_t^{(1)} = e^{-tL}$ and $Q_t^{(1)} = tLe^{-tL}$.
		
%		\item For $N \in \Z$, $Q_t^{(N)} = (-1)^N c_N^{-1} t^N \partial_t^N e^{-tL}$ and $P_t^{(N)} = \psi(tL)e^{-tL}$, where $\psi$ is a polynomial of degree $N-1$ with $\psi(0) = 1$. The elementary operators are related by the following identity:
%		\[t\partial_t P_t^{(N)} = tL {\Phi'_N} (tL) = - Q_t^{(N)}.\]
		
%		\item By the $L^p$-analyticity of the semigroup and \eqref{eq:maximal_accretive}, we know that for every $N \in \Z$ and every $t>0$ the operators $P_t^{(N)}$ and $Q_t^{(N)}$ satisfy off-diagonal estimates (as in assumption (T2) on $T$) at scale~$\sqrt{t}$. For the proof of this result, see, for example, \cite[Proposition 3.1]{Hofmann_Lu_Mitrea_Mitrea_Yan_2011}. 
		
%		\item In view of (iii), the operators $P_t^{(N)}$ and $Q_t^{(N)}$ are bounded in $L^p(M,\mu)$ uniformly for all $t>0$, see \cite[Theorem 2.3]{Auscher_Martell_II_2007}.
%	\end{enumerate}
%\end{rem}

\begin{rem} \label{rem:operator_estimates}
    Let $p \in (p_0, q_0)$.
    \begin{enumerate}[label=(\roman*)]
        \item We have $Q_t^{(N)} = (-1)^N c_N^{-1} t^N \partial_t^N e^{-tL}$ and $P_t^{(N)} = \psi(tL)e^{-tL}$, where $\psi$ is a polynomial of degree $N-1$ with $\psi(0) = 1$. Moreover, the following identity holds:
		\[t\partial_t P_t^{(N)} = tL {\Phi'_N} (tL) = - Q_t^{(N)}.\]
  
        \item By the $L^p$-analyticity of the semigroup and \eqref{eq:maximal_accretive}, we know that for every $t>0$ the operators $P_t^{(N)}$ and $Q_t^{(N)}$ satisfy off-diagonal estimates (as in assumption (T2) on $T$) at scale~$\sqrt{t}$. For the proof of this result, see, for example, \cite[Proposition 3.1]{Hofmann_Lu_Mitrea_Mitrea_Yan_2011}.

        \item In view of (ii), the operators $P_t^{(N)}$ and $Q_t^{(N)}$ are $L^p(M,\mu)$-bounded \emph{uniformly in $t>0$}, see \cite[Theorem 2.3]{Auscher_Martell_II_2007}.
    \end{enumerate}
\end{rem}

\begin{prop}[Calder{\'o}n Reproducing Formulae] \label{prop:calderon_reproducing_formula}
%Let $N>0$ and
	Let $p \in (p_0,q_0)$. For every $f \in L^p(M,\mu)$, the following limits hold in the $L^p(M,\mu)$ sense:
	\begin{gather*}
		\lim_{t \to 0+} P_t^{(N)}f = f;	\\
		\lim_{t \to +\infty} P_t^{(N)}f = 0;	\\
	\intertext{and}
		f = \int_0^{+\infty} Q_t^{(N)} f \, \frac{\mathrm{d}t}{t}.
	\end{gather*}
	In particular, one can write
	\begin{equation}
	    \label{tails_Calderon_formula}
            P_t^{(N)} = \id - \int_0^t Q_s^{(N)} \, \frac{\mathrm{d}s}{s},
	\end{equation}
	and as a consequence, we have
	\begin{equation} \label{eq:integration_over_Q_is_identity}
		\id = \int_0^\infty Q_t^{(N)} \, \frac{\mathrm{d}t}{t}.
	\end{equation}
\end{prop}

\smallskip

Since $T$ is sublinear, for any $f\in L^p(M,\mu)$ \cref{eq:integration_over_Q_is_identity} gives that
\begin{equation}
\label{eq:decomposition_linearity_operator_inequality}
	|T(f)| \leq \int_0^{\infty} |TQ_t^{(N)}(f)| \, \frac{\mathrm{d}t}{t}\quad\mu\text{-a.e. on }M. 
\end{equation}

%Let us fix $t>0$ and the elementary operator $TQ_t^{(N)}$. From Assumption (T2), $TQ_t^{(N)}$ enjoys $L^{p_0} - L^{q_0}$ off-diagonal estimates at the scale $\sqrt{t}$. Then, consider a ball $B$ of radius $r(B) > 0$ such that $r \leq \sqrt{t}$, and its dilated ball $\widetilde{B}:= \left(\frac{\sqrt{t}}{r}B\right)$. We observe that $B \subseteq \widetilde{B}$, and from \ref{prop:nudoubling}, we have $|\widetilde{B}| \lesssim \left(\frac{\sqrt{t}}{r} \right)^{\nu} |B|$, therefore
%\[
%	\left(\dashint_B |TQ_t^{(N)}f|^{q_0} \dd\mu \right)^{\frac{1}{q_0}}
%	\lesssim \left(\frac{\sqrt{t}}{r} \right)^{\frac{\nu}{q_0}} \underset{j \geq 0}{\sum} 2^{-j(\nu + 1)} \left(\dashint_{S_j(\widetilde{B})} |f|^{p_0} \dd\mu \right)^{\frac{1}{p_0}}.
%\]

\addtocontents{toc}{\SkipTocEntry}
\subsection*{The maximal operators associated to \texorpdfstring{$T$}{T}}

Next, we introduce a certain maximal operator related to $T$ which comes into play later.

\begin{dfn}\label{dfn:maximal_operator_sharp}
    For $x \in M$, the \emph{maximal operator $T^{\#}$ of $T$}, acting on functions $f \in L_{\loc}^{q_0}(M,\mu)$, is given by
    \[
        T^{\#} f(x) := \sup_{B \ni x} \left(\dashint_B \left|T \int_{r(B)^2}^\infty Q_t^{(N)} f \, \frac{\mathrm{d}t}{t}\right|^{q_0} \dd\mu \right)^{\frac{1}{q_0}},
    \]
    where the supremum is taken over all balls $B\subseteq M$ containing $x$.
\end{dfn}
Since $T$ is $L^2$-bounded, the Calder{\'o}n reproducing formulae (\Cref{prop:calderon_reproducing_formula}), imply the following properties for $T^{\#}$:
\begin{enumerate}[label=(\roman*)]
	\item For every function $f \in L^2(M,\mu)$, $|Tf| \leq T^{\#}f$ $\mu$-almost everywhere on $M$.
	\item The sublinear operator $T^{\#}$ is of weak type $(p_0,p_0)$ and is bounded in $L^p(M,\mu)$ for every $p \in (p_0, 2]$. 
\end{enumerate}
Details and proofs of these properties can be found in \cite[Section 4]{BFP2016}.

\medskip

One can also consider a localized dyadic version of $T^{\#}$. For this, start with a cube $Q\in\cald$. For $x\in Q$, we define the operator $T^\#_{Q}$ as follows:
	\begin{equation} \label{eq:sharp_T_operator}
		T^\#_{Q}f(x):= \sup_{\substack{P\in \dgrid(Q) \\ P\ni x}} \Biggl(\dashint_P \Bigg| T\int_{\ell(P)^2}^\infty \cQ_t^{(N)}(f) \,\frac{\mathrm{d}t}{t} \Bigg|^{q_0} \dd \mu \Biggr)^\frac{1}{q_0} \afterline{when}  x\in Q,
	\end{equation}
and zero when $x\not\in Q$. It turns out that for $\mu$-almost all $y\in M$ we have that $T^\#_{Q} f(y) \lesssim T^\# f(y)$, where the implicit constant depends only on the doubling measure $\mu$ and the exponent $q_0$.

Naturally, $|Tf| \leq T^\#_{Q}f$ $\mu$-almost everywhere for every $f\in L^2(M,\mu)$. All the remaining important properties of $T^\#_Q$ which we need, are inherited from $T^\#$. In particular, $T^\#_{Q}$ is of weak type $(p_0,p_0)$ and bounded in $L^p(M,\mu)$ for every $p\in (p_0,2]$ (see \cite[Section 4B]{BFP2016}).

\subsection{Matrix weights}

Let $W:M\to M_{d}(\C)$ be a measurable function with $W(x)\in\mathrm{P}_{d}(\C)$ for $\mu$-a.e.~$x\in M$, and $0<p<\infty$. For any measurable function $f:M\to\C^d$, we denote
\begin{align*}
    \norm{f}_{L^{p}_{W}(M,\mu)} = \left(\int_M |W(x)^{1/p}f(x)|^p\dd\mu(x) \right)^\frac{1}{p}.
\end{align*}
Naturally, we identify all functions for which $\norm{f}_{L^{p}_{W}(M,\mu)}=0$ into an equivalence class. Under this identification, this quantity defines a norm when $p\in[1,\infty)$, and therefore, for such $p$, we obtain the normed complex Banach space $(L^{p}_{W}(M,\mu),\norm{\cdot}_{L^{p}_{W}(M,\mu)})$ consisting of all ($\mu$-a.e.~equivalence classes of) measurable functions $f:M\to\C^d$ with $\norm{f}_{L^{p}_{W}(M,\mu)}<\infty$. We note that the dual $(L^{p}_{W}(M,\mu))^{*}$ of $L^{p}_{W}(M,\mu)$ can be naturally identified with $L^{p'}_{W'}(M,\mu)$ under the usual $L^2(M,\mu;\C^d)$-pairing, where $W':=W^{-1/(p-1)}$.

In general, ``matrix weighted'' spaces as above might behave rather pathologically. However, the class of functions $W$ considered in this paper will prevent such issues.

\begin{dfn}
Let $W:M\to M_{d}(\C)$ be a measurable function with $W(x)\in\mathrm{P}_{d}(\C)$ for $\mu$-a.e.~$x\in M$. We say that $W$ is a \emph{$d\times d$ matrix weight on $M$} if $\int_{E}|W(x)|\dd\mu(x)<\infty$ for any bounded Borel subset $E$ of $M$.
\end{dfn}

Let $W$ be a $d\times d$ matix weight on $M$. Then, it is not hard to see that the class of continuous compactly supported functions on $M$ is actually dense in $L^{p}_{W}(M,\mu)$, for any $1\leq p<\infty$. This observation allows us below to prove estimates only for functions in an appropriate ``nice class'' and then make use of ``standard'' approximation arguments to extend these estimates to more general functions.

\section{Preliminaries from convex-set-valued analysis}
\label{section:convex_set_analysis_preliminaries}

In this section, we review some notions and facts related to measurable convex-set-valued maps. This places the sparse domination of \Cref{section:bilinear_convex_body_sparse_domination}, as well as the limited range extrapolation of \Cref{section:extrapolation} on firm ground.

Our exposition follows closely \cite{Aubin_2009} and \cite{Cruz_Uribe_Bownik_Extrapolation}. We omit most proofs and instead refer the reader to these two references.

\subsection{Operations with subsets of Euclidean spaces}

As we deal with functions whose values are convex subsets of complex Euclidean spaces, we introduce some notation for the elementary operations between such sets.

Let $d$ and $n$ be positive integers, and let $A$, $B$, and $A_1,\ldots,A_n$ be non-empty subsets of $\C^d$. Also, consider arbitrary $\lambda\in\C$, $D\subseteq\C$, $x\in\C^d$ and $C\in\mathrm{M}_{d}(\C)$; we denote by $x^1,\ldots,x^{d}$ the coordinates of $x$ in $\C^d$. Then, we have the following definitions:
\begin{enumerate}[label=(\alph*)]
	\item The \emph{Minkowski sum} of $A_1,\ldots,A_n$:
	\begin{equation}
	\label{Minkowski sum}
		A_1+\ldots+A_n := \{a_1+\ldots+a_n \such a_i\in A_i \inmath{for all} i=1,\ldots,n\}\subseteq\C^d;
	\end{equation}
	
	\item the \emph{Minkowski (dot) product} of $A$ and $B$:
	\begin{equation}
	\label{Minkowski dot product}
	    A\cdot B := \{a\cdot b\such a\in A,~b\in B\}\subseteq\C,
	\end{equation}
	where for any $x,y\in\C^d$ the operation $x\cdot y$ stands for the usual Hermitian product of $x,y$, that is $x\cdot y:=\sum_{i=1}^{d}x^i\overline{y^i}$;
	
	\item \[ A\cdot x := A\cdot\{x\}\subseteq\C^d;\]
	
	\item \[\lambda A := \{\lambda a\such a\in A\}\subseteq\C^d;\]
	
	\item \[DA:=\{da \such d\in D,\ a\in A\}= \textstyle \bigcup_{d\in D}dA\subseteq\C^d;\]
	
	\item \[CA := \{Ca\such a\in A\}\subseteq\C^d;\]
	
	\item \[|A| := \sup\{|a|\such a\in A\}\in[0,\infty].\]
\end{enumerate}

\subsection{Convex subsets of Euclidean spaces}

Before talking about maps taking convex sets as values, we first review a few facts about convex sets themselves.

Let $d$ be a positive integer and consider the Euclidean space $\C^d$. A subset $K$ of $\C^d$ is said to be \emph{convex} if for all $x,y\in K$ and $t\in[0,1]$ it holds that $tx+(1-t)y\in K$. Any non-empty intersection of convex subsets is again convex. For any $A\subseteq\C^d$, the \emph{convex hull} $\mathrm{conv}(A)$ of $A$ is the smallest convex subset of $\C^d$ containing $A$. If $(K_n)^{\infty}_{n=1}$ is an increasing sequence of convex subsets of $\C^d$, then $\bigcup_{n=1}^{\infty}K_n$ is again a convex subset of $\C^d$.

The convex sets we are interested in exhibit some ``symmetry''. Namely, we say that a subset $K$ of $\C^d$ is \emph{complex-symmetric around $0$} (or sometimes in the current text simply \emph{symmetric}) if for any $x\in K$ and any $\lambda\in\C$ with $|\lambda|=1$ we have $\lambda x\in K$. We denote by $\cK_{\mathrm{cs}}(\C^d)$ the set of all non-empty, convex, closed, and complex-symmetric subsets of $\C^d$. Moreover, we denote by $\cK_{\bcs}(\C^d)$ the set of all non-empty, convex, closed, bounded and complex-symmetric subsets of $\C^d$. The elements of $\cK_{\bcs}(\C^d)$ are called \emph{convex bodies (around $0$)}. Observe that if $K$ is a non-empty, convex, closed, and complex-symmetric subset of $\C^d$, then $0\in K$, and $\lambda x\in K$ for any $x\in K$ and any $\lambda\in\overline{\D}$.

A particular type of convex bodies is so important that it deserves special terminology. We say that a subset $K$ of $\C^d$ is a \emph{complex ellipsoid around $0$}, (or sometimes in the current text simply \emph{ellipsoid}) if there exists an invertible matrix $A\in M_{d}(\C)$ with $K=A\overline{\mathbf{B}}$, where $\overline{\mathbf{B}}$ is the standard (Euclidean) closed unit ball in $\C^d$: $\overline{\mathbf{B}}:=\{x\in\C^d \such |x|\leq 1\}$. In this case, the polar decomposition of $A^{*}$ yields $K=C\overline{\mathbf{B}}$, where $C=(AA^{*})^{1/2}\in\mathrm{P}_{d}(\C)$.

In many theoretical and practical applications, one can ``replace'' a convex body with a non-empty interior with an ellipsoid. More precisely, if $K$ is a convex body in $\C^d$ with $0\in\mathrm{Int}(K)$, then there exists a unique ellipsoid $\cE_{K}$ in $\C^d$ satisfying the following two properties:
\begin{enumerate}
    \item $\cE_{K}\subseteq K$

    \item $\cE_{K}$ has maximum volume among all ellipsoids $\cE$ in $\C^d$ with $\cE\subseteq K$.
\end{enumerate}
This unique ellipsoid $\cE_{K}$ is called the \emph{John-ellipsoid of $K$}. It is well-known that
\begin{equation*}
    \cE_{K}\subseteq K\subseteq\sqrt{d}\,\cE_{K}.
\end{equation*}
For proofs and more details on John ellipsoids, we refer to \cite{Goldberg_2003, NPTV2017, Cruz_Uribe_Bownik_Extrapolation, DKPS2024}.

The following lemma follows immediately from the definitions, the fact that continuous maps preserve compactness, and the observation that if $C\subseteq\C$ is non-empty, compact, and $\lambda C\subseteq C$ for all $\lambda\in\overline{\D}$, then either $C=\lbrace0\rbrace$ or there exists $r>0$ with $C=\overline{D}(0,r)$.

\begin{lem}
    \label{lemma:Minkowski_product_of_convex_bodies_is_closed_disk}
    Let $K,L\in\cK_{\bcs}(\C^d)$. Then, $K+L\in\cK_{\bcs}(\C^d)$ and $|K+L|\leq|K|+|L|$. Moreover, either $K\cdot L=\lbrace0\rbrace$ or there exists $r>0$ with $r\leq|K||L|$ and such that $K\cdot L=\overline{D}(0,r)$.
\end{lem}

\subsection{Measurable convex-set-valued maps}
\label{section:convex_set_preliminaries}

We come to the central notion of this section, the measurability of convex-set-valued maps. Our definitions follow \cite[Chapter 8]{Aubin_2009}.

Recall that a \emph{Polish space} is a complete separable metric space. In the rest of this subsection, $X$ stands for a fixed Polish space, and a $(\Omega,\cA)$ for fixed measurable space. We denote by $\cK(X)$ the set of all closed subsets of $X$ and by $\cK^{*}(X)$ the set of all non-empty closed subsets of $X$.

\begin{dfn}
\label{dfn:measurable_close_set_valued_map}
   A map $F:\Omega\to\cK(X)$ is said to be \emph{weakly $\cA$-measurable} (or simply \emph{$\cA$-measurable} in the current text) if for any open subset $U$ of $X$ we have
    \[ F^{-1}(U) := \{\omega\in\Omega\such F(\omega)\cap U\neq\emptyset\}\in\cA.\]
\end{dfn}

The following lemma will allow us not to worry later about the measurability of maps arising, for example, through pointwise operations between measurable convex body valued maps (such as the Minkowski sum \eqref{Minkowski sum} or the Minkowski dot product) or through pointwise multiplication of measurable convex body valued maps with measurable matrix-valued functions.

\begin{lem}
    \label{lemma:measurability_preserved_through_continuous_operations}
    Let $Y, Z$ be Polish spaces, and let $F$ and $G$ be $\cA$-measurable maps $F,G:\Omega\to\cK^{*}(X)$, which take compact sets as values. If the map $\varphi:X \times Y\to Z$ is continuous, then the map $H:\Omega\to\cK^{*}(Z)$ given by
    \begin{equation*}
        H(\omega) := \varphi(F(\omega)\times G(\omega)) \afterline{for} \omega\in\Omega
    \end{equation*}
    is $\cA$-measurable.
\end{lem}

\begin{proof}
    First of all, observe that $H$ indeed takes closed subsets of $Z$ as values. Let now $W$ be any open subset of $Z$. Then we have
    \begin{equation*}
        H^{-1}(W) = \{\omega\in\Omega\such\varphi(F(\omega)\times G(\omega))\cap W\neq\emptyset\}
         = \{\omega\in\Omega\such (F(\omega)\times G(\omega))\cap \varphi^{-1}(W)\neq\emptyset\}.
    \end{equation*}
   Observe that $\varphi^{-1}(W)$ is an open subset of $X\times Y$, therefore, there exist sequences $(U_n)_{n=1}^{\infty}$ and $(V_n)^{\infty}_{n=1}$ of open subsets of $X$ and $Y$ respectively such that
   \begin{equation*}
       \varphi^{-1}(W)=\bigcup_{n=1}^{\infty}(U_n\times V_n).
   \end{equation*}
   Thus, we have
   \begin{equation*}
       H^{-1}(W) = \bigcup_{n=1}^{\infty}(F^{-1}(U_n)\cap G^{-1}(V_n))\in\cA
   \end{equation*}
   concluding the proof.
\end{proof}

The following useful alternative characterization of measurability for closed-set-valued maps follows immediately from \Cref{dfn:measurable_close_set_valued_map} coupled with the separability of the metric space $X$.

\begin{lem}
\label{lemma:alternative_characterization_measurability}

Let $F:\Omega\to\cK^{*}(X)$ be a map. The following statements are equivalent:

\begin{enumerate}[label=(\roman*)]
   \item The map $F:\Omega\to\cK^{*}(X)$ is $\cA$-measurable.

    \item For any $x\in X$, the function
    \[\Omega\ni\omega\mapsto\dist(x,F(\omega))\in[0,\infty)\]
    is $\cA$-measurable.

   \end{enumerate}

\end{lem}

Consequently, we do not have to worry about measurability issues when ``comparing'' the values of measurable closed-set-valued maps pointwise.

\begin{cor}
    \label{corollary:comparing_closed_set_valued_maps}
    Let $F,G:\Omega\to\cK^{*}(X)$ be $\cA$-measurable maps. Then, we have
    \begin{equation*}
        \{\omega\in\Omega\such F(\omega)\subseteq G(\omega)\}\in\cA.
    \end{equation*}
\end{cor}

\begin{proof}
	Let $(x_n)^{\infty}_{n=1}$ be a sequence of points of $X$ with
	\begin{equation*}
		X=\mathrm{clos}(\{x_n\such n=1,2,\ldots\}).
	\end{equation*}
	We claim that
	\begin{equation}
	\label{basic_inclusion}
		\{\omega\in\Omega\such F(\omega)\subseteq G(\omega)\}
		=\bigcap_{n=1}^{\infty}\{\omega\in\Omega\such \dist(x_n,F(\omega))\geq\dist(x_n,G(\omega))\}.
	\end{equation}
	Given \Cref{lemma:alternative_characterization_measurability}, this is enough to conclude the required result.
	
	The inclusion $\subseteq$ in \eqref{basic_inclusion} is clear. For the reverse inclusion $\supseteq$, let $\omega\in\Omega$ such that
	\[\dist(x_n,F(\omega))\geq\dist(x_n,G(\omega)) \afterline{for all} n=1,2,\ldots\]
	Let $x\in F(\omega)$ be arbitrary. Then, one can find a sequence $(y_m)^{\infty}_{m=1}$ of points in $\{x_n\such n=1,2,\ldots\}$ with $y_m\to x$ in $X$ as $m\to\infty$. In other words, $\lim_{m\to\infty}d(y_m,x)=0$, where $d$ is the metric of $X$. Since $x\in F(\omega)$, then $\lim_{m\to\infty}\dist(y_m,F(\omega))=0$. By assumption $\dist(y_m,G(\omega))\leq\dist(y_m,F(\omega))$, for all~$m=1,2,\ldots$ and therefore $\lim_{m\to\infty}\dist(y_m,G(\omega))=0$. It immediately follows that $\dist(x,G(\omega))=0$. Since $G(\omega)$ is a closed subset of $X$, $x\in G(\omega)$ concluding the proof.
\end{proof}

To develop a measurability theory for set-valued functions, we need to be able to study them through usual ``single-valued'' functions that \emph{represent} them.

\begin{dfn}
\label{dfn:selection_map}
	Let $F:\Omega\to\cK^{*}(X)$ be a map. A \emph{selection map for $F$} is a map $f:\Omega\to X$ such that $f(\omega)\in F(\omega)$ for any $\omega\in\Omega$.
\end{dfn}

A far-reaching characterization of the measurability of closed-set-valued maps in terms of selection maps was established by C.~Castaing and it is also known as the \emph{Castaing representation theorem}.

\begin{thm}[\protect{\cite[Theorem 8.3.1]{Aubin_2009}}]
\label{thm:Castaing_representation}

Let $F:\Omega\to\cK^{*}(X)$ be a map. The following statements are equivalent:

\begin{enumerate}[label=(\roman*)]
    \item The map $F:\Omega\to\cK^{*}(X)$ is $\cA$-measurable.

    \item There exists a sequence $f_n:\Omega\to X$ ($n=1,2,\ldots$) of $\cA$-measurable selection maps for $F$ such that
    \begin{equation*}
        F(\omega) = \mathrm{clos}(\{f_n(\omega)\such n=1,2,\ldots\}),\quad\forall\omega\in\Omega.
    \end{equation*}
\end{enumerate}
\end{thm}

Let now $d$ be any positive integer. We are particularly interested in $\cA$-measurable maps $F:\Omega\mapsto\cK_{\mathrm{cs}}(\C^d)$. For technical reasons, we need the following lemma:
\begin{lem}
    \label{lemma:a_e_boundedness_and_measurability_inner_product}
    Let $F:\Omega\to\cK_{\mathrm{cs}}(\C^d)$ be an $\cA$-measurable map.
    \begin{enumerate}
        \item The function $|F|:\Omega\to[0,\infty]$, where $|F|(\omega):=|F(\omega)|$ for $\omega\in\Omega$, is $\cA$-measurable.

        \item Let $\mu$ be a measure on $(\Omega,\cA)$ such that $F(\omega)$ is bounded for $\mu$-a.e.~$\omega\in\Omega$. Then there exists an $\cA$-measurable map $\tilde{F}:\Omega\to\cK_{\bcs}(\C^d)$ with $\tilde{F}(\omega) = F(\omega)$ for $\mu$-a.e.~$\omega\in\Omega$.
    \end{enumerate}
\end{lem}

\begin{proof}
	\begin{enumerate}
		\item It suffices to observe that
		\begin{equation*}
			\{\omega\in\Omega\such |F(\omega)|>r\}=F^{-1}(\{v\in\C^d\such |v|>r\})\in\cA \quad\forall r\in(0,\infty).
		\end{equation*}
		
		\item Set
		\begin{equation*}
			A := \{\omega\in\Omega\such F(\omega)\text{ is unbounded}\}.
		\end{equation*}
		By part (1), we have that $A\in\cA$. Thus, by assumption, it holds that $\mu(A)=0$. Consider the map
		\[\tilde{F}:=\1_{\Omega\setminus A}F:\Omega\to\cK_{\bcs}(\C^d).\]
		It is clear that $\tilde{F}$ takes values in $\cK_{\bcs}(\C^d)$, is $\cA$-measurable, and $\tilde{F}(\omega)=F(\omega)$ for $\mu$-a.e.~$\omega\in\Omega$.
	\end{enumerate}
\end{proof}

\subsection{Lebesgue spaces of convex-set-valued maps}

In this subsection, we briefly consider Lebesgue spaces of convex body valued functions, which furnish natural domains and/or target spaces for the operators considered in this paper.

First, fix a positive integer $d$ as well as a $\sigma$-finite measure space $(\Omega,\mu)$.

For a $\mu$-measurable map $F:\Omega\to\cK_{\mathrm{cs}}(\C^d)$, we define
\begin{equation*}
    \norm{F}_{L^{p}_{\cK}(\Omega,\mu)}:=\left(\int_{\Omega}|F(\omega)|^{p}\dd\mu(\omega)\right)^{1/p} \afterline{for all} 0<p<\infty.
\end{equation*}
Observe that if $\norm{F}_{L^{p}_{\cK}(\Omega,\mu)}<\infty$ for some $0<p<\infty$, then $F(\omega)$ is bounded for $\mu$-a.e.~$\omega\in\Omega$. Naturally, we identify all $\mu$-measurable maps $F,G:\Omega\to\cK_{\mathrm{cs}}(\C^d)$ for which $F(\omega)=G(\omega)$ for $\mu$-a.e.~$\omega\in\Omega$, and  for each $p\in(0,\infty)$ we define $L^{p}_{\cK}(\Omega,\mu)$ to be the set of all (equivalence classes of) $\mu$-measurable maps $F:\Omega\to\cK_{\mathrm{cs}}(\C^d)$ which satisfy $\norm{F}_{L^{p}_{\cK}(\Omega,\mu)}<\infty$. As usual, our notation does not distinguish between equivalence classes and representatives.

In the case when $p>1$, $\norm{\var}_{L^p_\cK(\Omega,\mu)}$ satisfies all the necessary properties (triangle inequality, absolute homogeneity, and positive definiteness) of a norm. Nevertheless, this does \emph{not} define a norm as the space $L^p_\cK(\Omega,\mu)$ is not a vector space; there is no additive inverse when the addition is defined through the Minkowski sum. However, this does not create a problem for our work here.

%%% Remarks

\begin{comment}

\note{Here, we cannot claim that $\norm{\var}_{L^p_\cK(\Omega,\mu)}$ is actually a norm when $p>1$, because $L^p_\cK(\Omega,\mu)$ is not a vector space (it doesn't have an additive inverse), even though it does satisfy the norm properties. Right?}

\bnote{That is right! The space $L_{\cK}^{p}(\Omega,\mu)$ is not a vector space, precisely for the reason you mention. The ``norm'' is still homogeneous, satisfies the triangle inequality and 0 if and only if the map is a.e.~equal to $\lbrace0\rbrace$, but it is not a norm in the strict sense.}

\note{Is there a chance that the non-existence of the additive inverse creates any problems for the arguments later on? For example, can one still implement a ``standard duality argument'' to get the boundedness for $\mathbf{T}$ for the $L^2$-boundedness assumption and the ``bilinear convex body sparse domination'' that we prove? In other words, is there any place where one \emph{requires all the properties of a vector space}?}

\bnote{No, the absense of additive inverses creates no problem.}

\note{added relevant text (probably you already saw it)}

\bnote{Yeah, and I think it is great}

\end{comment}

Note that in view of \Cref{lemma:a_e_boundedness_and_measurability_inner_product} whenever $F\in L^{p}_{\cK}(\Omega,\mu)$ we can assume without loss of generality that $F(\omega)$ is bounded ---i.e. $F(\omega)\in\cK_\bcs(\C^d)$--- for all $\omega\in\Omega$.

It is important to mention that a consequence of \cite[Lemma 3.9]{Cruz_Uribe_Bownik_Extrapolation} is that the computation of a Lebesgue ``norm'' of a convex body valued function reduces to the computation of the same Lebesgue ``norm'' of an appropriately chosen selection function. For future reference, we record a slightly more general form of \cite[Lemma 3.9]{Cruz_Uribe_Bownik_Extrapolation}:

\begin{lem}
\label{lemma:compute_Lebesgue_norm_through_selection_function}
	Let $W:\Omega\to\mathrm{P}_{d}(\C)$ be a $\mu$-measurable function, and $F:\Omega\to\cK_{\bcs}(\C^d)$ a $\mu$-measurable map. Then, there exists a $\mu$-measurable selection function $h:\Omega\to\C^d$ for $F$ such that $|W(\omega)h(\omega)|=|W(\omega)F(\omega)|$ for $\mu$-a.e.~$\omega\in\Omega$.
	
	In particular, there exists a $\mu$-measurable selection function $h:\Omega\to\C^d$ for $F$ such that $|h(\omega)|=|F(\omega)|$ for $\mu$-a.e.~$\omega\in\Omega$.
\end{lem}

\begin{proof}
	Let $\bar{\mu}$ be the completion of $\mu$. By \cite[Lemma 3.9]{Cruz_Uribe_Bownik_Extrapolation} coupled with \Cref{lemma:measurability_preserved_through_continuous_operations}, there exists a $\bar{\mu}$-measurable selection function $g:\Omega\to\C^d$ for the $\mu$-measurable map $W F:\Omega\to\cK_{\bcs}(\C^d)$ with $|W(\omega)F(\omega)|=|g(\omega)|$ for all $\omega\in\Omega$. Then, there exists a $\mu$-measurable function $\tilde{g}:\Omega\to\C^d$ with $\tilde{g}(\omega)=g(\omega)$ for $\mu$-a.e.~$\omega\in\Omega$. Let $A$ be a $\mu$-measurable subset of $\Omega$ such that $\mu(\Omega\setminus A)=0$ and $\tilde{g}(\omega)=g(\omega)$ for all $\omega\in A$. Consider the function $h:=\1_{A}W^{-1}\tilde{g}:\Omega\to\C^d$.
	It is clear that $h:\Omega\to\C^d$ is a $\mu$-measurable selection function for $F$ with $|W(\omega)h(\omega)|=|W(\omega)F(\omega)|$ for all $\omega\in A$ thus concluding the proof.
\end{proof}

\medskip

Since we are interested in weighted spaces, we also need to establish the weighted versions of the above definitions.
Let $F:\Omega\to\cK_\cs(\C^d)$ and $W:\Omega\to\mathrm{M}_d(\C)$ be $\mu$-measurable maps.
Then, for any $0<p<\infty$, we define the quantity
\[\norm{F}_{L^{p}_{\cK,W}(\Omega,\mu)} := \left( \int_{\Omega}|W(\omega)^{1/p}F(\omega)|^{1/p}\dd\mu(\omega) \right)^{1/p},\]
and the space $L^{p}_{\cK,W}(\Omega,\mu)$ of all (equivalence classes of) $\mu$-measurable maps $F:\Omega\to\cK_\cs(\C^d)$ such that $\norm{F}_{L^p_{\cK,W}(\Omega,\mu)}<\infty$.
The equivalence classes are defined as above.

The quantity $\norm{\var}_{L^{p}_{\cK,W}(\Omega,\mu)}$ again satisfies all the properties of a norm when $p>1$, but $L^{p}_{\cK,W}(\Omega,\mu)$ is not a vector space.
Also, just like above, whenever $F\in L^p_{\cK,W}(\Omega,\mu)$ we can assume without loss of generality that $F(\omega)\in\cK_\bcs(\C^d)$ for every $\omega\in\Omega$.

%%% Remarks

\begin{comment}

\note{check the definition here}

\bnote{The definition is correct.}

\note{Can we say like above that if $F\in L^p_{\cK,W}(\Omega,\mu)$ and $W$ is a matrix weight, then $F(\omega)\in\cK_\bcs(\C^d)$ for $\mu$-almost every $\omega\in\Omega$?
Of course, it is true that $W(\omega)^{1/p}F(\omega)\in\cK_\bcs(\C^d)$ for $\mu$-almost every $\omega\in\Omega$.}

\bnote{The fact $W(\omega)^{1/p}F(\omega)\in\cK_\bcs(\C^d)$ for $\mu$-almost every $\omega\in\Omega$ implies $F(\omega)\in\cK_\bcs(\C^d)$ for $\mu$-almost every $\omega\in\Omega$ trivially, because multiplication with matrices is continuous and complex-linear. The reason for saying $F(\omega)\in\cK_\cs(\C^d)$ instead of just $F(\omega)\in\cK_\bcs(\C^d)$ is because in the extrapolation later a weighted boundedness is stated for the convex body maximal function, which is not known a priori to be taking values in $\cK_{\mathrm{bcs}}$, only in $\cK_{\mathrm{cs}}$.}

\note{added relevant text (probably you already saw it)}

\bnote{Yes, I think it improves the clarity of the text}

\end{comment}

%%%%%%%%%%%%%%%%%%%%%%%%%%%%			section				%%%%%%%%%%%%%%%%%%%%%%%%%%%%%%%%%%%%%%
\section{Bilinear convex body sparse domination}
\label{section:bilinear_convex_body_sparse_domination}

In their paper \cite[Theorem 5.7]{BFP2016}, Bernicot, Frey, and Petermichl prove a sparse dual estimate for sublinear operators $T$, which satisfy the assumptions of \Cref{section:preliminaries} and act on functions from $L^p(M,\mu)$. Here, refining a recent result by Hyt{\"o}nen \cite{Hyt2024}, we extend their result to sublinear operators $\mathbf{T}$ which now act upon vector-valued functions $f\in L^p(M,\mu;\C^d)$ and yield as output convex body valued functions $\mathbf{T}f\in L^{p}_{\cK}(M,\mu)$. More specifically, we prove what one can call a \emph{bilinear convex body sparse domination} for sublinear operators.

The number $d$ stands for a positive integer. %%% Remarks
%\bnote{(Someone writes this at last! Mathematics is a little happier!)} \note{$\heartsuit$}
Also, $\overline{\D}$ denotes the closed unit disk in $\C$, and $\overline{\mathbf{B}}$ the (Euclidean) closed unit ball of $\C^d$.

%%% Remarks

\begin{comment}

\bnote{(We didn't use anywhere the open unit disk $\D$, only the closed one $\overline{\D}$, hence the change.)}
\note{I think we could change $\overline{\mathbf{B}}$ into something like $\overline{\D}_d$? Only because it is more suggestive of what it means.}

\bnote{I am not sure, because $\overline{D}_{d}$ stands in some contexts for the polydisk, which is not the same to the ball. (By the way, not even conformally equivalent to the ball, which shows one of the greate differences between complex analysis in one and many variables.) In the the third version we can think about the notation again.}

\end{comment}

%% \subsection{Bilinear operators} %%% Remarks The operators are not bilinear.

\subsection{Convex body sublinear operators}
\label{subsec:assumptions_operators}

We start by introducing the class of operators, whose sparse bounds are of interest to us. We need two fundamental definitions.

\begin{dfn}
\label{dfn:sublinear_convex_set}
	Let $1<p<\infty$. A map $\mathbf{T}:L^{p}(M,\mu;\C^d)\to L^{p}_{\cK}(M,\mu )$ is said to be \emph{sublinear} if for all $f,g\in L^{p}(M,\mu;\C^d)$ and for all $a\in\C$ it holds that
	\[\mathbf{T}(f+g)(x)\subseteq\mathbf{T}f(x)+\mathbf{T}g(x)\]
	and
	\[\mathbf{T}(af)(x)=a\mathbf{T}f(x)=|a|\mathbf{T}f(x)\]
	for $\mu$-a.e. $x\in M$.   
\end{dfn}

\begin{dfn}
\label{dfn:compatibility_scalar_convex_body}
	Let $1<p<\infty$. Consider a map $\mathbf{T}:L^{p}(M,\mu;\C^d)\to L^{p}_{\cK}(M,\mu)$ and a map $T:L^{p}(M,\mu)\to L^{p}(M,\mu)$. We say that $\mathbf{T}$ is \emph{compatible} with $T$ if for any $f\in L^{p}(M,\mu)$ and for any $v\in\C^d$ we have
	\begin{equation}
	\label{action_on_lines}
		\mathbf{T}(fv)(x)=Tf(x)\overline{\D}v=\{\lambda Tf(x)v\such\lambda\in\overline{\D}\}
	\end{equation}
	for $\mu$-a.e.~$x\in M$. 
\end{dfn}

%\note{Can we say that $|\mathbf{T}f(x)|=|Tf(x)|$ (for all of $\mu$-almost all  $x\in M$) whenever $\mathbf{T}$ is compatible with $T$?}

In the current text, we are interested in sublinear maps $\mathbf{T}:L^{p}(M,\mu;\C^d)\to L^{p}_{\cK}(M,\mu )$ that are compatible with those operators $T:L^{p}(M,\mu)\to L^{p}(M,\mu)$ which satisfy the assumptions of \Cref{section:preliminaries}. 

\subsubsection{Estimating the operators}
\label{subsec:estimate_operator}

Let $1<p<\infty$ and let $\mathbf{T}:L^{p}(M,\mu;\C^d)\to L^{p}_{\cK}(M,\mu )$ be a map. 
Let $W$ be a $d\times d$ matrix weight on $M$, such that the function $W^{-1/(p-1)}$ is also locally integrable. Suppose we want to show that there exists a finite positive constant $C$, such that
\begin{equation*}
    \norm{\mathbf{T}f}_{L^{p}_{\cK,W}(M,\mu)}\leq C\norm{f}_{L^{p}_{W}(M,\mu)}
\end{equation*}
for any compactly supported function $f\in L^{p}(M,\mu;\C^d)$. By \Cref{lemma:compute_Lebesgue_norm_through_selection_function}, it is enough to prove that for any $\mu$-measurable selection function $h:M\to\C^d$ for~$\mathbf{T}f$ we have
\begin{equation*}
    \norm{h}_{L^{p}_{W}(M,\mu)}\leq C\norm{f}_{L^{p}_{W}(M,\mu)}.
\end{equation*}
Thus, in view of the Riesz representation theorem and by a standard density argument we only have to show that for any \emph{compactly supported} function $g\in L^{p'}(M,\mu;\C^d)$ one has
\begin{equation*}
    \left|\int_{M}h(x)\cdot g(x)\dd\mu(x)\right|\leq C\norm{f}_{L^{p}_{W}(M,\mu)}\norm{g}_{L^{p'}_{W'}(M,\mu)},
\end{equation*}
where $W':=W^{-1/(p-1)}$.

Take $\lambda\geq1+\frac{2}{\delta}$, where $\delta$ is the parameter of $\dgrid$ as in Subsection \ref{subsec:dyadic}. (The choice of this particular constant $\lambda$ is explained in the proof of \Cref{thm:bilinear_convex_body_sparse_domination} below.)
%%% Claim about dyadic sets used also in the scalar case. It can be formulated about general lambda and general bounded sets, in our case the bounded set is the union of the two supports and the lambda is the usual one.
Looking back at the construction of the dyadic sets in Subsection \ref{subsec:dyadic} (also see \cite{HytKai2012}), thanks to the doubling properties of $\mu$, we see that there are cubes $Q_1,\ldots,Q_{K}\in\dgrid$, where $K$ is a large enough positive integer depending only on $\dgrid$ and $\mu$, such that all dilates $\lambda Q_i$, $i=1,\ldots,K$ contain the support of $f$ and $g$, and in addition the union $\bigcup_{i=1}^{K}Q_i$ also contains the support of $f$ and $g$ (up to a set of zero $\mu$-measure).
%%% End of claim. It should be easy to write a formal proof for it in section 2. In fact, the cubes Q_1, ... Q_K could even be taken to be cubes of the same generation, so in particular disjoint.
Therefore, we see that it suffices to show that whenever $Q_0$ is a dyadic set in $\dgrid$ such that $\lambda Q_0$ contains the support of both $f$ and $g$, then we have
\begin{equation*}
    \int_{Q_0}|h(x)\cdot g(x)|\dd\mu(x)\leq C\norm{f}_{L^{p}_{W}(M,\mu)}\norm{g}_{L^{p'}_{W'}(M,\mu)}.
\end{equation*}
Observe that it is always true that
\begin{equation*}
	\int_{Q_0}|h(x)\cdot g(x)|\dd\mu(x)\leq\int_{Q_0}|\mathbf{T}f(x)\cdot g(x)|\dd\mu (x).
\end{equation*}
Thus, all we need to show is that
\begin{equation*}
    \int_{Q_0}|\mathbf{T}f(x)\cdot g(x)|\dd\mu (x)\leq C\norm{f}_{L^{p}_{W}(M,\mu)}\norm{g}_{L^{p'}_{W'}(M,\mu)}.
\end{equation*}

\subsubsection{Extensions of scalar linear operators}

Let us observe that the above setup already covers the case of the usual extensions of \emph{linear} operators on vector-valued functions. Indeed, if $T:L^{p}(M,\mu)\to L^{p}(M,\mu)$ is a linear operator, then $\vec{T}:L^{p}(M,\mu;\C^d)\to L^{p}(M,\mu;\C^d)$ is defined as
\begin{equation}
\label{eq:standard_linear_extension}
	\vec{T}f:=(Tf_1,\ldots,Tf_d) \afterline{for all} f=(f_1,\ldots,f_d)\in L^{p}(M,\mu;\C^d).
\end{equation}
Now, consider the operator $\mathbf{T}:L^{p}(M,\mu;\C^d)\to L^{p}_{\cK}(M,\mu )$ which is given by
\begin{equation*}
    \mathbf{T}f(x):=\overline{\D}\,\vec{T}f(x)\quad\text{ for }x\in M \text{ and }f\in L^{p}(M,\mu;\C^d).
\end{equation*}
It is easy to see that $\mathbf{T}$ is sublinear and compatible with $T$.

%%% Remarks

\begin{comment}

{\color{red}
Let's check. First, we add definition for what it means to have $DA$ when $D\subseteq\C$ and $A\subseteq\C^d$. (I added above)

Take $g\in L^p(M,\mu)$ and $\vec{v}=(v^1,\dots,v^d)\in\C^d$. Then, (omitting the $x$'s)
\[\mathbf{T}(g\vec{v})=\overline{\D}\,\vec{T}(gv^1\dots,gv^d)=\{\lambda\big(T(gv^1),\dots,T(gv^d)\big)\such\lambda\in\overline{\D}\}\]
On the other hand,
\[T(g)\overline{\D}\,\vec{v}=\{\lambda \big(T(g)v^1,\dots,T(g)v^d\big) \such \lambda\in\overline{\D}\}.\]
These are \textbf{not} the same when $T$ is sublinear, because sublinear means
\[|T(ag)|=|a||T(g)| \inline{and not} T(ag)=aT(g).\]
Unless they are the same regardless?
}

\bnote{No, they are not the same regardless when $T$ is sublinear, you are right. But only in this last bit observe that we have assumed $T$ is actually linear.}

\end{comment}

Moreover,
\begin{equation*}
    \norm{\mathbf{T}f}_{L^{p}_{\cK,W}(M,\mu)}=\norm{\vec{T}f}_{L^{p}_{W}(M,\mu)}
\end{equation*}
for all $f\in L^{p}(M,\mu;\C^d)$ and for any $d\times d$ matrix weight $W$ on $M$. Thus, if there exists a finite positive constant $C$ such that
\begin{equation*}
    \norm{\mathbf{T}f}_{L^{p}_{\cK,W}(M,\mu)}\leq C\norm{f}_{L^{p}_{W}(M,\mu)},
\end{equation*}
for every compactly supported function $f\in L^{p}_{W}(M,\mu)$, then we deduce
\begin{equation*}
    \norm{\vec{T}f}_{L^{p}_{W}(M,\mu)}\leq C\norm{f}_{L^{p}_{W}(M,\mu)},
\end{equation*}
for every compactly supported function $f\in L^{p}_{W}(M,\mu)$. If in addition the original linear operator $T:L^{p}(M,\mu)\to L^{p}(M,\mu)$ is bounded, then since $\vec{T}$ is itself linear, a standard density argument yields that we can find a bounded linear extension $\widetilde{\vec{T}}:L^{p}_{W}(M,\mu)\to L^{p}_{W}(M,\mu)$ of $\vec{T}$ (that is, agreeing with $\vec{T}$ on $L^{p}(M,\mu;\C^d)\cap L^{p}_{W}(M,\mu)$) with $\norm{\widetilde{\vec{T}}}_{L^{p}_{W}(M,\mu)}\leq C$.

Finally, note that \cref{eq:standard_linear_extension} is the standard way to extend a linear operator $T$ acting on scalar-valued functions to a linear operator $\vec{T}$ acting on vector-valued functions. Typically, ---as is the case for the current text as well--- the arrow in $\vec{T}$ is omitted and ``$T$'' is used to refer to both $T:L^{p}(M,\mu)\to L^{p}(M,\mu)$ or $\vec{T}:L^{p}(M,\mu;\C^d)\to L^{p}(M,\mu;\C^d)$ depending on the function space the operator is acting on.

\subsection{Bootstrapping scalar sparse dual bounds}

Bilinear convex body bounds for the operators of interest are achieved through an inductive process. In this subsection, we treat the inductive step of this procedure. We assume that the ``base'' operator which acts on scalar-valued functions admits a sparse domination principle that relies on an inductive step of an appropriate form. Then, we adapt \cite[Proposition 4.6]{Hyt2024}.

%%	\note{}

\subsubsection{Notation}

We begin by introducing some notation and conventions. Let $X, Y$ be vector spaces over $\C$.

If $x$ is an element in $X^{d}$, then we denote by $x_1,\ldots,x_d$ its components (i.e. $x_i\in X$ for all $i=1,\dots,d$). Moreover, if $v$ is a vector in $\C^d$, then we denote by $v^1,\ldots,v^{d}$ its coordinates. Also for $v,w\in\C^d$, the operation $v\cdot w$ stands for the usual Hermitian product of $v$ and $w$ in $\C^d$:
\[v\cdot w=\sum_{i=1}^d v^i\overline{w^i}.\]

If $T:X\to Y$ is a linear map, then ---as explained earlier--- we extend $T$ to a linear map $\vec{T}:X^{d}\to Y^{d}$ by setting
\begin{equation*}
    \vec{T}(x):=(T(x_1),\ldots,T(x_d)) \afterline{for} x\in X^{d}.
\end{equation*}
Abusing the notation, we denote $\vec{T}$ simply by $T$.

For $x\in X^{d}$ and $u\in\C^d$, we follow \cite{Hyt2024} and set
\begin{equation*}
    v\bullet x:=\sum_{i=1}^{d}v^ix_i.
\end{equation*}
In the case when $X=\C$,
\begin{equation} \label{eq:bullet_to_Hermitian_product}
    v\bullet x=\sum_{i=1}^{d}v^{i}x^{i}=v\cdot\bar{x}\qquad\forall v,x\in\C^d,
\end{equation}
where $\bar{x}=(\overline{x^1},\ldots,\overline{x^d})$.

For $A\in M_{d}(\C)$ and $x\in X^d$ we set
\begin{equation*}
    A\bullet x:=\left(\sum_{j=1}^{d}A_{1j}x_{j},\ldots,\sum_{j=1}^{d}A_{dj}x_{j}\right)\in X^{d}.
\end{equation*}
Notice that if $X=\C$, then we simply have $A\bullet x=Ax$ for all $x\in\C^d$ and $A\in M_{d}(\C)$.

\subsubsection{Convex body ``averages''}

Here, we describe the ``building blocks'' of our sparse domination principle.

Let $X$ be a normed space over $\C$, and $X^{*}$ its dual space. For all $x\in X^{d}$, we denote
\begin{equation*}
    \convex{x}_{X} := \{ x^{*}(x):~x^{*}\in X^{*},\ \norm{x^{*}}_{X^*} \leq 1 \}.
\end{equation*}
It is clear that $\convex{x}_{X}$ is a complex-symmetric subset of $\C^d$. It is proved in \cite[Lemma 2.3]{Hyt2024} that $\convex{x}_{X}$ is a convex body in $\C^d$ for all $x\in X^{d}$; the proof is given there only for real normed spaces, but it works verbatim for complex ones. In our application, the role of the Banach space $X$ is  taken over by the Lebesgue spaces $L^p(Q,\frac{\mathrm{d}\mu}{\mu(Q)})$, where $1\leq p<\infty$ and $Q\subseteq M$ is some Borel set (with $0<\mu(Q)<\infty$), equipped with the normalized measure $\frac{\mathrm{d}\mu}{\mu(Q)}$; for brevity we write $\L^{p}(Q):=L^{p}(Q,\frac{\mathrm{d}\mu}{\mu(Q)})$. In this case, $X^{d}$ can be identified with the normalized Lebesgue space $\L^{p}(Q;\C^d):=L^{p}(Q,\frac{\mathrm{d}\mu}{\mu(Q)};\C^d)$. Naturally, for $f=(f_1,\dots,f_d)\in \L^p(Q;\C^d)$ we have
\begin{equation*}
	\dashint_Q \varphi f \dd\mu= \Bigg(\dashint_Q \varphi f_1 \dd\mu,\dots,\dashint_Q \varphi f_d \dd\mu \Bigg) \quad\text{ for any }\varphi\in \L^{p'}(Q),
\end{equation*}
and under the usual identification of $(\L^{p}(Q))^{*}$ with $\L^{p'}(Q)$ we have
\begin{equation*}
	\convex{f}_{\L^p(Q;\C^d)}= \Biggl\{\dashint_Q \varphi f \dd\mu \such \varphi\in\L^{p'}(Q),~\norm{\varphi}_{\L^{p'}(Q)}\leq 1 \Biggr\}.
\end{equation*}

\subsubsection{The bootstrapping principle}

Here, we state and prove our bootstrapping principle.
% To this end, we first need a variant of \cite[Definition 3.1]{NPTV2017} introduced in \cite[Proposition 4.6]{Hyt2024}. \note{definition is moved to the dyadic cubes section}
Recall that a (non-empty) family $\cG\subseteq\dgrid$ covers the (non-empty) family $\cF\subseteq\dgrid$ when $\bigcup_{Q\in\cF}Q\subseteq\bigcup_{P\in\cG}P$, and no dyadic set in $\cG$ is strictly contained in any dyadic set in $\cF$.

Our goal is to prove the following.

%\note{What is the role of the family $\cF$ in the statement of Proposition 4.3? Shouldn't there be an assumption connecting the operators and the family? Or are we saying something else here?}
%
%\bnote{I decided to use a notation closer to the one in the convex body paper, not the one used by Hytonen. So in Hytonen's notation: $\cF=\{B_j\}$, $\cG=\{P_j\}$. The family $\cF$ depends on $\varepsilon, p, T$, the doubling constants, $\lambda$ and $f$. However, for the proposition itself it is not necessary to add this information, although you can add it if you want. However, all this information is not there also in Hytonen's statement.}
%
%\note{Right! I'm blind!}

\begin{prop}
\label{prop:sublinear_get_vectors}
	Fix a dyadic set $Q\in\cald$ and a number $\lambda>1$. Also, take $\varepsilon\in(0,1)$ and $1\leq r_1, r_2<\infty$ with $r_1<r_2'$. For $p\in[r_1,r_2']$, let $\mathbf{T}:L^{p}(M,\mu;\C^d)\to L^{p}_{\cK}(M,\mu)$ be a sublinear map that is compatible with a map $T:L^{p}(M,\mu)\to L^{p}(M,\mu)$. Assume that there exist linear maps $\Lambda_{P}:L^{p}(M,\mu)\to L^{p}(M,\mu)$ indexed by dyadic sets $P\in\cald(Q)$ and constants $C_1,C_2>0$ such that the following holds:
	
\smallskip

	For all $f\in L^{p}(M,\mu)$ supported on $\lambda Q$, there exists a family $\cF$ of pairwise disjoint dyadic sets in $\cald(Q)$ with $\sum_{B\in\cF}\mu(B)\leq\varepsilon\mu(Q)$ such that for any family $\cG$ of pairwise disjoint dyadic sets in $\cald(Q)$ which covers $\cF$ the following estimates hold:
	\begin{gather}
	\label{outside_estimate}
		\int_{Q\setminus\cup_{P\in\cG}P}|T(\Lambda_{Q}f)(x)g(x)|\dd\mu(x)
		\leq C_1\norm{f}_{\L^{r_1}(\lambda Q)}\norm{g}_{\L^{r_2}(\lambda Q)}
	\intertext{and}
	\label{other_statement_which_i_dont_understand_completely}
		\sum_{P\in\cG}\int_{P}|T((\Lambda_{Q}-\Lambda_{P})f)(x)g(x)|\dd\mu(x)
		\leq C_2\norm{f}_{\L^{r_1}(\lambda Q)}\norm{g}_{\L^{r_2}(\lambda Q)}
	\end{gather}
	for any $g\in L^{p'}(M,\mu)$ supported on $\lambda Q$.

\smallskip
	
	Then, for any $f\in L^{p}(M,\mu;\C^d)$ supported on $\lambda Q$, there exists a family $\cG$ of pairwise disjoint dyadic sets in $\cald(Q)$ with $\sum_{P\in\cG}\mu(P)\leq d\varepsilon\mu(Q)$ such that
	\begin{equation}
	\label{eq:sublinear_convex_body_with_tail}
		\int_{Q}|\mathbf{T}(\Lambda_{Q}f)(x)\cdot g(x)|\dd\mu(x)\leq Cd^{3/2}|\convex{f}_{\L^{r_1}(\lambda Q;\C^d)}\cdot\convex{g}_{\L^{r_2}(\lambda Q;\C^d)}|+\sum_{P\in\cG}\int_{P}|\mathbf{T}(\Lambda_{P}f)(x)\cdot g(x)|\dd\mu(x)
	\end{equation}
	for any $g\in L^{p'}(M,\mu;\C^d)$ supported on $\lambda Q$, with $C=C_1+C_2$.
\end{prop}

Our proof of \Cref{prop:sublinear_get_vectors} is a refinement of the proof of \cite[Proposition 4.6]{Hyt2024}. The following result from \cite{Hyt2024} will play a central role.

\begin{lemabc}[\protect{\cite[Lemma 4.1]{Hyt2024}}]
\label{lemabc:one_scale}
	Let $X,Y$ be complex normed spaces, and take $x\in X^{d}$ and $y\in Y^{d}$. Assume that $0\in\mathrm{Int}\,(\convex{x}_{X})$. Let $\cE$ be the John ellipsoid of $\convex{x}_{X}$, and consider an invertible matrix $A\in M_{d}(\C)$ such that $A\cE=\overline{\mathbf{B}}$. Also, let $\{v_1,\ldots,v_d\}$ be an orthonormal basis of $\C^d$. If we set
	\begin{equation*}
		x_i:=v_i\bullet(A\bullet x) \inline{and} y_i:=v_i\bullet((A^{\ast})^{-1}\bullet y) \afterline{for} i=1,\ldots,d,
	\end{equation*}
	then it holds that
	\begin{equation*}
		\sum_{i=1}^{d}\norm{x_i}_{X}\cdot\norm{y_i}_{Y}\leq d^{3/2}|\convex{x}_{X}\cdot\convex{y}_{Y}|.
	\end{equation*}
\end{lemabc}

We omit the proof of \Cref{lemabc:one_scale} and instead refer the reader to \cite[Lemma 4.1]{Hyt2024}. We note that the proof given there concerns the case of normed spaces over the real numbers, but it extends to normed spaces over the complex numbers with only minor modifications.

\begin{proof}[Proof of \Cref{prop:sublinear_get_vectors}]
Take functions $f\in L^{p}(M,\mu;\C^d)$ and $g\in L^{p'}(M,\mu;\C^d)$ both supported on $\lambda Q$. Consider the convex body $K:=\convex{f}_{\L^{r_1}(\lambda Q)}$. The case when $f=0$ $\mu$-a.e. on $\lambda Q$ is trivial. Assume now that $f$ is not $\mu$-a.e. equal to $0$ on $\lambda Q$. Then, we distinguish two cases.

\item[\bf Case 1.]
	$0\in\mathrm{Int}(K)$.
	
	Let $\cE$ be the John ellipsoid of $K$, and let $R:\C^d\to\C^d$ be an invertible linear map such that $R\cE=\overline{\mathbf{B}}$. Pick an orthonormal basis $\{v_1,\ldots,v_d\}$ of $\C^d$. Then, we can write
	\begin{equation*}
		Rf=\sum_{i=1}^{d}f_i v_i \inline{and} (R^{*})^{-1}g=\sum_{i=1}^{d}g_i v_i
	\end{equation*}
	for some $f_1,\ldots,f_d\in L^{p}(M,\mu)$ and $g_1,\ldots,g_d\in L^{p'}(M,\mu)$. For each $i=1,\ldots,d$, let $\cF_i$ be the subfamily of $\cald(Q)$ which is given by the assumptions for the function $f_i\in L^{p}(M,\mu)$. Consider the dyadic family
	\begin{equation*}
		\cG := \left\{\text{maximal dyadic sets in }\bigcup_{i=1}^{d}\cF_{i}\right\},
	\end{equation*}
	where by ``maximal'' we mean that we select those sets from $\bigcup_{i=1}^{d}\cF_{i}$ which are not properly contained in any other set of $\bigcup_{i=1}^{d}\cF_{i}$. Then, observe that $\cG$ consists of disjoint dyadic sets (thanks to dyadic structure) and that it covers $\cF_i$ for all $i=1,\ldots,d$.
	
	Now, we estimate
	\begin{equation*}
		\int_{Q}|\mathbf{T}(\Lambda_{Q}f)(x)\cdot g(x)|\dd\mu(x)
		= \int_{Q\setminus\cup_{P\in\cG}P}|\mathbf{T}(\Lambda_{Q}f)(x)\cdot g(x)|\dd\mu (x)
		+ \sum_{P\in\cG}\int_{P}|\mathbf{T}(\Lambda_{Q}f)(x)\cdot g(x)|\dd\mu (x).
	\end{equation*}
	Applying the subadditivity of $\mathbf{T}$, we obtain
	\begin{equation*}
		\mathbf{T}(\Lambda_{Q}f)(x)\cdot g(x)
		= \mathbf{T}(\Lambda_{Q}f-\Lambda_{P}f+\Lambda_{P}f)(x)\cdot g(x)
		\subseteq \mathbf{T}((\Lambda_{Q}-\Lambda_{P})f)(x)\cdot g(x)
		+ \mathbf{T}(\Lambda_{P}f)(x)\cdot g(x)
	\end{equation*}
	for $\mu$-a.e.~$x\in M$ and for every $P\in\cG$. It follows that
	\begin{equation*}
		\begin{split}
			\int_{Q}|\mathbf{T}(\Lambda_{Q}f)(x) & \cdot g(x)|\dd\mu(x)
			\leq \int_{Q\setminus\cup_{P\in\cG}P}|\mathbf{T}(\Lambda_{Q}f)(x)\cdot g(x)|\dd\mu (x)\\
			& + \sum_{P\in\cG}\int_{P}|\mathbf{T}((\Lambda_{Q}-\Lambda_{P})f)(x)\cdot g(x)|\dd\mu(x)
			+ \sum_{P\in\cG}\int_{P}|\mathbf{T}(\Lambda_{P}f)(x)\cdot g(x)|\dd\mu(x).
		\end{split}
	\end{equation*}
	This means that it only remains to show that
	\begin{equation*}
		\begin{split}
			\int_{Q\setminus\cup_{P\in\cG}P}|\mathbf{T}(\Lambda_{Q}f)(x)\cdot g(x)|\dd\mu (x)
			+\sum_{P\in\cG}\int_{P}|\mathbf{T}( & (\Lambda_{Q}-\Lambda_{P})f)(x)\cdot g(x)|\dd\mu(x)	\\
			&	\leq Cd^{3/2}|\convex{f}_{\L^{r_1}(\lambda Q;\C^d)}\cdot\convex{g}_{\L^{r_2}(\lambda Q;\C^d)}|.
		\end{split}
	\end{equation*}
	
	We start with the first summand.
	In order to lighten the formulae, let us set $F:=Q\setminus\bigcup_{P\in\cG}P$.
	The sublinearity of $\mathbf{T}$ and \eqref{action_on_lines} give us that
	\begin{align*}
		\mathbf{T}(\Lambda_{Q}f)(x) \cdot g(x)
		&	\subseteq\left(\sum_{i=1}^{d}T(\Lambda_{Q}f_i)(x)\overline{\D}R^{-1}v_i\right) \cdot g(x)	\\
		&	= \left(\sum_{i=1}^{d}T(\Lambda_{Q}f_i)(x)\overline{\D}v_i\right)\cdot ((R^{*})^{-1}g(x))	\\
		&	=\sum_{i=1}^{d}(T(\Lambda_{Q}f_i)(x)\overline{g_i(x)})\overline{\D}
	\end{align*}
	for $\mu$-a.e.~$x\in M$. It follows that
	\begin{align*}
		\int_F|\mathbf{T}(\Lambda_{Q}f)(x)\cdot g(x)|\dd\mu (x)
		&	\leq \sum_{i=1}^{d}\int_F|(T(\Lambda_{Q}f_i)(x)\overline{g_i(x)})\overline{\D}|\dd\mu (x)	\\
		&	= \sum_{i=1}^{d}\int_F|T(\Lambda_{Q}f_i)(x)\overline{g_i(x)}|\dd\mu(x).
	\end{align*}
	For each $i=1,\ldots,d$, assumption \eqref{outside_estimate} for the functions $f_i$ and $\overline{g_i}$ implies that
	\begin{equation*}
		\int_F|T(\Lambda_{Q}f_i)(x)\overline{g_i(x)}|\dd\mu(x)\leq
	C_1\norm{f_i}_{\L^{r_1}(\lambda Q)}\norm{g_i}_{\L^{r_2}(\lambda Q)}.
	\end{equation*}
	Thus, using \Cref{lemabc:one_scale} we deduce
	\begin{align*}
		\int_{Q\setminus\bigcup_{P\in\cG}P}|\mathbf{T}(\Lambda_{Q}f)(x)\cdot g(x)|\dd\mu (x)
		&	\leq C_1\sum_{i=1}^{d}\norm{f_i}_{\L^{r_1}(\lambda Q)}\norm{g_i}_{\L^{r_2}(\lambda Q)}	\\
		&	\leq C_1d^{3/2}|\convex{f}_{\L^{r_1}(\lambda Q;\C^d)}\cdot\convex{g}_{\L^{r_2}(\lambda Q;\C^d)}|.
	\end{align*}
	
	Similarly, we obtain
	\begin{equation*}
		\sum_{P\in\cG}\int_{P}|\mathbf{T}((\Lambda_{Q}-\Lambda_{P})f)(x)\cdot g(x)|\dd\mu(x)
		\leq C_2d^{3/2}|\convex{f}_{\L^{r_1}(\lambda Q;\C^d)}\cdot\convex{g}_{\L^{r_2}(\lambda Q;\C^d)}|
	\end{equation*}
	concluding the proof in this case.

%%%%%%%%%%%%%%%%%%%%	Dimitris (note to self): proofread this case	%%%%%%%%%%%%%%%%%%%%%%%%
\item[\bf Case 2.]
	$\mathrm{Int}(K)=\emptyset$. %%% Remarks
	%\note{Why doesn't this case work the same way as in Hyt{\"o}nen? Maybe his argument doesn't work? But why?}

 %\bnote{But it does work in the exact same way as in Hytonen. I have merely done the exact same adaptation as Hytonen did. He just sweeps a lot of details under the rug that I chose to write explicitly, like the operator $S$. Ok, I do concede that there is the difference that because of the absence of linearity of $\mathbf{T}$, we cannot just crudely say ``repeat the proof of the non-degenerate case with $f$ replaced by $Pf$ and $g$ replaced by $Pg$'' as he says. Instead we have to do the thing a bit more carefully.}
	
	Let $H$ be the affine subspace of $\C^d$ generated by $K$, and observe that $H$ is a vector subspace of $\C^d$. Let $P:\C^d\to H$ be the orthogonal projection of $\C^d$ on $H$. Recall that $P$ is a Hermitian linear map. Moreover, one can easily see that $f=Pf$; see the second part of the proof of \cite[Proposition 4.2]{Hyt2024}.
	
	Set $m:=\dim_{\C}H>0$ ($m$ is positive, because $f$ is not the zero function), and let $S:H\to\C^m$ be an invertible linear map with $S(v)\cdot S(w)=v\cdot w$ for all $v,w\in H$. Let $\cE$ be the John ellipsoid of the convex body $S(K)$ in $\C^m$ and let $R:\C^m\to\C^m$ be an invertible linear map such that $R\cE=\overline{\mathbf{B}}_{m}$, where $\overline{\mathbf{B}}_m$ is the Euclidean closed unit ball of $\C^m$. Pick an orthonormal basis $\{v_1,\ldots,v_m\}$ of $\C^m$. Let us write
	\begin{equation*}
		RSf=RSPf=\sum_{i=1}^{m}f_iv_i,\quad (R^{\ast})^{-1}SPg=\sum_{i=1}^{m}g_iv_i,
	\end{equation*}
	for some $f_1,\ldots,f_m\in L^{p}(M,\mu)$ and $g_1,\ldots,g_m\in L^{p'}(M,\mu)$. For each $i=1,\ldots,m$, let $\cF_i$ be the subfamily of $\cald(Q)$ offered by the assumptions for the function $f_i\in L^{p}(M,\mu)$. Let
	\begin{equation*}
		\cG := \left\{\text{maximal dyadic sets in }\bigcup_{i=1}^{m}\cF_{i}\right\}.
	\end{equation*}
	Observe that $\cG$ covers $\cF_i$, for all $i=1,\ldots,m$. Then, similarly to Case 1, we only have to show that
	\begin{equation}
	\label{estimate_goal}
		\begin{split}
			\int_{Q\setminus\cup_{P\in\cG}P}|\mathbf{T}(\Lambda_{Q}f)(x)\cdot g(x)|\dd\mu (x)
			+ \sum_{P\in\cG}\int_{P} |\mathbf{T}( & (\Lambda_{Q}-\Lambda_{P})f)(x) \cdot g(x)|\dd\mu(x)	\\
			& \leq Cd^{3/2}|\convex{f}_{\L^{r_1}(\lambda Q;\C^d)}\convex{g}_{\L^{r_2}(\lambda Q;\C^d)}|.
		\end{split}
	\end{equation}
	First of all, set $F:=Q\setminus\bigcup_{P\in\cF}P$. Then, by the sublinearity of $\mathbf{T}$ and \eqref{action_on_lines} we have
	\begin{align*}
		\mathbf{T}(\Lambda_{Q}f)(x)\cdot g(x) \dd\mu(x)
		&	\subseteq \left( \sum_{i=1}^{m} T(\Lambda_{Q}f_i)(x)\overline{\D}S^{-1}R^{-1}v_i \right) \cdot g(x)	\\
		&	= \left(\sum_{i=1}^{m}T(\Lambda_{Q}f_i)(x)\overline{\D}PS^{-1}R^{-1}v_i\right)\cdot g(x)	\\
		&	= \left( \sum_{i=1}^{m}T(\Lambda_{Q}f_i)(x)\overline{\D}v_i \right) \cdot (R^{*})^{-1}SPg(x)	\\
		&	= \sum_{i=1}^{m}(T(\Lambda_{Q}f_i)(x)\overline{g_i(x)})\overline{\D}.
	\end{align*}
	The estimates proceed then as in case 1 with $d$ replaced by $m$, $f$ by $Sf$ and $g$ by $SPg$, so that at the end we obtain
	\begin{equation*}
		\begin{split}
			\int_{Q\setminus\bigcup_{P\in\cG}P}|\mathbf{T}(\Lambda_{Q}f)(x)\cdot g(x)\dd\mu (x)|
			+\sum_{P\in\cG}\int_{P}|\mathbf{T}( & (\Lambda_{Q}-\Lambda_{P})f)(x)\cdot g(x)|\dd\mu(x)	\\
			&	\leq Cm^{3/2}|\convex{Sf}_{\L^{r_1}(\lambda Q;\C^d)}\cdot\convex{SPg}_{\L^{r_2}(\lambda Q;\C^d)}|.
		\end{split}
	\end{equation*}
	Finally, we observe that
	\begin{align*}
		\convex{Sf}_{\L^{r_1}(\lambda Q;\C^d)}\cdot\convex{SPg}_{\L^{r_2}(\lambda Q;\C^d)}
		&	= (S\convex{f}_{\L^{r_1}(\lambda Q;\C^d)})\cdot(SP\convex{g}_{\L^{r_2}(\lambda Q;\C^d)})	\\
		&	= (P\convex{f}_{\L^{r_1}(\lambda Q;\C^d)})\cdot\convex{g}_{\L^{r_2}(\lambda Q;\C^d)}	\\
		&	= \convex{f}_{\L^{r_1}(\lambda Q;\C^d)}\cdot\convex{g}_{\L^{r_2}(\lambda Q;\C^d)},
	\end{align*}
	concluding the proof.
\end{proof}

\begin{rem} \label{rem:cubes_can_be_taken_small}
	Examining the previous proof, we see that in the assumptions of \Cref{prop:sublinear_get_vectors} we can restrict the covering collections $\cG$ to the ones which additionally satisfy $\sum_{P\in\cG}\mu(P) \leq d\varepsilon\mu(Q)$.
\end{rem}

\subsection{Applying the bootstrapping principle}

Here we accomplish the extension of \cite[Theorem 5.7]{BFP2016} to vector-valued functions, with the help of \Cref{prop:sublinear_get_vectors}.

The set up is as follows. Recall that $1\leq p_0< 2< q_0\leq \infty$, and fix $p\in(p_0,q_0)$. Let $\mathbf{T}:L^{p}(M,\mu;\C^d)\to L^{p}_{\cK}(M,\mu)$ be a sublinear operator compatible with a bounded sublinear operator $T:L^{p}(M,\mu)\to L^{p}(M,\mu)$ satisfying the assumptions in \Cref{section:preliminaries} (see Definitions \ref{dfn:sublinear_convex_set} and \ref{dfn:compatibility_scalar_convex_body}).

\newtheorem*{bilinear_convex_body_sparse_domination}{\Cref{thm:bilinear_convex_body_sparse_domination}}
\begin{bilinear_convex_body_sparse_domination}
	Let $p\in (p_0,q_0)$. Let $\cald$ be a dyadic system in $M$ as in \Cref{subsec:dyadic}. Let also $0<\varepsilon<1$. Then, there exists a constant $C=C(T,\mu,p_0,q_0,\varepsilon)>0$ such that for all $f\in L^p(M,\mu;\C^d)$ and $g\in L^{p'}(M,\mu;\C^d)$, whose coordinate functions are all supported on $\lambda Q_0$ for some cube $Q_0\in \cald$ and sufficiently large (depending on $\dgrid$) $\lambda\geq 2$, there exists an $\varepsilon$-sparse collection $\cals\subseteq \cald$ (depending on $T,p_0,q_0,d,f$ and $g$) with
    \[\int_{Q_0} |\mathbf{T}f \cdot g| \dd\mu\leq Cd^{3/2}\sum_{P\in \cals}μ(P)\Big|\convex{f}_{\L^{p_0}(\lambda P;\C^d)}\cdot\convex{g}_{\L^{q_0'}(\lambda P;\C^d)}\Big|.\]
%{\color{green}
%	\[\left|\int_{Q_0} T f \cdot g \dd\mu\right|\leq Cd^{3/2}\sum_{P\in \cals}μ(P)\Big|\convex{f}_{\L^{p_0}(\lambda P;\C^d)}\cdot\convex{g}_{\L^{q_0'}(\lambda P;\C^d)}\Big|.\]
%}
%{\color{blue}
%}
%for any\note{?} sublinear operator $\mathbf{T}:L^{p}(M,\mu;\C^d)\to L^{p}_{\cK}(M,\mu)$ which is compatible with~$T$.
%%% Remarks
%\note{check last sentence} \bnote{ok we will check}
\end{bilinear_convex_body_sparse_domination}

The proof of this theorem follows very closely the original proof for scalar-valued functions up to the point where one needs to check that the assumptions of \Cref{prop:sublinear_get_vectors} hold. Here, we present the important definitions and briefly mention all necessary properties, but the reader would have to confer \cite{BFP2016} for more details. In addition, as we explained in \Cref{subsec:dyadic}, the parameter $\lambda$ ($5$ in the original proof of \cite[Theorem 5.7]{BFP2016}) is made here more precise.

\begin{proof}[Sketch of proof]
	For $p\in (p_0,q_0)$ and $\lambda \geq 2$ (to be fixed later), consider functions $f\in L^p(M,\mu)$ and $g\in L^{p'}(M,\mu)$ supported on $\lambda Q_0$, where $Q_0\in \dgrid$.
%	Fix also a parameter $b\in\{1,\dots,K\}$ so that $Q_0\in \dgrid^b$.
	
\medskip
    
	Now, let $\eta>0$ be some large enough constant (also to be fixed later) and define the following set:
	\[E:= E_{\eta, \lambda}:= \left\{x\in Q_0 \such \max\{\cM^*_{Q_0,p_0}f(x), T^\#_{Q_0}f(x)\}> \eta \norm{f}_{\L^{p_0}(\lambda Q_0)}\right\}.\]
	(Remember the definitions \eqref{eq:sharp_maximal_operator} and \eqref{eq:sharp_T_operator}.)
	
	For large enough $\eta> 0$, the set $E$ is properly contained inside $Q_0$, and $\mu(E)\leq \frac{K}{\eta} \mu(Q_0)$ (for some $K>0$). In fact, $E$ is a union of dyadic cubes, therefore $E$ is open (as dyadic sets are open themselves). Hence, one can find a ``maximal'' covering of $E$ by dyadic sets $\{B_j\}_j \subseteq \dgrid(Q_0)$, which are disjoint up a to a set of $\mu$-measure zero; readily this means that
	\begin{itemize}[leftmargin=1.5\parindent]
		\item $E= \bigsqcup_j B_j$ (up to a null set), and
		\item $\hat B_j\cap E^c\not= \emptyset$ for all $j$. (Recall $\hat B_j$ is the parent of $B_j$.) 
	\end{itemize}
	In particular, $\sum_{j=1}^\infty \mu(B_j) = \mu(E) \leq \frac{K}{\eta} \mu(Q_0)$.
	
\medskip
%{\color{green}
	Next, for any $P\in \dgrid(Q_0)$ we define the linear operators $\Lambda_{P}$:
	\begin{gather*}
		\Lambda_{P}f:=\begin{cases}
			\int^{\ell(P)^2}_{0}\cQ_{t}^{(N)}(f\1_{\lambda P})\,\frac{\mathrm{d}t}{t}	&	\text{when}~P\not=Q_0,	\\
			f	&	\text{when}~P=Q_0.
		\end{cases}
	\end{gather*}
	
%	\[T_P f:= T\int_0^{\ell(P)^2}\cQ_t^{(N)}(f\1_{\lambda P})\,\frac{\mathrm{d}t}{t}.\]
	
	In Step 1 of their proof, Bernicot et al. show that
%	\begin{equation} \label{eq:thm5.7}
%		\Biggl|\int_{Q_0} Tf \cdot g \dd \mu - \sum_j \int_{B_j} T(\Lambda_{B_j}f) \cdot g \dd\mu \Biggr|\leq C\eta \mu(Q_0) \norm{f}_{\L^{p_0}(\lambda Q_0)} \norm{g}_{\L^{q_0'}(\lambda Q_0)}.
%	\end{equation}
	\begin{equation} \label{eq:thm5.7}
		\Biggl|\int_{Q_0} Tf \cdot g \dd \mu \Biggr|\leq C\eta \mu(Q_0) \norm{f}_{\L^{p_0}(\lambda Q_0)} \norm{g}_{\L^{q_0'}(\lambda Q_0)} + \Biggl| \sum_j \int_{B_j} T_{B_j}f \cdot g \dd\mu \Biggr|.
	\end{equation}
	This is attained by noticing that each $\hat{B}_j$ is contained in $Q_0$, contains points which lie inside $Q_0$ but outside $E$, and for $\mu$-almost every such $x\in Q_0\setminus E$ the quantities $\cM^*_{Q_0,p_0}f(x)$ and $T^\#_{Q_0}f(x)$ are sufficiently small. The rest comes down to a series of several precise decompositions and approximations.
%}

%{\bf\color{blue} Instead should be added: $\Lambda_{P}f=\int^{\ell(P)^2}_{0}Q_{t}^{(N)}(f\1_{\lambda P})\frac{\mathrm{d}t}{t}$ if $P\neq Q_0$, $\Lambda_{Q_0}f=f$.} 
	
\medskip
	
%\textcolor{red}{
%	Here, \eqref{eq:thm5.7} has exactly the form of \cref{eq:sublinear_convex_body_with_tail} for the forms $t_Q(f,g):=\int_{Q}T_{Q} f \cdot g \dd \mu$. Our first goal is to show condition \ref{HytRealValue} of \Cref{thm:Hytonen_to_get_vectors} holds true. Then, according to \Cref{thm:Hytonen_to_get_vectors}, this would immediately imply that condition \ref{HytVecValue} is true as well, which is precisely what we need here. Eventually, the infinite sum on the left-hand side enters an iterative process and goes to zero.
%}
	
	Here, \eqref{eq:thm5.7} is similar to \cref{eq:sublinear_convex_body_with_tail} for $Q=Q_0$.
%	
%	{\bf \color{blue} I cannot see why this is true. I think the exact analog would be
%	
%	\begin{equation*}
%	\int_{Q_0} |Tf \cdot g| \dd \mu \leq \sum_j \int_{B_j} |T(\Lambda_{B_j}f) \cdot g| \dd\mu + C\eta \mu(Q_0) \norm{f}_{\L^{p_0}(\lambda Q_0)} \norm{g}_{\L^{q_0'}(\lambda Q_0)}.
%	\end{equation*}
%	
%	}
%	
	Our aim is to show that the conditions of the assumption in \Cref{prop:sublinear_get_vectors} hold. Then, the sum on the left-hand side will enter an iterative process and go to zero.

	Towards this goal, let $\{P_i\}_i \subseteq \dgrid(Q_0)$ be a dyadic collection, which is comprised of disjoint dyadic sets, and covers $\{B_j\}_j$, that is, the $P_i$'s are not strictly contained in any of the $B_j$'s, and $\bigcup_j B_j \subseteq \bigcup_i P_i$.
%	{\bf \color{green} We want to show that \eqref{eq:thm5.7} holds true with $B_j$'s replaced by $P_i$'s: 
%	\begin{equation*} \label{eq:thm5.7_for_larger_cubes}
%		\Biggl|\int_{Q_0} Tf \cdot g \dd \mu - \sum_i \int_{P_i} T(\Lambda_{P_i}f) \cdot g \dd\mu\Biggr|\leq C_0\eta \mu(Q_0) \norm{f}_{\L^{p_0}(\lambda Q_0)} \norm{g}_{\L^{q_0'}(\lambda Q_0)}.
%	\tag{\ref*{eq:thm5.7}'}
%	\end{equation*}
% }
	In fact, we can additionally assume that $\sum_i \mu(P_i) \leq \frac{dK}{\eta} \mu(Q_0)$; however, this is will not be necessary.
	Moreover, the $P_i$'s satisfy that $E \subseteq \bigsqcup_i P_i$ (up to a $\mu$-null set) and $\hat{P}_i \cap E^c \not= \emptyset$, ``almost like'' the collection of the $B_j$'s.
	
\medskip
	
	Naturally, our proof is very close to the one for Theorem 5.7 in \cite{BFP2016}, but this is unavoidable. For this reason and to keep things contained, we omit the parts which apply verbatim in our case.
	
\smallskip
	
	For starters, there might exist only one such $P_i$, namely $Q_0$ itself. In this case, both \eqref{outside_estimate} and \eqref{other_statement_which_i_dont_understand_completely} hold trivially as the left-hand side would be zero, and thus \eqref{eq:sublinear_convex_body_with_tail} also holds.
	
	From now on, we assume that all $P_i$'s are properly contained in $Q_0$. Thanks to the dyadic structure, this means that $\hat{P}_i \subseteq Q_0$. And since $\hat{P}_i \cap E^c \not= \emptyset$, the definitions of $E$, $M_{Q_0,p_0}^*$ and $T_{Q_0}^\#$ imply that
	\begin{gather}
	\label{eq:inf_of_maximal}
		\inf_{y\in \hat{P}_i} \cM_{p_0}f(y)  \leq \eta \norm{f}_{\L^{p_0}(\lambda Q_0)}
	\intertext{and}
	\label{eq:non_sharp_T}
		\Biggl(\dashint_{\hat{P}_i} \Bigg| T\int_{\ell(\hat{P}_i)^2}^\infty \cQ_t^{(N)}(f) \,\frac{\mathrm{d}t}{t} \Bigg|^{q_0} \dd \mu \Biggr)^\frac{1}{q_0} \leq \eta \norm{f}_{\L^{p_0}(\lambda Q_0)}.
	\end{gather}
%	{\bf \color{green}
%	Next, we write
%	\begin{align*}
%		\int_{Q_0} Tf \cdot g \dd\mu
%		&	= \int_{Q_0 \setminus \cup_i P_i} Tf \cdot g \dd\mu + \int_{\cup_i P_i} Tf \cdot g \dd\mu	\\
%		&	= \int_{Q_0 \setminus \cup_i P_i} Tf \cdot g \dd\mu + \sum_i \int_{P_i} T(\Lambda_{P_i}f) \cdot g \dd\mu + \sum_i \int_{P_i} (T-T\Lambda_{P_i})f \cdot g \dd\mu,
%	\end{align*}
%	which gives us that
%	\begin{equation}
%	\label{eq:the_first_split}
%		\int_{Q_0} Tf \cdot g \dd\mu - \sum_i \int_{P_i} T(\Lambda_{P_i}f) \cdot g \dd\mu
%		= \int_{Q_0 \setminus \cup_i P_i} Tf \cdot g \dd\mu + \sum_i \int_{P_i} (T-T\Lambda_{P_i})f \cdot g \dd\mu.
%	\end{equation}
%	}
%	{\bf \color{blue} The following computation you have done checks already the validity of the first condition. No modification is needed, up to replacing $\left| \int_{Q_0 \setminus \cup_i P_i} Tf \cdot g \dd\mu \right|$ with $ \int_{Q_0 \setminus \cup_i P_i} |Tf \cdot g| \dd\mu $.}
	
	For the first condition of \Cref{prop:sublinear_get_vectors}, we can apply H{\"o}lder's inequality which, combined with the fact that $|Tf(x)| \leq T^\#_{Q_0}f(x) \leq \eta \norm{f}_{\L^{p_0}(\lambda Q_0)}$ for almost every $x\in Q_0 \setminus E$, gives us
	\begin{equation*}
%	\label{eq:nice_first_part}
		\int_{Q_0 \setminus \cup_i P_i} |Tf \cdot g| \dd\mu
		\leq \eta \mu(Q_0) \norm{f}_{\L^{p_0}(\lambda Q_0)} \norm{g}_{\L^{q_0'}(Q_0)}
		\lesssim \eta \mu(Q_0) \norm{f}_{\L^{p_0}(\lambda Q_0)} \norm{g}_{\L^{q_0'}(\lambda Q_0)}.
	\end{equation*}
	This is exactly \cref{outside_estimate}.
	
%	 {\bf \color{blue} In the next computations, there will be no $(T-T_{P_i})f$, but rather $T(f-\Lambda_{P_i}f)$. So one wll instead break directly $f-\Lambda_{P_i}f$, as you have actually already done, and use subadditity of $T$ (not $\mathbf{T}$!), as you have already done. This way the second condition is checked. I remain by my opinion that absolutely no modifications are needed. Up to replacing each $\left| \sum_i \int_{P_i} (T-T_{P_i})f \cdot g \dd\mu \right|$ with $\sum_i \int_{P_i} |T(f-\Lambda_{P_i}f) \cdot g| \dd\mu$ and $(T-T_{P_i})f$ with $T(f-\Lambda_{P_i}f)$.
%}
	
	The verification of the second condition is much trickier. Like above, H{\"o}lder's inequality gives us that
	\begin{equation} \label{eq:Hoelder_of_the_ugly_part}
		\sum_i \int_{P_i} |T(f-\Lambda_{P_i}f) \cdot g| \dd\mu
			\leq \sum_i \mu(P_i) \Bigg( \dashint_{P_i} |T(f-\Lambda_{P_i}f)|^{q_0} \dd\mu \Bigg)^\frac{1}{q_0} \Bigg( \dashint_{P_i} |g|^{q_0'} \dd\mu \Bigg)^\frac{1}{q_0'}.
	\end{equation}
	
	Further, one needs to bound the $\L^{q_0}(P_i)$ norm of $T(f-\Lambda_{P_i}f)$. We do not present all the details in our paper, as this was already done in \cite{BFP2016}. Nonetheless, we explain what we think is important and necessary for the proof to work for the larger (than $B_j$) dyadic sets $P_i$.
	
	Start by fixing an index $i \geq 1$, and recall the definition \ref{def:sublinearity_definition} of sublinearity for $T$.
	Then, we have
	\begin{equation} \label{eq:split_to_work_outside_the_balls}
	\begin{aligned}
		|T(f-\Lambda_{P_i}f)|
			&	= \Big| T \left( \int_0^\infty \cQ_t^{(N)}(f) \,\frac{\mathrm{d}t}{t} - \int_0^{\ell(P_i)^2} \cQ_t^{(N)}(f\1_{\lambda P_i}) \,\frac{\mathrm{d}t}{t} \right) \Big|	\\
			&	= \Big| T \left( \int_0^{\ell(P_i)^2} \!\cQ_t^{(N)}(f\1_{\lambda P_i}+f\1_{(\lambda P_i)^c}) \,\frac{\mathrm{d}t}{t} + \int_{\ell(P_i)^2}^\infty \!\cQ_t^{(N)}(f) \,\frac{\mathrm{d}t}{t} - \int_0^{\ell(P_i)^2} \!\cQ_t^{(N)}(f\1_{\lambda P_i}) \,\frac{\mathrm{d}t}{t} \right) \!\Big|	\\
			&	\leq \Big| T\int_0^{\ell(P_i)^2} \cQ_t^{(N)}(f\1_{(\lambda P_i)^c}) \,\frac{\mathrm{d}t}{t} \Big| + \Big| T\int_{\ell(P_i)^2}^\infty \cQ_t^{(N)}(f) \,\frac{\mathrm{d}t}{t} \Big| =: |\cT_1(P_i;\lambda)| + |\cT_2(P_i)|.
	\end{aligned}
	\end{equation}
%{\bf\color{blue} In the second line of the previous computation, you already broke $f-\Lambda_{P_i}f$. To go to the third line, you used subadditivity of $T$ (not $\mathbf{T}$!). All this should remain totally unchanged.}
%{\bf\color{blue} The rest of the proof requires no modifications whatsoever.}
	Here, we used Calder{\'o}n reproducing formula (\Cref{prop:calderon_reproducing_formula}), and the linearity of~$\cQ_t^{(N)}$.
	
	For the $\L^{p_0}(P_i)$ norms of $\cT_1(P_i;\lambda)$ and $\cT_2(P_i)$, fix $k_0\in \Z_{\geq 0}$ and take $\lambda$ to be at least $2^{k_0}$ so that for any $Q\in \dgrid$
	\[(\lambda Q)^c \subseteq \bigcup_{k \geq k_0} S_k(Q).\]
	
	After a series of geometrical observations and computations, Bernicot et al. show that for $k_0=2$ and any $j=1,2,\dots$
	\begin{equation} \label{eq:geometric_argument}
	\begin{aligned}
		\norm{\cT_1(B_j;\lambda)}_{\L^{q_0}(B_j)}
			&	= \left( \dashint_{B_j} \Big| T\int_0^{\ell(B_j)^2} \cQ_t^{(N)}(f\1_{(\lambda B_j)^c}) \,\frac{\mathrm{d}t}{t} \Big|^{q_0} \dd\mu \right)^\frac{1}{q_0}	\\
			&	\leq \int_0^{\ell(B_j)^2} \left( \dashint_{B_j} |T\cQ_t^{(N)}(f\1_{(\lambda B_j)^c})|^{q_0} \dd\mu \right)^\frac{1}{q_0} \,\frac{\mathrm{d}t}{t}	\\
			&	\lesssim\dots\lesssim \sum_{k\geq k_0} 2^{-k} \norm{f}_{\L^{p_0}(S_k(B_j))} \lesssim \sup_{k\geq k_0} \norm{f}_{\L^{p_0}(2^{k+1}B_j)},
	\end{aligned}
	\end{equation}
	where the implicit constants depend only on $p_0$, $q_0$, the doubling measure $\mu$, and the parameters $c_0$, $C_0$, and $\delta$ of the dyadic system. (The first inequality comes from the sublinearity \eqref{eq:decomposition_linearity_operator_inequality} of the operator~$T$ followed by Minkowski's inequality.) In fact, their proof holds true for any $k_0\geq 0$ and for any $Q\in \dgrid$ in place of the $B_j$'s; in particular
	\[\norm{\cT_1(P_i;\lambda)}_{\L^{q_0}(P_i)} \lesssim \sup_{k\geq k_0} \norm{f}_{\L^{p_0}(2^{k+1}P_i)}.\]
	
	Notice, now, that if $\hat{P}_i \subseteq 2^{k+1}P_i$ for all $k\geq k_0$, then
	\[\sup_{k\geq k_0} \norm{f}_{\L^{p_0}(2^{k+1}P_i)} \lesssim \inf_{y\in \hat{P}_i} \cM_{p_0}f(y)\]
	thanks to \Cref{lemma:covering_dilation_of_dyadic_set_with_ball}. The smallest $k_0$ for which this can be true (recall the definition \eqref{eq:dilation}) is
	\[k_0 = \lceil \log_2\big(1+\frac{2}{\delta}\big) - 1\rceil.\]
	Consequently, we can take $\lambda \geq 1+\frac{2}{\delta}$. (We note that in the case of the ``usual'' dyadic cubes in $\R^m$, where $\delta=\frac{1}{2}$, one can simply take $\lambda=5$.)
	
	Using \eqref{eq:inf_of_maximal}, we eventually get
	\begin{equation} \label{eq:T_the_first}
		\norm{\cT_1(P_i;\lambda)}_{\L^{q_0}(P_i)} \lesssim \eta \norm{f}_{\L^{p_0}(\lambda Q_0)}.
	\end{equation}
	
	For $\cT_2$, the sublinearity of $T$ gives us that
	\[|\cT_2(P_i)| \leq \Big| T\int_{\ell(\hat{P}_i)^2}^\infty \cQ_t^{(N)}(f) \,\frac{\mathrm{d}t}{t}\Big| + \Big|T\int_{\ell(P_i)^2}^{\ell(\hat{P}_i)^2} \cQ_t^{(N)}(f) \,\frac{\mathrm{d}t}{t}\Big|.\]
	Taking the $\L^{q_0}$-norm over $P_i$, we get
	\begin{align*}
		\norm{\cT_2(P_i)}_{\L^{q_0}(P_i)}
			&	\lesssim \left( \dashint_{\hat{P}_i} \Big| T\int_{\ell(\hat{P}_i)^2}^\infty \cQ_t^{(N)}(f) \,\frac{\mathrm{d}t}{t} \Big|^{q_0} \dd\mu \right)^\frac{1}{q_0} + \left( \dashint_{\hat{P}_i} \Big| T\int_{\ell(P_i)^2}^{\ell(\hat{P}_i)^2} \cQ_t^{(N)}(f) \,\frac{\mathrm{d}t}{t} \Big|^{q_0} \dd\mu \right)^\frac{1}{q_0}	\\
			&	 \leq \eta \norm{f}_{\L^{p_0}(\lambda Q_0)} + \int_{\ell(P_i)^2}^{\ell(\hat{P}_i)^2} \left( \dashint_{\hat{P}_i} |T\cQ_t^{(N)}(f)|^{q_0} \dd\mu \right)^\frac{1}{q_0} \,\frac{\mathrm{d}t}{t}.
	\end{align*}
	The first part of the second inequality comes from \eqref{eq:non_sharp_T}; the second from Minkowski's inequality and the sublinearity of $T$.
	
	At this point, the same argument used in \eqref{eq:geometric_argument} works verbatim for our second summand above, since
	\[
		\int_{\ell(P_i)^2}^{\ell(\hat{P}_i)^2} \left( \dashint_{\hat{P}_i} |T\cQ_t^{(N)}(f)|^{q_0} \dd\mu \right)^\frac{1}{q_0} \,\frac{\mathrm{d}t}{t}
		\leq \int_0^{\ell(\hat{P}_i)^2} \left( \dashint_{\hat{P}_i} |T\cQ_t^{(N)}(f)|^{q_0} \dd\mu \right)^\frac{1}{q_0} \,\frac{\mathrm{d}t}{t}.
	\]
	The difference is that the absence of the $\1_{(\lambda \hat{P}_i)^c}$ term forces the sum, and in turn the supremum, to start from $k_0=0$:
	\[
		\int_{\ell(P_i)^2}^{\ell(\hat{P}_i)^2} \left( \dashint_{\hat{P}_i} |T\cQ_t^{(N)}(f)|^{q_0} \dd\mu \right)^\frac{1}{q_0} \,\frac{\mathrm{d}t}{t}
		\lesssim \sup_{k\geq 0} \norm{f}_{\L^{p_0}(2^{k+1}\hat{P}_i)}.
	\]
	
	But now, $\hat{P}_i \subseteq 2^{k+1}\hat{P}_i$ is true for all $k\geq0$, and therefore through \Cref{lemma:covering_dilation_of_dyadic_set_with_ball} again one gets
	\[\sup_{k\geq 0} \norm{f}_{\L^{p_0}(2^{k+1}\hat{P}_i)} \lesssim \inf_{y\in \hat{P}_i} \cM_{p_0}f(y).\]
	Hence, \eqref{eq:inf_of_maximal} along with the above inequalities give us that
	\begin{equation} \label{eq:T_the_second}
		\norm{\cT_2(P_i)}_{\L^{q_0}(P_i)} \lesssim \eta \norm{f}_{\L^{p_0}(\lambda Q_0)}.
	\end{equation}
	
	Combining \eqref{eq:T_the_first} and \eqref{eq:T_the_second} with \eqref{eq:split_to_work_outside_the_balls}, we finally have that
	\[\norm{(T-T\Lambda_{P_i})f}_{\L^{q_0}(P_i)} \lesssim \eta \norm{f}_{\L^{p_0}(\lambda Q_0)}.\]
	And substituting this back to \eqref{eq:Hoelder_of_the_ugly_part}, we get
	\begin{equation} \label{eq:ugly_put_together}
		\sum_i \int_{P_i} |T(f-\Lambda_{P_i}f) \cdot g| \dd\mu
		\lesssim  \eta \norm{f}_{\L^{p_0}(\lambda Q_0)}	\sum_i \mu(P_i) \Big( \dashint_{P_i} |g|^{q_0'} \dd\mu \Big)^\frac{1}{q_0'}.
	\tag{\ref*{eq:Hoelder_of_the_ugly_part}'}
	\end{equation}
	
	For $q_0'=1$ (or $q_0=\infty$),
	\[
		\sum_i \mu(P_i) \Big( \dashint_{P_i} |g|^{q_0'} \dd\mu \Big)^\frac{1}{q_0'}
		= \sum_i \int_{P_i} |g| \dd\mu
		= \int_{\cup_i P_i} |g| \dd\mu
		\lesssim \mu(Q_0) \dashint_{\lambda Q_0} |g| \dd\mu,
	\]
	since the $P_i$'s are disjoint and ${\cup_i P_i}\subseteq Q_0$. Inputting this back to \eqref{eq:ugly_put_together}, we get exactly \eqref{other_statement_which_i_dont_understand_completely}.
	
	For $q_0'\in(1,2)$ (or $q_0\in(2,\infty)$), first observe that $\big( \dashint_{Q} |g|^{q_0'} \dd\mu \big)^\frac{1}{q_0'} \leq \inf_{y\in Q} \cM_{q_0'}g(y)$ for any $Q\in \dgrid$. Then, since the $P_i$'s are disjoint, we have that
	\begin{align*}
		\sum_i \mu(P_i) \Big( \dashint_{P_i} |g|^{q_0'} \dd\mu \Big)^\frac{1}{q_0'}
			&	\leq \sum_i \mu(P_i) \inf_{y\in P_i} \cM_{q_0'}g(y) \leq \sum_i \int_{P_i} \cM_{q_0'}g(y) \dd\mu(y)	\\
			&	= \int_{\cup_i P_i} \cM_{q_0'}g(y) \dd\mu(y) \lesssim \mu(\cup_i P_i)^{1-\frac{1}{q_0'}} \norm{g^{q_0'}}_{L^1(\cup_i P_i)}^\frac{1}{q_0'}	\\
			&	\leq \mu(Q_0) \norm{g}_{\L^{q_0'}(Q_0)} \lesssim \mu(Q_0) \norm{g}_{\L^{q_0'}(\lambda Q_0)}.
	\end{align*}
	For the inequality in the second line, we have used Kolmogorov's inequality and the boundedness of~$\cM_{q_0'}$; in the third line, we used the fact that $\sum_i \mu(P_i) \leq \mu(Q_0)$ and \Cref{lemma:covering_dilation_of_dyadic_set_with_ball}. This combined with \eqref{eq:ugly_put_together} again gives us exactly~\eqref{other_statement_which_i_dont_understand_completely}.
	
	Therefore, since both \cref{outside_estimate,other_statement_which_i_dont_understand_completely} are true, the assumptions of \Cref{prop:sublinear_get_vectors} are fulfilled. Consequently, there exists a dyadic family $\{Q_1^j\}_j$ of pairwise disjoint sets in $\dgrid(Q_0)$ with $\sum_j\mu(Q_1^j)\leq\frac{dK}{\eta}\mu(E)$ and such that
	\begin{equation*}
		\int_{Q_0}|\mathbf{T}(\Lambda_{Q_0}f)(x)\cdot g(x)|\dd\mu(x)\leq Cd^{3/2}|\convex{f}_{\L^{p_0}(\lambda Q_0;\C^d)}\cdot\convex{g}_{\L^{q_0'}(\lambda Q_0;\C^d)}|+\sum_{j}\int_{Q_1^{j}}|\mathbf{T}(\Lambda_{Q_1^{j}}f)(x)\cdot g(x)|\dd\mu(x)
	\end{equation*}
	for all $f\in L^p(M,\mu;\C^d)$ and $g\in L^{p'}(M,\mu,\C^d)$ whose coordinates are supported on $\lambda Q_0$.
	
%{\color{green}
%	That is, there exists some collection of disjoint dyadic sets $\{\tilde{B}_j\}_j \subseteq \dgrid(Q_0)$ with $\sum_j \mu(\tilde{B}_j) \leq d\frac{K}{\eta}\mu(Q_0)$ such that
%	\begin{equation}
%	\label{eq:thm5.7_for_vectors}
%		\left|\int_{Q_0} Tf \cdot g \dd\mu - \sum_j \int_{\tilde{B}_j} T_{\tilde{B}_j}f \cdot g \dd\mu \right|
%		\leq Cd^{3/2}\eta μ(Q_0)\Big|\convex{f}_{\L^{p_0}(\lambda Q_0;\C^d)}\cdot\convex{g}_{\L^{q_0'}(\lambda Q_0;\C^d)}\Big|,
%	\end{equation}
%	for all $f\in L^p(M,\mu;\C^d)$ and all $g\in L^{p'}(M,\mu;\C^d)$ whose coordinate functions are supported on $\lambda Q_0$.}
	
\medskip
	
	Now, we need to construct our sparse family. This can be done by following Step 2 in the proof of Bernicot et al. almost without any changes; the fact that here we are dealing with vector-valued functions does not affect their argument. The conclusion is similar to \cite[Corollary 5.4]{Hyt2024}.

 	Here, we explain the iterative procedure with which one constructs the sparse family. Initially, set~$\cals:=\{Q_0\}$. Then, we can rewrite the above inequality in the following form
	\begin{align*}
		\int_{Q_0}|\mathbf{T}(\Lambda_{Q_0}f)(x)\cdot g(x)|\dd\mu(x)&\leq Cd^{3/2}|\convex{f}_{\L^{p_0}(\lambda Q_0;\C^d)}\cdot\convex{g}_{\L^{q_0'}(\lambda Q_0;\C^d)}|\\
        &+\sum_{j}\int_{Q_1^{j}}|\mathbf{T}(\Lambda_{Q_1^{j}}f^j)(x)\cdot g^j(x)|\dd\mu(x),
	\end{align*}
	where (for all $j$) $f^j=\1_{\lambda Q^j_1}f$, and $g^j=\1_{\lambda Q^j_1}g$; $f^j$ and $g^j$ are supported on $\lambda Q_1^j$ (by the definition of $\Lambda_P$ and since the integration is over $Q_1^j$).
	
	Next, we can add the collection of all such $Q_1^j$ into $\cals$ by redefining $\cals := \cals \cup \bigcup_j\{Q_1^j\}$, and then iterate on each $Q_1^j$ to get
	\begin{align*}
		\int_{Q_1^j}|\mathbf{T}(\Lambda_{Q_1^j}f^j)(x)\cdot g^j(x)|\dd\mu(x)&\leq Cd^{3/2}|\convex{f}_{\L^{p_0}(\lambda Q_1^{j};\C^d)}\cdot\convex{g}_{\L^{q_0'}(\lambda Q_1^{j};\C^d)}|\\
        &+\sum_{k}\int_{Q_2^{j,k}}|\mathbf{T}(\Lambda_{Q_2^{j,k}}f^{j,k})(x)\cdot g^{j,k}(x)|\dd\mu(x),
	\end{align*}
	and so on.
	
	Eventually, as we show below, the remainder terms go to $0$ and we conclude that
	\[\int_{Q_0} |\mathbf{T}f \cdot g|\dd\mu \leq Cd^{3/2}\eta \sum_{Q\in \cals} μ(Q) \Big| \convex{f}_{\L^{p_0}(\lambda Q;\C^d)}\cdot\convex{g}_{\L^{q_0'}(\lambda Q;\C^d)} \Big|.\]
	
	Following the argument in \cite{BFP2016} one needs to pick an $\eta$ large enough so that the series $\sum_{\alpha\geq1} \big(d\frac{K}{\eta}\big)^\alpha \leq \frac{dK}{\eta-dK}$ converges. For such large $\eta$, the dyadic collection $\cals\subseteq\dgrid(Q_0)$ is $\varepsilon$-sparse with~$\varepsilon = 1-\frac{dK}{\eta-dK}$.
	
	All that is left to do is to check that the remainder term indeed goes to $0$.

	Let $\cG_n$ be the family of all cubes added at the initial sparse collection $\{ Q_0\}$ at the $n$-th step of our construction. The remainder term at the $n$-th step is clearly
	\begin{equation*}
		\sum_{P\in\cG_n}\int_{P}|\mathbf{T}(\Lambda_{P}f)(x)\cdot g(x)|\dd\mu(x).
	\end{equation*}
	By construction, $\lim_{n\rightarrow\infty}\mu(F_n)=0$, where $F_n:=\bigcup_{P\in\cG_n}P$.
 
    Write now $f=(f_1,\ldots,f_d)$ and $g=(g_1,\ldots,g_d)$. Then, similarly to the proof of \Cref{prop:sublinear_get_vectors} we have
	\begin{equation*}
		\sum_{P\in\cG_n}\int_{P}|\mathbf{T}(\Lambda_{P}f)(x)\cdot g(x)|\dd\mu(x)
		\leq \sum_{i=1}^{d}\sum_{P\in\cG_n}\int_{P}|T(\Lambda_{P}f_i)(x)g_i(x)|\dd\mu(x).
	\end{equation*}
	Fix $i\in\{1,\ldots,d\}$. We show that
    \begin{equation*}
        \lim_{n\rightarrow\infty}\sum_{P\in\cG_n}\int_{P}|T(\Lambda_{P}f_i)(x)g_i(x)|\dd\mu(x)=0.
    \end{equation*}
    Although this follows by a standard argument, we include the details for the reader's convenience.

    Apply first H{\"o}lder's inequality:
	\begin{equation*}
		\sum_{P\in\cG_n}\int_{P}|T(\Lambda_{P}f_i)(x)g_i(x)|\dd\mu(x) \leq \sum_{P\in\cG_n}\left(\int_{P}|T(\Lambda_{P}f_i)(x)|^{p}\dd\mu(x)\right)^{1/p} \left(\int_{P}|g_i(x)|^{p'}\dd\mu(x)\right)^{1/p'}.
	\end{equation*}
	By \Cref{rem:operator_estimates} (iii) and \Cref{tails_Calderon_formula},  we deduce that the operators $\Lambda_{P}$, $P\in\dgrid(Q_0)$ are $L^{p}(M,\mu)$-bounded, \emph{uniformly in the cube $P$}. Thus, since $T$ is also by assumption $L^{p}(M,\mu)$ bounded we obtain
	\begin{align*}
		&\sum_{P\in\cG_n}\left(\int_{P}|T(\Lambda_{P}f_i)(x)|^{p}\dd\mu(x)\right)^{1/p}\left(\int_{P}|g_i(x)|^{p'}\dd\mu(x)\right)^{1/p'}\\
		&\lesssim\norm{T}_{L^{p}(M,\mu)}\sum_{P\in\cG_n}\left(\int_{P}|f_i(x)|^{p}\dd\mu(x)\right)^{1/p}\left(\int_{P}|g_i(x)|^{p'}\dd\mu(x)\right)^{1/p'}.
	\end{align*}
	Applying H{\"o}lder's inequality (for sums) and taking into account the disjointness of the sets $P\in\cG_n$, we get
	\begin{align*}
		&	\sum_{P\in\cG_n}\left(\int_{P}|f_i(x)|^{p}\dd\mu(x)\right)^{1/p}\left(\int_{P}|g_i(x)|^{p'}\dd\mu(x)\right)^{1/p'}	\\
		&	\leq \left(\sum_{P\in\cG_n}\int_{P}|f_i(x)|^{p}\dd\mu(x)\right)^{1/p}\left(\sum_{P\in\cG_n}\int_{P}|g_i(x)|^{p'}\dd\mu(x)\right)^{1/p'}	\\
		&	=\left(\int_{F_n}|f_i|^{p}\dd\mu\right)^{1/p}\left(\int_{F_n}|g_i|^{p'}\dd\mu\right)^{1/p'}.
	\end{align*}
	Since the functions $|f_i|^{p}$ and $|g_i|^{p'}$ are $\mu$-integrable and $\lim_{n\rightarrow\infty}\mu(F_n)=0$, we deduce the desired result.
\end{proof}

%%%%%%%%%%%%%%%%%%%%%%%%%%%%			section				%%%%%%%%%%%%%%%%%%%%%%%%%%%%%%%%%%%%%%%%

\section{Matrix weighted bounds for bilinear convex body sparse forms }\label{section:matrix_weighted_bounds}

In this section we obtain matrix weighted bounds for bilinear convex body sparse forms of the type appearing in \Cref{thm:bilinear_convex_body_sparse_domination}. Our main result is the following.

\newtheorem*{matrix_weighted_bound}{\Cref{thm:matrix_weighted_bound}}
\begin{matrix_weighted_bound}
Let $1\leq p_0<q_0\leq\infty$ and $p\in(p_0,q_0)$. Let $\cald$ be a dyadic system in $M$ as in Subsection~\ref{subsec:dyadic} and let $\lambda\geq2$. Let $0<\varepsilon<1$ and let $\cals\subseteq\cald$ be an $\varepsilon$-sparse collection. Set
    \begin{equation*}
        t:=\frac{p}{p_0},\quad s:=\frac{q_0}{p},\quad a:=\frac{t'p}{t},\quad b:=s'p.
    \end{equation*}
    Let $W$ be a $d\times d$ matrix weight on $M$ such that $[W^{s'}]_{A_{a',b}}<\infty$. Then, there exists a constant $C=(\mu,d,p_0,q_0,p)$ such that for all $f\in L^{p}_{W}(M,\mu)$ and all $g\in L^{p'}_{W'}(M,\mu)$ the following estimate holds:
    \begin{equation*}
        \sum_{P\in\cals}\mu(P)|\convex{f}_{\L^{p_0}(\lambda P)}\cdot\convex{g}_{\L^{q_0'}(\lambda P)}|\leq\frac{C}{\varepsilon}([W^{s'}]_{A_{a',b}})^{\alpha}\norm{f}_{L^{p}_{W}(M,\mu)}\norm{g}_{L^{p'}_{W'}(M,\mu)},
    \end{equation*}
    where $W':=W^{-1/(p-1)}$ and
    \begin{equation*}
        \alpha:=\frac{1}{s'(p-p_0)}+\frac{1}{p'}.
    \end{equation*}
\end{matrix_weighted_bound}

After reviewing some notation and facts on the so-called \emph{reducing operators} in Subsection \ref{subsection:reducing_operators}, which furnish a major tool for manipulating matrix weights in estimates, we introduce the so-called \emph{Muckenhoupt characteristics} for matrix weights in Subsection \ref{subsection:muckenhoupt_weights}. Finally, we prove \Cref{thm:matrix_weighted_bound} in Subsection \ref{subsec:proof_matrix_weighted_estimate}.

Further in the paper, we implement the following elementary facts without special mention:
\begin{enumerate}[label=(\roman*)]
    \item For all $x_1,\ldots,x_k\geq0$, for all $a>0$, and for all positive integers $k$, we have
    \begin{equation}
    \label{break power of sum}
        (x_1+\ldots+x_k)^{a}\simeq_{a,k} x_1^{a}+\ldots+x_k^{a}.
    \end{equation}
    In the rest of the text, we mostly apply this for $k=1,2,3,4$ or $d$ non-negative numbers. Moreover, the place of $a$ will be taken by the exponent or exponents at hand.
	
%    \bnote{Nein, ``$k$'' ist in der Tat als ``$k$'' zu verstehen, also als die Anzahl der eingesetzten Zahlen, nicht als den Exponenten. Unten gibt es weder $a$ noch $k$, da wir diese Abschätzung für mehrere verschiedene Werte von $a$ und $k$ anwenden. Typischerweise nehmen wir $k=2$ oder $k=d$. Letzteres erklärt u.a.~warum die unmittelber folgende Equivalenz zu Matrizennormen so sinnvoll ist. Was $a$ angeht, wird sein Platz von den jeweiligen Exponenten eingenommen. Z.B.~beim Beweis von Lemma 4.14 (1) ist $a$ mal $st$, mal $s$, mal $t$. Dennoch war die obige Formulierung zugegebenermaßen u.U.~leicht verwirrend. Daher die Präzisierung.}
    
    \item For a matrix $A\in M_{d}(\C)$, we denote by $|A|$ the \emph{spectral norm} of $A$, that is, the largest singular value of $A$ (i.e. the square root of the largest eigenvalue of the matrix $A^*A$). Then, it is true that
    \begin{equation}
    \label{matrix norm as sum of norms of columns}
        |A|\simeq_{d}\sum_{i=1}^{d}|Ae_i|,
    \end{equation}
    where $(e_1,\ldots,e_d)$ is the standard orthonormal basis of $\C^d$.

    %\item If $A,B\in\mathrm{P}_{d}(\C)$ with $|Ae|\leq|Be|$, for all $e\in\C^d$, then $|B^{-1}e|\leq|A^{-1}e|$, for all $e\in\C^d$.
\end{enumerate}

\subsection{Reducing operators} \label{subsection:reducing_operators}

Integral averages of scalar weights play a major role in the investigation of scalar weighted estimates. In the context of matrix weighted estimates, it is often the case that plain integral averages of positive definite matrix valued functions are poorly suited for the problems at hand. The concept of the so-called \emph{reducing operators} offers in many cases a possible remedy.

Let $(X,m)$ be a measure space and let $E$ be a measurable subset of $X$ with $0<m(E)<\infty$. Recall that an $m$-measurable function $f:E\to M_{d}(\C)$ (for some positive integer $d$) is said to be \emph{$m$-integrable over $E$} if
\begin{equation*}
    \int_{E}|f(x)|\dd m(x)<\infty,
\end{equation*}
and that in this case we denote
\begin{equation*}
    \aver{f}_{E}:=\frac{1}{m(E)}\int_{E}f(x)\dd m(x)=\dashint_{E}f(x)\dd m(x).
\end{equation*}

	\begin{dfn}
	\begin{enumerate}[labelindent=!,leftmargin=1.5\parindent]
	
		\item Let $1\leq p<\infty$, and $W:E\to\mathrm{M}_{d}(\C)$ be an $m$-integrable function such that $W(x)$ is a positive definite matrix for $m$-a.e.~$x\in E$. Then, one can consider the norm
		\begin{equation*}
		    e \mapsto \rho_{W,E,p}(e):=\left(\dashint_{E}|W(x)^{1/p}e|^{p}\dd m(x)\right)^{1/p}
		\end{equation*}
		acting on vectors $e\in\C^d$. Let $K\subseteq\C^d$ be the closed unit ball of $\rho_{W, E,p}$. In the terminology of \Cref{section:convex_set_preliminaries}, $K$ is a convex body in $\C^d$ with $0\in\mathrm{Int}(K)$. Let $\cE$ be the John ellipsoid of $K$. Subsequently, as explained in \Cref{section:convex_set_preliminaries}, there exists a positive definite matrix $A\in\mathrm{M}_{d}(\C)$ with $A\overline{\mathbf{B}}=\cE$, where $\overline{\mathbf{B}}$ is the standard (Euclidean) closed unit ball in $\C^d$. Following \cite{Goldberg_2003}, we call the matrix $A^{-1}$ \emph{the reducing operator of $W$ over $E$ with respect to the exponent $p$} and denote $\cW_{E,p}:=A^{-1}$.
		
		(We omit the measure $m$ from the notation of the reducing operator.) Since $\cE\subseteq K\subseteq d^{1/2}\cE$ (see \cite{Goldberg_2003, NPTV2017}), one easily deduces that
		\begin{equation*}
		    \left(\dashint_{E}|W(x)^{1/p}e|^{p}\dd m(x)\right)^{1/p}\leq|\cW_{E,p}e| \leq d^{1/2}\left(\dashint_{E}|W(x)^{1/p}e|^{p}\dd m(x)\right)^{1/p}\quad\forall e\in\C^d.
		\end{equation*}
		Notice that if $p=2$ or $d=1$, we simply have $\cW_{E,p}=\left(\dashint_{E}W(x)\dd m(x)\right)^{1/p}$.
		
		%\bnote{In der Forschungs- und gelegentlich auch Mathematik-Standardsprache bzw.~in dem Jargon bedeutet der Ausdruck ``kanonische Auswahl'', dass man sich nicht des Auswahlaxioms bedienen muss, um diese Auswahl zu treffen. Es handelt also nicht um ein Objekt, dessen bloße Existenz durch ein abstraktes Ergebnis wie z.~B.~der Satz von Hahn--Banach gesichert wird. Vielmehr steht ein systematisches, zumindest teilweise \emph{konstruktives} Verfahren zur Verfügung. Ich gebe zu, der Ausdruck wirkt vage und ungeschickt. Daher ergänze ich eine hoffentlich einleuchtende Erläuterung.}
		
		\item If $1<p<\infty$ and $W:E\to\mathrm{M}_{d}(\C)$ is a $m$-measurable function with $W(x)\in\mathrm{P}_{d}(\C)$ for $m$-a.e.~$x\in E$ such that the function $V:=W^{-1/(p-1)}$ is $m$-integrable over $E$, then we denote $\cW'_{E,p}:=\cV_{E,p'}$.
	\end{enumerate}
\end{dfn}

Notice that, if $1<s<\infty$, and $W:E\to\mathrm{M}_{d}(\C)$ is a $m$-integrable function with $W(x)\in\mathrm{P}_{d}(\C)$ for $m$-a.e.~$x\in E$, then, by Jensen's inequality, we have
\begin{equation*}
    \dashint_{E}|W(x)^{1/s}|\dd m(x)=\dashint_{E}|W(x)|^{1/s}\dd m(x)\leq\left(\dashint_{E}|W(x)|\dd m(x)\right)^{1/s}<\infty.
\end{equation*}
Therefore, the function $W^{1/s}$ is also $m$-integrable over $E$. However, this might not be the case for $W^s$.

The following standard lemma, whose proof we omit, helps to avoid ambiguous terminology below.
\begin{lem}
\label{lemma:standardlemma}
    Let $(X,m)$ be a measure space and let $E$ be a measurable subset of $X$ with $0<m(E)<\infty$.
 Let $1\leq a,b<\infty$, and $W,V:E\to\mathrm{M}_{d}(\C)$ be $m$-measurable functions with $W(x),V(x)\in\mathrm{P}_{d}(\C)$ for $m$-a.e.~$x\in E$. Then, we have
    \begin{equation*}
        \dashint_{E}\left(\dashint_{E}|W(x)^{1/b}V(y)^{1/a}|^{a}\dd m(y)\right)^{b/a}\dd m(x)<\infty\iff W,~V\text{ are $m$-integrable over }E.
    \end{equation*}
    Moreover, in this case, the following holds:
    \begin{equation*}
    \dashint_{E}\left(\dashint_{E}|W(x)^{1/b}V(y)^{1/a}|^{a}\dd m(y)\right)^{b/a}\dd m(x)\simeq_{d,a,b}|\cW_{E,b}\cV_{E,a}|^{b}.
    \end{equation*}
\end{lem}

\subsection{Matrix Muckenhoupt weights} \label{subsection:muckenhoupt_weights}

The weight $W$ in \Cref{thm:matrix_weighted_bound} was assumed to satisfy the condition $[W^{s'}]_{A_{a',b}}<\infty$, which is a so-called \emph{matrix Muckenhoupt condition}. We explain the terminology in this subsection.

%%% Moved to subsection 2.5
%\begin{dfn}
%Let $W:M\to M_{d}(\C)$ be a measurable function with $W(x)\in\mathrm{P}_{d}(\C)$ for $\mu$-a.e.~$x\in M$. We say that $W$ is a \emph{$d\times d$ matrix weight on $M$} if $\int_{E}|W(x)|\dd\mu(x)<\infty$ for any bounded Borel subset $E$ of $M$.
%\end{dfn}

\begin{dfn}\label{dfn:muckenhoupt_weights}
Let $1<p\leq q<\infty$, and $W:M\to \mathrm{M}_{d}(\C)$ be a measurable function with values $\mu$-a.e.~in the set of positive-definite $d \times d$ complex-valued matrices. We define
\begin{equation}
    \label{matrix-weight}
    [W]_{A_{p,q}} := \sup_{B} \dashint_B \left(\dashint_B |W(x)^{\frac{1}{q}} W(y)^{-\frac{1}{q}}|^{p'} \dd\mu(y) \right)^{\frac{q}{p'}} \dd\mu(x),
\end{equation}
where the supremum ranges over all open balls $B\subseteq M$. If $[W]_{A_{p,q}}<\infty$, then we say that $W$ is a \emph{$d\times d$ matrix $A_{p,q}$ weight}, and we write $W\in A_{p,q}$. If $p=q$, then we denote $A_{p,p}:=A_{p}$.
\end{dfn}

Variants of the matrix $A_{p,q}$ weights were already introduced and studied in \cite{Hunt_Bellman_1996, Lauzon_Treil_2007}. The particular form of condition \eqref{matrix-weight} considered here was first introduced and studied in \cite{Isralowitz_Moen_2019} and further investigated in \cite{Cardenas_Isralowitz_I2021}.

Below, we present several properties of the matrix $A_{p,q}$ weights that are important for our purposes. These were all established in \cite{Isralowitz_Moen_2019} in the case when the ambient space $M$ is simply $\R^n$ equipped with the Euclidean metric and the $n$-dimensional Lebesgue measure, but they still hold true in arbitrary doubling measure metric spaces simply by replacing cubes with balls. See \cite[Section 3]{Isralowitz_Moen_2019} for the proofs.

\begin{lem}
\label{lemma:properties_of_Apq_weights}
	Let $1<p\leq q<\infty$, and let $W$ be a $d\times d$ matrix $A_{p,q}$ weight on $M$. Set $r:=1+\frac{q}{p'}$. Then, the following are true:
	\begin{align*}
		[W^{-p'/q}]_{A_{q',p'}}	&	\simeq_{d,p,q}[W]_{A_{p,q}}^{1/(r-1)};	\\
		\sup_{e\in\C^d}[|W^{1/q}e|^{q}]_{A_{r}}	&	\lesssim_{d,p,q}[W]_{A_{p,q}};	\\
		\sup_{e\in\C^d}[|W^{-1/q}e|^{p'}]_{A_{r'}}	&	\lesssim_{d,p,q}[W^{-p'/q}]_{A_{q',p'}}.
	\end{align*}
\end{lem}

We also need the so-called ``weak'' reverse H{\"o}lder property of scalar Muckenhoupt weights, in general, doubling measure metric spaces, as obtained from \cite[Theorem 1.1]{Hytonen_Perez_Rela_2012} coupled with (a trivial generalization of) \cite[Lemma 4.1]{NPTV2017}.

\begin{lem}
    \label{lemma:weak_reverse_Holder_for_scalar_Muckenhoupt_weights}
    Let $w$ be a scalar weight on $M$ with $[w]_{A_{p}}<\infty$ for some $1<p<\infty$. Then, there exist constants $a_1,a_2>0$ depending only on the doubling measure $\mu$ and $p$ such that for all $0<\varepsilon\leq\frac{1}{1+a_1[w]_{A_{p}}}$ we have
    \begin{equation*}
        \aver{w^{1+\varepsilon}}_{B}^{1/(1+\varepsilon)}\leq a_2\aver{w}_{2B},
    \end{equation*}
    for all open balls $B\subseteq M$, where $2B$ denotes the ball co-centric with $B$ and twice its radius.
\end{lem}

In particular, \Cref{lemma:weak_reverse_Holder_for_scalar_Muckenhoupt_weights} can be used to show that scalar $A_p$ weights for $1<p<\infty$ satisfy an openness property, see, for instance, \cite{Hytonen_Perez_Rela_2012} for a precise statement. However, counterexamples due to \cite{Bownik_2001} show that matrix $A_p$ weights on $\R^n$ fail to satisfy an analogous openness property for any $1<p<\infty$. 

\subsection{Proving the matrix-weighted bound}
\label{subsec:proof_matrix_weighted_estimate}

We devote this subsection to a proof of \Cref{thm:matrix_weighted_bound}.

First, we need some preparation. Fix $1\leq p_0<p<q_0\leq\infty$. Next, set
\begin{equation*}
	t:=\frac{p}{p_0},\quad s:=\frac{q_0}{p},\quad \tilde{s}:=\frac{p'}{q_0'},
\end{equation*}
and
\begin{equation*}
	a:=\frac{t'p}{t},\quad b:=s'p.
\end{equation*}
Observe that $t, s, \tilde{s}, a, b > 1$, and $t, \tilde{s}, a, b <\infty$. Direct computation gives
\begin{equation*}
	a=\frac{p}{t-1}=t'p_0,\quad b=\frac{p'}{\tilde{s}-1}=\tilde{s}'q_0',
\end{equation*}
and
\begin{equation*}
	\frac{1}{a}+\frac{1}{b}=\frac{1}{p_0}-\frac{1}{q_0}\leq1.
\end{equation*}
In particular, we have $b\geq a'$. Finally, define
\begin{equation*}
	r:=s'(t-1)+1.
\end{equation*}
Observe that $r \in (1,\infty)$ and
\begin{equation*}
	\frac{b}{a}=r-1=\frac{s't}{t'}.
\end{equation*} 

Let $W$ be a $d\times d$ matrix weight on $M$. Set $U:=W^{-1/(t-1)}$ and $V:=W^{s'}$. Notice that
\begin{equation*}
	U=V^{-1/(r-1)} \inline{and} V^{1/b} = W^{1/p} = U^{-1/a}.
\end{equation*}
Assume now that
\begin{equation*}
	[V]_{A_{a',b}}<\infty.
\end{equation*}
Then, by \Cref{lemma:standardlemma} we deduce the following:
\begin{equation}
	\label{Ap reducing operators}
	|\cU_{B,a}\cV_{B,b}|\lesssim_{d,a,b}[V]_{A_{a',b}}^{1/b},
\end{equation}
for all open balls $B\subseteq M$. Also, \Cref{lemma:properties_of_Apq_weights} for $p=a'$ and $q=b$ gives us that
\begin{gather}
\label{uniform scalar Ap}
	[|V^{1/b}e|^{b}]_{A_{r}}\lesssim_{d,a,b}[V]_{A_{a',b}}\quad\forall e\in\C^d\setminus\{0\},	\\
\label{dual uniform scalar Ap}
	[|U^{1/a}e|^{a}]_{A_{r'}}\lesssim_{d,a,b}[U]_{A_{b',a}}\quad\forall e\in\C^d\setminus\{0\},	\\
\intertext{and}
\label{dual characteristic}
	[U]_{A_{b',a}}\simeq_{d,a,b}[V]_{A_{a',b}}^{1/(r-1)}.
\end{gather}

Now, we are finally ready to prove \Cref{thm:matrix_weighted_bound}

\begin{proof}[Proof (of \Cref{thm:matrix_weighted_bound})]
	We adapt the arguments from \cite[Section 5]{NPTV2017}. We keep the setup and notation from the previous subsection. Throughout the proof, we ignore all implied constants depending only on $\mu,d,p_0,q_0$ and $p$. First of all, it suffices to prove that
	\begin{equation*}
		\sum_{P\in\cals}\mu(P)|\convex{W^{-1/p}f}_{\L^{p_0}(\lambda P)}\cdot\convex{W^{1/p}g}_{\L^{q_0'}(\lambda P)}|\lesssim\frac{1}{\varepsilon}([W^{s'}]_{A_{a',b}})^{\alpha}\norm{f}_{L^{p}(M,\mu;\C^d)}\norm{g}_{L^{p'}(M,\mu;\C^d)},
	\end{equation*}
	for all $f\in L^{p}(M,\mu;\C^d)$ and all $g\in L^{p'}(M,\mu;\C^d)$. In what follows, we fix such $f$ and $g$. From \Cref{lemma:covering_dilation_of_dyadic_set_with_ball} we have that for all $P\in\cals$, there exists an open ball $B(P)\subseteq M$ such that $\lambda P\subseteq B(P)$ and $\mu(B(P))\lesssim\mu(\lambda P)$. Thus, we begin by estimating
	\begin{align*}
		A
		&	:=\sum_{P\in\cals}\mu(P)|\convex{W^{-1/p}f}_{\L^{p_0}(\lambda P)}\cdot\convex{W^{1/p}g}_{\L^{q_0'}(\lambda P)}|\\
		&	=
		\sum_{P\in\cals}\mu(P)|(\calu_{2B(P),a}^{-1}\convex{W^{-1/p}f}_{\L^{p_0}(\lambda P)})\cdot(\calu_{2B(P), a}\convex{W^{1/p}g}_{\L^{q_0'}(\lambda P)})|\\
		&	\leq\sum_{P\in \cals} \mu(P) \aver*{|\calu_{2B(P),a}^{-1} W^{-1/p} f|^{p_0}}_{\lambda P}^{1/p_0}\aver*{|\calu_{2B(P),a} W^{1/p} g|^{q_0'}}_{\lambda P}^{1/q_0'}\\
		&	\lesssim\sum_{P\in \cals} \mu(P) \aver*{|\calu_{2B(P),a}^{-1} W^{-1/p} f|^{p_0}}_{B(P)}^{1/p_0}\aver*{|\calu_{2B(P),a} W^{1/p} g|^{q_0'}}_{B(P)}^{1/q_0'}.
	\end{align*}
	We apply H{\"o}lder's inequality for the exponents $p$ and $p'$:
	\begin{equation*}
		A \leq
		\left( \sum_{P \in \cals} \mu(P) \aver*{|\calu_{2B(P),a}^{-1} W^{-1/p} f|^{p_0}}_{B(P)}^{p/p_0} \right)^{1/p}
		\left( \sum_{P \in \cals} \mu(P)\aver*{|\calu_{2B(P),a} W^{1/p} g|^{q_0'}}_{B(P)}^{p'/q_0'} \right)^{1/p'}
		=: A_1 \cdot A_2.
	\end{equation*}
	
	Now we estimate $A_1$ and $A_2$ from above separately. We start with $A_1$:
	\begin{align*}
		A_1^p
		&	= \sum_{P \in \cals} \mu(P) \aver*{|\calu_{2B(P),a}^{-1} W^{-1/p} f|^{p_0}}_{B(P)}^{p/p_0} = \sum_{P \in \cals} \mu(P)  \aver*{|\calu_{2B(P),a}^{-1} U^{1/a} f|^{a/t'}}_{B(P)}^{t}\\
		&	\leq \sum_{P \in \cals} \mu(P)  \aver*{|\calu_{2B(P),a}^{-1} U^{1/a}|^{a/t'} \cdot |f|^{a/t'}}_{B(P)}^{t}.
	\end{align*}
	Set  $\theta:=t'(1+\delta)$, for some $\delta>0$ to be chosen later. Then, applying H{\"o}lder's inequality for the exponents $\theta$ and $\theta'$, we get
	\begin{align*}
		 A_1^p
		 &	\leq \sum_{P \in \cals} \mu(P) \aver*{|\calu_{2B(P),a}^{-1} U^{1/a}|^{a\theta/t'}}_{B(P)}^{t/\theta} \cdot {\aver*{|f|^{a\theta'/t' }}}_{B(P)}^{t/\theta'} \\
		 &	=\sum_{P \in \cals} \mu(P) \aver*{|\calu_{2B(P),a}^{-1} U^{1/a}|^{a(1+\delta)}}_{B(P)}^{(t-1)/(1+\delta)} \cdot {\aver*{|f|^{a\theta'/t' }}}_{B(P)}^{t/\theta'}.
	\end{align*}
	For every $P\in\cals$, consider the scalar weight
	\begin{equation*}
		\omega_{P}:=|\calu_{2B(P),a}^{-1} U^{1/a}|^a = |U^{1/a} \calu_{2B(P),a}^{-1}|^a
	\end{equation*}
	on $M$. Then, by \eqref{dual uniform scalar Ap} coupled with (a tiny modification of) \cite[Lemma 3.4]{DKPS2024}, we deduce
	\begin{equation*}
		\sup_{P\in\cals}[\omega_{P}]_{A_{r'}} \lesssim[U]_{A_{b',a}}.
	\end{equation*}
	Thus, by the weak Reverse H{\"o}lder property for scalar Muckenhoupt weights, \Cref{lemma:weak_reverse_Holder_for_scalar_Muckenhoupt_weights}, there exists a constant $a_1>0$ depending only on $\mu,d,p_0,q_0$ and $p$ such that for the choice
	\begin{equation*}
		\delta: = \frac{1}{1+a_1[U]_{A_{b', a}}}
	\end{equation*}
	we have, for every $P\in\cals$,
	\begin{equation*}
		\aver*{\omega_{P}^{1+\delta}}_{B}^{1/(1+\delta)}\lesssim\aver*{\omega_{P}}_{2B},
	\end{equation*}
	for every open ball $B\subseteq M$. Thus, for every $P\in\cals$, we obtain
	\begin{equation*}
		\aver*{\omega_P^{1 + \delta}}_{B(P)}^{1/(1 + \delta)} \lesssim \aver*{\omega_P}_{2B(P)} =  \dashint_{2B(P)}|U(x)^{1/a} \calu_{2B(P), a}^{-1}|^{a} \dd\mu(x)\simeq |\cU_{2B(P),a}\cU_{2B(P),a}^{-1}|^{a}=1.
	\end{equation*}
	Thus, we have
	\begin{align*}
		A_1^{p}
		&\lesssim \sum_{P \in \cals}\mu(P)\aver*{|f|^{a\theta'/t'}}_{B(P)}^{t/\theta'}
		\leq\frac{1}{\varepsilon}\sum_{P\in\cals}\mu(F_{P})\aver*{|f|^{a\theta'/t'}}_{B(P)}^{t/\theta'}\\
		&\leq\frac{1}{\varepsilon}\sum_{P\in\cals}\int_{F_{P}}(\calm(|f|^{a\theta'/t'})(x))^{t/\theta'}\dd\mu(x)
		\leq\frac{1}{\varepsilon}\int_{M}(\calm(|f|^{a\theta'/t'})(x))^{t/\theta'}\dd\mu(x),
	\end{align*}
	since the sets $F_{P}$, $P\in\cals$ are pairwise disjoint. Observe that $\theta>t'$, therefore $\theta'<t$, that is $q:=t/\theta' > 1$. Thus, applying \eqref{Hardy_Litlewood_bound} we obtain
	\begin{equation*}
		A_1^{p}
		\lesssim\frac{1}{\varepsilon}(q')^{q}\int_{M}(|f(x)|^{a\theta'/t'})^{q}\dd\mu(x)
		=(q')^{q}\frac{1}{\varepsilon}\int_{M}|f(x)|^{p}\dd\mu(x).
	\end{equation*}
	\Cref{lemma:final_estimate} below shows that $(q')^{q}\lesssim\delta^{-1}$, therefore we finally obtain
	\begin{equation*}
		A_1\lesssim\frac{1}{\varepsilon^{1/p}}[U]_{A_{b',a}}^{1/p}\norm{f}_{L^{p}(M,\mu;\C^d)}.
	\end{equation*}
	
	To estimate $A_2$, we first observe that for all $P\in\cals$, by \eqref{Ap reducing operators} we get
	\begin{align*}
		\aver*{|\cU_{2B(P),a} W^{1/p} g|^{q_0'}}_{B(P)}
		&	\leq\aver*{|\cU_{2B(P),a} \cV_{2B(P), b} \cV_{2B(P),b}^{-1} W^{1/p}g|^{q_0'}}_{B(P)}	\\
		&	\leq |\cU_{2B(P),a} \cV_{2B(P),b}|^{q_0'} \cdot \aver*{|\cV_{2B(P),b}^{-1} W^{1/p}g|^{q_0'}}_{B(P)}	\\
		&	\lesssim [V]_{A_{a',b}}^{q_0'/b} \cdot \aver*{|\cV_{2B(P),b}^{-1} W^{1/p}g|^{q_0'}}_{B(P)}.
	\end{align*}
	Similarly to the above and using \eqref{uniform scalar Ap} instead of \eqref{dual uniform scalar Ap}, we obtain
	\begin{equation*}
		\left( \sum_{P \in \cals} \mu(P)\aver*{|\cV_{2B(P),b}^{-1} W^{1/p} g|^{q_0'}}_{B(P)}^{p'/q_0'} \right)^{1/p'}
		\lesssim\frac{1}{\varepsilon^{1/p'}}[V]_{A_{a',b}}^{1/p'}\norm{g}_{L^{p'}(M,\mu;\C^d)}.
	\end{equation*}
	Putting all estimates together and using \eqref{dual characteristic}, we finally deduce
	\begin{equation*}
		A \lesssim \frac{1}{\varepsilon}([V]_{A_{a',b}})^{\alpha}\norm{f}_{L^{p}(M,\mu;\C^d)}\norm{g}_{L^{p'}(M,\mu;\C^d)},
	\end{equation*}
	where
	\begin{equation*}
		\alpha
		:=\frac{1}{p(r-1)}+\frac{1}{b}+\frac{1}{p'}
		=\frac{1}{s'p(t-1)}+\frac{1}{s'p}+\frac{1}{p'}
		=\frac{t}{s'p(t-1)}+\frac{1}{p'}
		=\frac{1}{s'(p-p_0)}+\frac{1}{p'},
	\end{equation*}
	concluding the proof.
\end{proof}

\begin{lem}
\label{lemma:final_estimate}
	Let $0<\delta<1$. Set $\theta:=t'(1+\delta)$ and $q:=\frac{t}{\theta'}$. Then, we have $(q')^{q}\lesssim\delta^{-1}$.
\end{lem}

\begin{proof}
	Set $y:= q - 1> 0$, then
	
	\begin{equation*}
		y = \frac{t}{\theta'} - 1 = \frac{t}{\frac{\theta}{\theta - 1}} - 1 = \frac{t(\theta -1) - \theta}{\theta}.
	\end{equation*}
	
	Since $\theta = t' (1 + \delta)$, we rewrite
	
	\begin{equation*}
		y = \frac{t(t' (1 + \delta) -1) -t' (1 + \delta)}{t' (1 + \delta)},
	\end{equation*}
	which gives
	\begin{equation}\label{yyy}
		y
		= \frac{t\left(\frac{t}{t-1} (1 + \delta) -1\right) - \frac{t}{t-1} (1 + \delta)}{\frac{t}{t-1} (1 + \delta)}
		= \frac{(t(1 + \delta) - t + 1) - (1 + \delta)}{(1 + \delta)}
		= \frac{\delta}{1 + \delta} (t - 1).
	\end{equation}
	
	Now, we observe that
	\begin{equation*}
		(q')^q = ((y+1)')^{y + 1} = \left(1 + \frac{1}{y} \right)^{y + 1} \leq e \left(1 + \frac{1}{y} \right).
	\end{equation*}
	Thus, using \eqref{yyy}, we get
	\begin{equation*}
		(q')^q  \leq e \left( 1 + \frac{1+ \delta}{\delta (t - 1)}\right) = e \left(\frac{\delta t + 1}{\delta (t - 1)} \right) < e \frac{t + 1}{\delta (t -1)}\lesssim\delta^{-1}
	\end{equation*}
	concluding the proof.
\end{proof}

\subsection{An open question on sharpness of estimates} Although we expect our estimates to be both qualitatively and quantitatively sharp, especially in the view of \cite{Matrix_A2}, we have not been able to prove that as of the moment of writing.

%%%%%%%%%%%%%%%%%%%%%%%%%%%%%%%%%% section %%%%%%%%%%%%%%%%%%%%%%%%%%%%%%%%%%%%%%%
%%%%%%%

\section{A hierarchy of two-exponent matrix Muckenhoupt conditions}
\label{section:matrix-weight}

In this section we want to understand the relationship of the various two-exponent matrix Muckenhoupt characteristics $[W]_{A_{p,q}}$ for $1<p\leq q<\infty$ with $\frac{q}{p'}+1=r$ for some \emph{fixed} $r$. In the scalar case, all these characteristics coincide. However, in the matrix case, they turn out to be quite different from each other, and only one-sided inclusions remain in general true. We provide explicit, fully constructive counterexamples demonstrating that. As a result, we deduce that the matrix weighted bound of \Cref{thm:matrix_weighted_bound} improves the one obtained earlier in \cite{Laukkarinen_2023}. Our methods are very flexible and allow us to consider even several endpoint cases that play an important role in the limited range extrapolation in \Cref{section:extrapolation} later.

Throughout this section, for the sake of clarity and simplicity we focus only on $\R^n$ equipped with the Euclidean metric and the $n$-dimensional Lebesgue measure. In particular, we simply write $\mathrm{d} x$ for integration under the Lebesgue measure, and $|E|$ for the $n$-dimensional Lebesgue measure of any Lebesgue measurable subset $E$ of $\R^n$.

\subsection{The members of the hierarchy}
\label{subsec:members_hierarchy}

Let us first recall the two-exponent matrix Muckenhoupt conditions from 
\Cref{dfn:muckenhoupt_weights}. In the case of $\R^n$, one usually takes the supremum in \Cref{dfn:muckenhoupt_weights} to range over cubes $Q\subseteq\R^n$ (with faces parallel to the coordinate hyperplanes) instead of open balls $B\subseteq\R^n$. Then, the resulting Muckenhoupt characteristics are equivalent to their ``ball'' versions up to constants depending only on $n$ and the exponents in question. The resulting equivalence classes of (scalar and matrix) weights coincide with their ``ball'' versions.

\begin{dfn}\label{dfn:muckenhoupt_weights_Rn}
Let $1<p\leq q<\infty$, and $W:\R^n\to \mathrm{M}_{d}(\C)$ be a measurable function with values a.e.~in the set of positive-definite $d \times d$ complex-valued matrices. We define
\begin{equation*}
    [W]_{A_{p,q}} := \sup_{Q} \dashint_Q \left(\dashint_Q |W(x)^{\frac{1}{q}} W(y)^{-\frac{1}{q}}|^{p'} \dd y \right)^{\frac{q}{p'}} \dd x < \infty,
\end{equation*}
where the supremum ranges over all cubes $Q\subseteq\R^n$ (with faces parallel to the coordinate hyperplanes). If $[W]_{A_{p,q}}<\infty$, then we say that $W$ is a \emph{$d\times d$ matrix $A_{p,q}$ weight}, and we write $W\in A_{p,q}$. If $p=q$, then we denote $A_{p,p}:=A_{p}$.
\end{dfn}

Furthermore, we are interested in the following endpoint cases.

\begin{dfn}\label{dfn:endpoint_infty_muckenhoupt_weights}
Let $W:\R^n\to\mathrm{M}_{d}(\C)$ be a measurable function with values a.e.~in the set of positive-definite $d \times d$ complex-valued matrices, and $1\leq t<\infty$. We define
\begin{equation*}
 [W]_{A_{t,\infty}}:=\sup_{B}\esssup_{x\in B}\dashint_{B}|W(x)^{1/t}W(y)^{-1/t}|^{t}\dd y,
\end{equation*}
where the supremum ranges over all cubes $Q\subseteq\R^n$ (with faces parallel to the coordinate hyperplanes). If $[W]_{A_{t,\infty}}<\infty$, then we say that $W$ is a \emph{$d\times d$ matrix $A_{t,\infty}$ weight}, and we write $W\in A_{t,\infty}$. Moreover, we denote $A_{1,\infty}:=A_{\infty}$.
\end{dfn}

Our definition of $A_{\infty}$ follows \cite{Cruz_Uribe_Bownik_Extrapolation}. However, our definitions of $A_{t,\infty}$ and $A_{\infty}$ deviate from the ones appearing in other places in the literature, even in the scalar case. See, for example, \cite{Hunt_Bellman_1996, Volberg_1997, Hytonen_Perez_Rela_2012, NPTV2017, Cruz_Uribe_OFS_2018, Isralowitz_Davey_2024}. We also refer to \cite{Duoandikoextea_Ombrosi_2016} for a detailed discussion of the relationship between the various definitions in the scalar case.

It is worth noting that if $W: \R^n\to \mathrm{P}_{d}(\C)$ is a measurable function with $W\in A_{t,\infty}$ for some $1\leq t<\infty$, then the function $|W|$ is locally bounded and the function $W^{-1}$ is locally integrable.

We turn our attention to reverse H{\"o}lder classes.

\begin{dfn}
If $w$ is a scalar weight on $\R^n$ and $1<s<\infty$, then we define
\begin{equation*}
    [w]_{RH_{s}}:=\sup_{Q}\frac{\aver{w^{s}}_{Q}^{1/s}}{\aver{w}_{Q}},
\end{equation*}
where the supremum ranges over all cubes $Q$ in $\R^n$ (with faces parallel to the coordinate hyperplanes). We say that $w$ belongs to the \emph{Reverse H{\"o}lder class of order $s$}, whenever $[w]_{RH_{s}}<\infty$, and we write $w\in RH_{s}$.

Moreover, following \cite{Auscher_Martell_I_2007}, for a scalar weight $w$ on $\R^n$ we define
\begin{equation*}
    [w]_{RH_{\infty}}:=\sup_{Q}\esssup_{x\in Q}\frac{w(x)}{\aver{w}_{Q}},
\end{equation*}
where the supremum ranges over all cubes $Q$ in $\R^n$ (with faces parallel to the coordinate hyperplanes). We say that $w$ belongs to the \emph{Reverse H{\"o}lder class of order $\infty$} whenever $[w]_{RH_{\infty}}<\infty$, and we write $w\in RH_{\infty}$.
\end{dfn}

\begin{dfn}
If $W$ is a $d\times d$ matrix weight on $\R^n$, $1<t<\infty$, and $1<s\leq\infty$, then following \cite{Laukkarinen_2023} we define
\begin{equation*}
    [W]_{RH_{t,s}}:=\sup_{e\in\C^d\setminus\{0\}}[|W^{1/t}e|^{t}]_{RH_{s}}<\infty.
\end{equation*}
We say that $W$ belongs to the \emph{Reverse H{\"o}lder class of exponent $t$ and order $s$} and write $W\in RH_{t,s}$ when $[W]_{RH_{t,s}}<\infty$.
\end{dfn}

\subsection{Situation in the scalar case}
\label{subsec:situation_scalar}

The various condition introduced in \Cref{subsec:members_hierarchy} are related to each other in the case of scalar weights in a very simple way. We briefly review how below.

Let $w$ be a scalar weight on $\R^n$. Then, for all $1\leq p\leq q<\infty$, we clearly have
\begin{equation}
\label{two_exponent_scalar}
    [w]_{A_{p,q}}=[w]_{A_r},\quad\text{where }r:=1+\frac{q}{p'}.
\end{equation}
It is also obvious that
\begin{equation}
\label{two_exponent_endpoint_scalar}
    [w]_{A_{t,\infty}}=[w]_{A_{\infty}},
\end{equation}
for all $1\leq t<\infty$.

Finally, let $1<s,t<\infty$ and set $r:=s(t-1)+1$. It is a well-known fact that
\begin{equation}
\label{analyze_two_exponent_scalar}
    w\in A_{t}\cap RH_{s} \iff w^{s}\in A_{r},
\end{equation}
see, for example, \cite{Johnson_Neugebauer_1991, Auscher_Martell_I_2007,BFP2016}.

\subsection{Commonalities with the scalar case}
\label{subsection:positive_results_on_weights}

In this subsection we show that at least parts of \eqref{two_exponent_scalar} and \eqref{two_exponent_endpoint_scalar}, as well as an appropriate form of \eqref{analyze_two_exponent_scalar} survive in the matrix case.

We begin with the correct analog of \eqref{analyze_two_exponent_scalar}.

\begin{lem}
\label{lemma:containmentsofmatrixclasses1}
	Let $W$ be a $d\times d$ matrix weight on $\R^n$, and $1<t,s<\infty$. Then, the following hold:
	\begin{enumerate}
		\item 
		\begin{equation*}
			\max([W]_{A_t},[W]_{RH_{t,s}})^{s}\lesssim_{d,t,s}[W^{s}]_{A_{t,st}}\lesssim_{d,t,s}([W]_{A_{t}}[W]_{RH_{t,s}})^{s}.
		\end{equation*}
		In particular, $W\in A_t\cap RH_{t,s}$ if and only if $W^{s}\in A_{t,st}$.
		
		\item 
		\begin{equation*}
			\max([W]_{A_{t}},[W]_{RH_{t,\infty}})^{t'/t}\lesssim_{d,t}[W^{t'/t}]_{A_{t',\infty}}\lesssim_{d,t}([W]_{A_{t}}[W]_{RH_{t,\infty}})^{t'/t}.
		\end{equation*}
		In particular, $W\in A_t\cap RH_{t,\infty}$ if and only if $W^{t'/t}\in A_{t',\infty}$.
	\end{enumerate}
\end{lem}

\begin{proof}
     For (1), set $V:=W^{s}$. Then, $V^{1/(st)}=W^{1/t}$.
     
     First, assume that $[W^{s}]_{A_{t,st}}<\infty$. For all cubes $Q\subseteq\R^n$, since $s\geq1$, Jensen's inequality yields
    \begin{equation*}
        \dashint_{Q}\left(\dashint_{Q}|W(x)^{1/t}W(y)^{-1/t}|^{t'}\dd y\right)^{t/t'}\dd x \leq \left(\dashint_{Q}\left(\dashint_{Q}|W(x)^{1/t}W(y)^{-1/t}|^{t'}\dd y\right)^{st/t'}\dd x\right)^{1/s},
    \end{equation*}
    which directly implies that $[W]_{A_{t}}\leq[V]_{A_{t,st}}^{1/s}$.
    
    Let us now show that $[W]_{RH_{t,s}}\lesssim_{d,t,s}[V]_{A_{t,s't}}^{1/s}$. Let $Q\subseteq\R^n$ be any cube and let $e\in\C^d\setminus\{0\}$. Since $[W]_{A_{t}}<\infty$, $W^{-1/(t-1)}$ is integrable over $Q$. Therefore, repeatedly applying \eqref{matrix norm as sum of norms of columns} and \eqref{break power of sum}, we can estimate
    \begin{equation*}
        \dashint_{Q}|W(x)^{1/t}\cW_{Q,t}'|^{st}\dd x\simeq_{d,t,s}
        \dashint_{Q}\left(\dashint_{Q}|W(x)^{1/t}W(y)^{-1/t}|^{t'}\dd y\right)^{st/t'}\mathrm{d}x\leq[V]_{A_{t,st}}.
    \end{equation*}
    It follows that
    \begin{equation*}
        \dashint_{Q}|W(x)^{1/t}e|^{st}\dd x=\dashint_{Q}|W(x)^{1/t}\cW_{Q,t}'(\cW_{Q,t}')^{-1}e|^{st}\dd x\lesssim_{d,t,s}[V]_{A_{t,st}}|(\cW_{Q,t}')^{-1}e|^{st}.
    \end{equation*}
    It is well-known (see e.g. \cite{Hunt_Bellman_1996, Lauzon_Treil_2007, DKPS2024}) that $ |(\cW_{Q,t}')^{-1}e|\leq|\cW_{Q,t}e|$.
    From this, we get that
    \begin{equation*}
        \dashint_{Q}|W(x)^{1/t}e|^{st}\dd x\lesssim_{d,t,s}[V]_{A_{t,st}}|\cW_{Q,t}e|^{st}\simeq_{d,t,s}
        [V]_{A_{t,st}}\left(\dashint_{Q}|W(x)^{1/t}e|^{t}\dd x\right)^{s}.
    \end{equation*}
    Consequently, $[W]_{RH_{t,s}}\lesssim_{d,t,s}[V]_{A_{t,s't}}^{1/s}$.

    Conversely, assume that $W\in A_{t}\cap RH_{t,st}$. Then, repeatedly applying \eqref{matrix norm as sum of norms of columns} and \eqref{break power of sum}, we can estimate
    \begin{align*}
        \dashint_{Q}\left(\dashint_{Q}|W(x)^{1/t}W(y)^{-1/t}|^{t'}\dd y\right)^{st/t'}\dd x
        &	\simeq_{d,t,s}\dashint_{Q}|W(x)^{1/t}\cW_{Q,t}'|^{st}\dd x	\\
		&	\lesssim_{d,t,s}[W]_{RH_{t,s}}^{s}\left(\dashint_{Q}|W(x)^{1/t}\cW_{Q,t}'|^{t}\dd x\right)^{s}	\\
        &	\simeq_{d,t,s}[W]_{RH_{t,s}}^{s}|\cW_{Q,t}\cW_{Q,t}'|^{ts}	\\
		&	\lesssim_{d,t,s}[W]_{RH_{t,s}}^{s}[W]_{A_{t}}^{s},
    \end{align*}
    for all cubes $Q\subseteq\R^n$. It follows that $[V]_{A_{t,st}}\lesssim_{d,t,s} [W]_{RH_{t,s}}^{s}[W]_{A_{t}}^{s}$.

    We omit the proof for (2) as it is very similar to the proof for part (1).
\end{proof}

Next, we treat the parts of \eqref{two_exponent_scalar}, \eqref{two_exponent_endpoint_scalar} that survive in the scalar case.

\begin{lem}
\label{lemma:containments_of_matrix_classes_2}
	Let $W$ be a $d\times d$ matrix weight on $\R^n$.
	\begin{enumerate}
		\item Let $1<p_1,q_1,p_2,q_2<\infty$ with $p_i\leq q_i$ (for $i=1,2$), $p_1\leq p_2$ and $\frac{q_2}{p_2'}=\frac{q_1}{p_1'}$. Then,
		\begin{equation*}
			[W]_{A_{p_1,q_1}}\leq[W]_{A_{p_2,q_2}}.
		\end{equation*}
		
		\item Let $1\leq s\leq t<\infty$. Then,
		\begin{equation*}
			[W]_{A_{t,\infty}}\leq[W]_{A_{s,\infty}}.
		\end{equation*}
	\end{enumerate}
\end{lem}

In the proof, we apply the following inequality due to Cordes \cite{Cordes_book}.

\begin{lem}[Cordes Inequality]
    Let $A,B\in\mathrm{P}_{d}(\C)$ and $0\leq a\leq1$. Then, it holds that
    \begin{equation*}
        |A^{a}B^{a}|\leq|AB|^{a}.
    \end{equation*}
\end{lem}

\begin{proof}[Proof (of \Cref{lemma:containments_of_matrix_classes_2})]
	For (1), first observe that
	\begin{equation*}
		0<\frac{q_2}{q_1}=\frac{p_2'}{p_1'}\leq 1.
	\end{equation*}
	Therefore, by Cordes inequality, we estimate
	\begin{align*}
		\dashint_{Q}\left(\dashint_{Q}|W(x)^{1/q_1}W(y)^{-1/q_1}|^{p_1'}\dd y\right)^{q_1/p_1'}\dd x
		&	= \dashint_{Q}\left(\dashint_{Q}|(W(x)^{1/q_2})^{q_2/q_1}(W(y)^{-1/q_2})^{q_2/q_1}|^{q_1'}\dd y\right)^{q_1/p_1'}\dd x	\\
		&	\leq\dashint_{Q}\left(\dashint_{Q}|W(x)^{1/q_2}W(y)^{-1/q_2}|^{q_2p_1'/q_1}\dd y\right)^{q_1/p_1'}\dd x	\\
		&	=\dashint_{Q}\left(\dashint_{Q}|W(x)^{1/q_2}W(y)^{1/q_2}|^{p_2'}\dd y\right)^{q_2/p_2'}\dd x,
	\end{align*}
	for all cubes $Q\subseteq\R^n$. The desired result follows.
	
	The proof for (2) is analogous to the proof of part (1).
\end{proof}

\subsection{Differences from the scalar case} \label{subsection:negative_results_on_weights}

Despite the positive results in Subsection \ref{subsection:positive_results_on_weights}, \eqref{two_exponent_scalar} and \eqref{two_exponent_endpoint_scalar} can actually fail in the matrix case. The following two propositions demonstrate that. Let
\begin{equation*}
    \cD:=\{[m2^{k},(m+1)2^{k}):~m,k\in\Z\}
\end{equation*}
be the family of dyadic intervals in $\R$. For $1<p\leq q<\infty$, we denote
\begin{equation*}
[W]_{A_{p,q},\cD}:=\sup_{I\in\cD}\dashint_{I}\left(\dashint_{I}|W(x)^{1/q}W(y)^{-1/q}|^{p'}\dd y\right)^{q/p'}\dd x.
\end{equation*}
Similarly to \Cref{dfn:endpoint_infty_muckenhoupt_weights} (with the cubes replaced by the dyadic intervals of $\R$), we define $[W]_{A_{t,\infty},\cD}$ for $1\leq t<\infty$.

\begin{prop}
\label{prop:weightscounterexample}

\begin{enumerate}
    \item Let $1<p_1\leq q_1<\infty$. Then, there exists a $2\times 2$ matrix weight $W$ on $\R$ with $[W]_{A_{p_1,q_1}}<\infty$ such that for all $1<p_2\leq q_2<\infty$ with $p_1<p_2$ and $\frac{q_2}{p_2'}=\frac{q_1}{p_1'}$ we have $[W]_{A_{p_2,q_2},\cD}=\infty$.

    \item Let $1<t<\infty$. Then, there exists a $2\times 2$ matrix weight $W$ on $\R$ with $[W]_{A_{t,\infty}}<\infty$ such that $[W]_{A_{s,\infty},\cD}=\infty$ for all $1\leq s<t$.

\end{enumerate}
\end{prop}

\begin{prop}
\label{prop:weights_counterexample_2}
	\begin{enumerate}
		\item Let $1<p_2\leq q_2<\infty$. Then, there exists a $2\times 2$ matrix weight $W$ on $\R$ with $[W]_{A_{p_2,q_2},\cD}=\infty$ such that $[W]_{A_{p_1,q_1}}<\infty$, for all $1<p_1\leq q_1<\infty$ with $\frac{q_1'}{p_1}=\frac{q_2'}{p_2}$ and $p_1<p_2$.
		
		\item Let $1\leq s<\infty$. Then, there exists a $2\times 2$ matrix weight $W$ on $\R$ with $[W]_{A_{s,\infty},\cD}=\infty$ such that $[W]_{A_{t,\infty}}<\infty$, for all $s<t<\infty$.
	\end{enumerate}
\end{prop}

Before proving \Cref{prop:weightscounterexample} and \Cref{prop:weights_counterexample_2}, let us briefly examine their implications for the relationship of the matrix weighted bounds of \Cref{thm:matrix_weighted_bound} to those obtained earlier by \cite{Laukkarinen_2023}. We first recall the notation of \Cref{thm:matrix_weighted_bound}: we consider $1\leq p_0<q_0\leq\infty$, $p\in(p_0,q_0)$ and let
    \begin{equation*}
        t:=\frac{p}{p_0},\quad s:=\frac{q_0}{p},\quad a:=\frac{p}{t-1},\quad b:=s'p.
    \end{equation*}
    
    In this notation, the matrix-weighted bounds in \cite{Laukkarinen_2023} are valid for matrix weights $W\in A_{t}\cap RH_{t,s'}$. By \Cref{lemma:containmentsofmatrixclasses1}, this condition is equivalent to $W^{s'}\in A_{t,s't}$. Observe that $t=\frac{p}{p_0}\leq p$. Therefore
\begin{equation*}
    a=\frac{p}{t-1}\geq\frac{t}{t-1}=t',
\end{equation*}
and consequently $a'\leq t$. Since $\frac{b}{a}=\frac{s't}{t'}$, by \Cref{lemma:containments_of_matrix_classes_2} we deduce that the condition $W^{s'}\in A_{a',b}$ featuring in \Cref{thm:matrix_weighted_bound} is weaker than the condition $W\in A_{t}\cap RH_{t,s'}$. Observe also that $a'=t$ if and only if $p_0=1$. Therefore, if $p_0>1$ and $d\geq2$, then using \Cref{prop:weightscounterexample} one can find a $d\times d$ matrix weight $W$ on $\R^n$ with $W^{s'}\in A_{a',b}$ but $W\notin A_{t}\cap RH_{t,s'}$.

Let us now prove \Cref{prop:weightscounterexample} and \Cref{prop:weights_counterexample_2}. The core of both proofs boils down to constructing appropriate families of counterexamples to the Cordes inequality. One such family is given in the lemma below.

\begin{lem}
    \label{first lemma}
    Set
    \begin{equation*}
        C_n:=
        \begin{bmatrix}
            1&0\\
            0&n
        \end{bmatrix}
        ,\quad
        D_n:=
        \begin{bmatrix}
            1&\frac{1}{2n^{1/2}}\\
            \frac{1}{2n^{1/2}}&\frac{1}{n}
        \end{bmatrix}
        ,\quad n=1,2,\ldots.
    \end{equation*}
    Then, $C_n,D_n\in\mathrm{P}_{2}(\C)$ with
    \begin{equation*}
        \tr(C_nD_n),~\tr(C_n^{-1}D_n^{-1})\simeq 1,
    \end{equation*}
    for all $n=1,2,\ldots$. Moreover, for all $a>1$,
    \begin{equation*}
        \lim_{n\to\infty}\tr(C_n^{a}D_n^{a})=\infty.
    \end{equation*}
\end{lem}

 We apply \Cref{first lemma} to prove \Cref{lemma:weightscounterexample}. For completeness, we give a proof of \Cref{first lemma} in \Cref{section:appendix}, together with a second family of counterexamples to the converse of the Cordes inequality. The latter family allows one to prove in a similar way \Cref{prop:weights_counterexample_2}. Thus, we concentrate on proving \Cref{prop:weightscounterexample}. We only prove part (1) of \Cref{prop:weightscounterexample}; it is clear that part (2) can be proved similarly. 

Going forward, by ``interval'' we always mean a left-closed, right-open, non-trivial bounded interval in $\R$. Before proving part (1) of \Cref{prop:weightscounterexample}, we construct an analogous counterexample on the unit interval $I_0:=[0,1)$. A \emph{weight} $W$ on $I_0$ is an integrable function $W:I_0\to\mathrm{M}_{d}(\C)$ with $W(x)\in\mathrm{P}_{d}(\C)$ for a.e.~$x\in I_0$. Moreover, for $1<p\leq q<\infty$ we denote
\begin{equation*}
[W]_{A_{p,q}(I_0)}:=\sup_{I\subseteq I_0}\dashint_{I}\left(\dashint_{I}|W(x)^{1/q}W(y)^{-1/q}|^{p'}\dd y\right)^{q/p'}\dd x,
\end{equation*}
and we define $[W]_{A_{p,q},\cD(I_0)}$ analogously.

\begin{lem}
\label{lemma:weightscounterexample}
    Let $1<p_1\leq q_1<\infty$. Then, there exists a $2\times 2$ matrix weight $W$ on $I_0$ with $[W]_{A_{p_1,q_1}(I_0)}<\infty$ such that for all $1<p_2\leq q_2<\infty$ with $p_1<p_2$ and $\frac{q_2}{p_2'}=\frac{q_1}{p_1'}$, we have $[W]_{A_{p_2,q_2},\cD(I_0)}=\infty$.
\end{lem}

\begin{proof}
    Throughout the proof, we suppress from the notation dependence of implied constants on $p_1$ and~$q_1$. 
    
    For all positive integers $n$, consider the matrices $C_n,D_n$ from \Cref{first lemma} and set
    \begin{equation*}
        A_n:=C_n^{-q_1/2},\quad B_n:=D_n^{q_1/2}.
    \end{equation*}
    Notice that
    for all $1<p_2\leq q_2<\infty$ with $p_1<p_2$ and $\frac{q_2}{p_2'}=\frac{q_1}{p_1'}$ we have
    \begin{equation*}
        |A_n^{-1/q_2}B_n^{1/q_2}|=|C_n^{q_1/(2q_2)}D_n^{q_1/(2q_2)}|\simeq\tr(C_n^{q_1/q_2}D_n^{q_1/q_2})^{1/2},
    \end{equation*}
    for all $n=1,2,\ldots$, therefore since $\frac{q_1}{q_2}=\frac{p_1'}{p_2'}>1$, \Cref{first lemma} implies $\lim_{n\to\infty}|A_n^{-1/q_2}B_n^{1/q_2}|=\infty$.

    Set now
    \begin{equation*}
        W:=\sum_{n=0}^{\infty}(\1_{(J_n)_{-}} A_{n+1}+\1_{(J_n)_{+}} B_{n+1}),
    \end{equation*}
    where
    \begin{equation*}
    J_n:=\left[\frac{1}{2^{n+1}},\frac{1}{2^n}\right),\quad n=0,1,2,\ldots,
    \end{equation*}
    and for all intervals $I=[a,b)\subseteq\R$ we denote
    \begin{equation*}
        I_{-}:=\left[a,\frac{a+b}{2}\right),\quad I_{+}:=\left[\frac{a+b}{2},b\right).
    \end{equation*}
    Notice that for all $1<p_2\leq q_2<\infty$ with $p_1<p_2$ and $\frac{q_2}{p_2'}=\frac{q_1}{p_1'}$ we have
    \begin{align*}
        [W]_{A_{p_2,q_2}(I_0)}&\geq\dashint_{J_n}\left(\dashint_{J_n}|W(x)^{1/q_2}W(y)^{-1/q_2}|^{p_2'}\dd y\right)^{q_2/p_2'}\dd x\\
        &\gtrsim_{q_2,p_2}\dashint_{(J_n)_{+}}\left(\dashint_{(J_n)_{-}}|W(x)^{1/q_2}W(y)^{-1/q_2}|^{p_2'}\dd y\right)^{q_2/p_2'}\dd x\\
        &=|B_{n+1}^{1/q_2}A_{n+1}^{-1/q_2}|^{q_2}=|A_{n+1}^{-1/q_2}B_{n+1}^{1/q_2}|^{q_2},
    \end{align*}
    for all $n=0,1,2,\ldots$, therefore $[W]_{A_{p_2,q_2},\cD(I_0)}=\infty$.

    It remains now to show that $[W]_{A_{p_1,q_1}(I_0)}<\infty$. We show a slightly stronger property, that is:
    \begin{equation*}
        \esssup_{x\in I}\dashint_{I}|W(x)^{1/q_1}W(y)^{-1/q_1}|^{p_1'}\dd y\lesssim1,
    \end{equation*}
    for all subintervals $I$ of $I_0$.
    
    Let us first show that for all $k=0,1,2,\ldots$ we have
   \begin{equation*}
        \esssup_{x\in I_k}\dashint_{I_k}|W(x)^{1/q_1}W(y)^{-1/q_1}|^{p_1'}\dd y\lesssim1,
    \end{equation*}
    where
    \begin{equation*}
    I_k:=\left[0,\frac{1}{2^k}\right),\quad k=0,1,2\ldots.
    \end{equation*}

    Let $k$ be any nonnegative integer. Observe that $I_k=\bigcup_{i=k}^{\infty}J_k$. Thus, it suffices to prove that
    \begin{equation*}
        \esssup_{x\in J_i}\dashint_{I_k}|W(x)^{1/q_1}W(y)^{-1/q_1}|^{p_1'}\dd y\lesssim1,\quad\forall i=k,k+1,\ldots.
    \end{equation*}
    We show the following more general estimate, which is later useful for the proof of \Cref{prop:weightscounterexample}:
    \begin{equation*}
        \esssup_{x\in J_i}\dashint_{I_k}|W(x)^{1/q_1}W(y)^{-1/q_1}|^{p_1'}\dd y\lesssim\max\left(1,\frac{k+1}{i+1}\right)^{p_1'/2},\quad\forall i=0,1,2,\ldots.
    \end{equation*}
    Let $i$ be any nonnegative integer. For all $x\in (J_i)_{-}$, we compute
    \begin{align*}
        \dashint_{I_k}|W(x)^{1/q_1}W(y)^{-1/q_1}|^{p_1'}\dd y
        &	= \sum_{j=k}^{\infty}2^{k-j}\dashint_{J_j}|W(x)^{1/q_1}W(y)^{-1/q_1}|^{p_1'}\dd y	\\
        	\begin{split}
		    	& \simeq \sum_{j=k}^{\infty}2^{k-j}\dashint_{(J_j)_{-}}|W(x)^{1/q_1}W(y)^{-1/q_1}|^{p_1'}\dd y	\\
				&\qquad\qquad + \sum_{j=k}^{\infty}2^{k-j}\dashint_{(J_j)_{+}}|W(x)^{1/q_1}W(y)^{-1/q_1}|^{p_1'}\dd y
        	\end{split}	\\
        &	= \sum_{j=k}^{\infty}2^{k-j}|A_{i+1}^{1/q_1}A_{j+1}^{-1/q_1}|^{p_1'}
        	+ \sum_{j=k}^{\infty}2^{k-j}|A_{i+1}^{1/q_1}B_{j+1}^{-1/q_1}|^{p_1'}	\\
        &	\simeq \sum_{j=k}^{\infty}2^{k-j}\tr(C_{i+1}^{-1}C_{j+1})^{p_1'/2}
        	+ \sum_{j=k}^{\infty}2^{k-j}\tr(C_{i+1}^{-1}D_{j+1}^{-1})^{p_1'/2}.
    \end{align*}
    A similar computation yields
    \begin{equation*}
        \dashint_{I_k}|W(x)^{1/q_1}W(y)^{-1/q_1}|^{p_1'}\dd y
        \simeq \sum_{j=k}^{\infty}2^{k-j}\tr(D_{i+1}C_{j+1})^{p_1'/2}
        + \sum_{j=k}^{\infty}2^{k-j}\tr(D_{i+1}D_{j+1}^{-1})^{p_1'/2},
    \end{equation*}
    for all $x\in (J_i)_{+}$. Direct computation gives
    \begin{gather*}
        \tr(C_{m}^{-1}C_n)\simeq\max\left(1,\frac{n}{m}\right),\quad  \tr(D_{m}C_n)\simeq\max\left(1,\frac{n}{m}\right),	\\
		\tr(C_m^{-1}D_{n}^{-1})\simeq\max\left(1,\frac{n}{m}\right), \quad \tr(D_{m}D_n^{-1})\simeq\max\left(1,\frac{n}{m}\right)
    \end{gather*}
    for all $m,n=1,2,\ldots$. We distinguish two cases:
    
    \begin{enumerate}[label=\bf Case \arabic*.]
		\item $i\geq k$. Then, for all $x\in J_i$, we have
		\begin{align*}
			\dashint_{I_k}|W(x)^{1/q_1}W(y)^{-1/q_1}|^{p_1'}\dd y
			&	\simeq \sum_{j=k}^{i}2^{k-j}+\sum_{j=i+1}^{\infty}2^{k-j}\left(\frac{j+1}{i+1}\right)^{p_1'/2}	\\
			&	\lesssim1+\sum_{j=k}^{\infty}2^{k-j}\left(\frac{j+1}{k+1}\right)^{p_1'/2}
				= 1 + \sum_{j=k}^{\infty}2^{k-j}\left(\frac{j-k}{k+1}+1\right)^{p_1'/2}	\\
			&	\lesssim 1 + \sum_{j=k}^{\infty}2^{k-j}((j-k)^{p_1'/2}+1)
				\lesssim 1.
		\end{align*}
		
		\item $i<k$. Then, for all $x\in J_i$, we have
		\begin{align*}
				\dashint_{I_k}|W(x)^{1/q_1}W(y)^{-1/q_1}|^{p_1'}\dd y
			&	\simeq \sum_{j=k}^{\infty}2^{k-j}\left(\frac{j+1}{i+1}\right)^{p_1'/2}	\\
			&	\lesssim \frac{1}{(i+1)^{p_1'/2}}\sum_{j=k}^{\infty}2^{k-j}((j-k)^{p_1'/2}+(k+1)^{p_1'/2})	\\
			&	\lesssim \left(\frac{k+1}{i+1}\right)^{p_1'/2}.
		\end{align*}
    \end{enumerate}
    
    Let now $I$ by any subinterval of $I_0$. Let $k$ be the largest nonnegative integer with $I\subseteq I_k$. If $|I|\geq\frac{|I_k|}{512}$, then clearly
    \begin{equation*}
        \esssup_{x\in I}\dashint_{I}|W(x)^{1/q_1}W(y)^{-1/q_1}|^{p_1'}\dd y\lesssim\esssup_{x\in I_{k}}\dashint_{I_k}|W(x)^{1/q_1}W(y)^{-1/q_1}|^{p_1'}\dd y\lesssim1.
    \end{equation*}
    Assume now that $|I|<\frac{|I_k|}{512}$. Then, at least one of the two must hold:
    \begin{itemize}[leftmargin=1.5\parindent]
        \item $I\subseteq J_k=(J_{k})_{+}\cup(J_{k})_{-}$
        \item $I\subseteq (J_{k+1})_{+}\cup(J_{k})_{-}$.
    \end{itemize}
    And since
    \begin{equation*}
        \tr(C_{i}D_{j}), ~ \tr(C_{i}^{-1}D_{j}^{-1}), ~ \tr(C_{i}C_{j}^{-1}), ~ \tr(D_iD_{j}^{-1}) \simeq 1
    \end{equation*}
    for all $i,j\in\{ k+1,k+2\}$, it easily follows that
    \begin{equation*}
        \esssup_{x\in I}\dashint_{I}|W(x)^{1/q_1}W(y)^{-1/q_1}|^{p_1'}\dd y\lesssim1,
    \end{equation*}
    concluding the proof.
\end{proof}

To extend the counterexample of \Cref{lemma:weightscounterexample} to the whole real line, we use a standard trick that was also used in \cite[Lemma 2.2]{Kakaroumpas_2019}.

\begin{proof}[Proof (of part (1) of \Cref{prop:weightscounterexample})]
     Throughout the proof, we suppress from the notation dependence of implied constants on $p_1$ and $q_1$.
     
     Let $W$ be the $2\times 2$ matrix weight on $I_0$ from the proof of \Cref{lemma:weightscounterexample}. Consider the function $\widetilde{W}:\R\to\mathrm{M}_{2}(\R)$ given by
    \begin{equation*}
        \widetilde{W}(x):=w(x-k),\quad x\in(k,k+1),\quad\text{if }k\text{ is even},
    \end{equation*}
    \begin{equation*}
        \widetilde{W}(x):=w(k+1-x),\quad x\in(k,k+1),\quad\text{if }k\text{ is odd},
    \end{equation*}
    \begin{equation*}
        \widetilde{W}(k):=I\in M_2(\R),\quad k\in\Z.
    \end{equation*}
    It is obvious that $[\widetilde{W}]_{A_{p_2,q_2},\cD}=\infty$, for all $1<p_2\leq q_2<\infty$ with $p_1<p_2$ and $\frac{q_2}{p_2'}=\frac{q_1}{p_1'}$. It remains to show that $[\widetilde{W}]_{A_{p_1,q_1}}<\infty$. We prove a slightly stronger fact, namely:
    \begin{equation*}
        \esssup_{x\in I}\dashint_{I}|\widetilde{W}(x)^{1/q_1}\widetilde{W}(y)^{-1/q_1}|^{p_1'}\dd y\lesssim1,
    \end{equation*}
    for all intervals $I\subseteq\R$.

    First of all, let $m,n\in\Z$ with $m<n$ be arbitrary. We check that
    \begin{equation*}
        \esssup_{x\in[m,n)}\dashint_{[m,n)}|\widetilde{W}(x)^{1/q_1}\widetilde{W}(y)^{-1/q_1}|^{p_1'}\dd y\lesssim1.
    \end{equation*}
    Since $[m,n)=\bigcup_{k=m}^{n-1}[k,k+1)$, we may compute
    \begin{align*}
        \esssup_{x\in[m,n)}\dashint_{[m,n)}|\widetilde{W}(x)^{1/q_1}\widetilde{W}(y)^{-1/q_1}|^{p_1'}\dd y
        &	= \esssup_{x\in I_0}\dashint_{[m,n)}|W(x)^{1/q_1}\widetilde{W}(y)^{-1/q_1}|^{p_1'}\dd y	\\
        &	= \esssup_{x\in I_0}\sum_{k=m}^{n-1} \frac{1}{n-m} \dashint_{[k,k+1)}|W(x)^{1/q_1}\widetilde{W}(y)^{-1/q_1}|^{p_1'}\dd y	\\
        &	= \esssup_{x\in I_0}\sum_{k=m}^{n-1}\frac{1}{n-m}\dashint_{I_0}|W(x)^{1/q_1}W(y)^{-1/q_1}|^{p_1'}\dd y	\\
        &	= \esssup_{x\in I_0}\dashint_{I_0}|W(x)^{1/q_1}W(y)^{-1/q_1}|^{p_1'}\dd y
        \lesssim1.
    \end{align*}
    Moreover, let $k,\ell$ be arbitrary non-negative integers. We check that
    \begin{equation*}
        \esssup_{x\in\left[-\frac{1}{2^{k}},\frac{1}{2^{\ell}}\right)}\dashint_{\left[-\frac{1}{2^{k}},\frac{1}{2^{\ell}}\right)}|\widetilde{W}(x)^{1/q_1}\widetilde{W}(y)^{-1/q_1}|^{p_1'}\dd y\lesssim1.
    \end{equation*}
    First of all, we have
	\begin{align*}
		\esssup_{x\in\left[-\frac{1}{2^{k}},0\right)} \dashint_{\left[-\frac{1}{2^{k}},\frac{1}{2^{\ell}}\right)} & |\widetilde{W}(x)^{1/q_1}\widetilde{W}(y)^{-1/q_1}|^{p_1'}\dd y	\\
		\begin{split}
			&	\leq \esssup_{x\in\left[-\frac{1}{2^{k}},0\right)} \frac{2^{-k}}{2^{-k}+2^{-\ell}}\dashint_{\left[-\frac{1}{2^{k}},0\right)}|\widetilde{W}(x)^{1/q_1}\widetilde{W}(y)^{-1/q_1}|^{p_1'}\dd y	\\
			&\qquad\qquad	+ \esssup_{x\in\left[-\frac{1}{2^{k}},0\right)} \frac{2^{-\ell}}{2^{-k}+2^{-\ell}}\dashint_{\left[0,\frac{1}{2^{\ell}}\right)}|\widetilde{W}(x)^{1/q_1}\widetilde{W}(y)^{-1/q_1}|^{p_1'}\dd y
		\end{split}	\\
		\begin{split}
			&	= \esssup_{x\in I_k}\frac{2^{\ell}}{2^{k}+2^{\ell}} \dashint_{I_k}|W(x)^{1/q_1}W(y)^{-1/q_1}|^{p_1'}\dd y	\\
			&\qquad\qquad + \esssup_{x\in I_k}\frac{2^{k}}{2^{k}+2^{\ell}}\dashint_{I_{\ell}}|W(x)^{1/q_1}W(y)^{-1/q_1}|^{p_1'}\dd y.
		\end{split}
	\end{align*}
    Since $I_k=\bigcup_{i=k}^{\infty}J_{i}$, by the proof of \Cref{lemma:weightscounterexample} we deduce
    \begin{equation*}
        \esssup_{x\in I_k}\dashint_{I_k}|W(x)^{1/q_1}W(y)^{-1/q_1}|^{p_1'}\dd y \lesssim 1
    \end{equation*}
    and
    \begin{equation*}
        \esssup_{x\in I_k}\dashint_{I_{\ell}}|W(x)^{1/q_1}W(y)^{-1/q_1}|^{p_1'}\dd y\lesssim\max\left(1,\frac{\ell+1}{k+1}\right)^{p_1'/2}.
    \end{equation*}
    It follows that
    \begin{align*}
    &\esssup_{x\in\left[-\frac{1}{2^{k}},0\right)}\dashint_{\left[-\frac{1}{2^{k}},\frac{1}{2^{\ell}}\right)}|\widetilde{W}(x)^{1/q_1}\widetilde{W}(y)^{-1/q_1}|^{p_1'}\dd y\lesssim
    1+\frac{2^{k}}{2^{k}+2^{\ell}}\max\left(1,\frac{\ell+1}{k+1}\right)^{p_1'/2}.
    \end{align*}
    Analogously, we obtain
    \begin{align*}
    &\esssup_{x\in\left[0,\frac{1}{2^{\ell}}\right)}\dashint_{\left[-\frac{1}{2^{k}},\frac{1}{2^{\ell}}\right)}|\widetilde{W}(x)^{1/q_1}\widetilde{W}(y)^{-1/q_1}|^{p_1'}\dd y\lesssim
    1+\frac{2^{\ell}}{2^{k}+2^{\ell}}\max\left(1,\frac{k+1}{\ell+1}\right)^{p_1'/2}.
    \end{align*}
    If $k<\ell$, then
    \begin{align*}
        \frac{2^{k}}{2^{k}+2^{\ell}}\left(\frac{\ell+1}{k+1}\right)^{p_1'/2}\leq 2^{k-\ell}\left(\frac{\ell-k}{k+1}+1\right)^{p_1'/2}\lesssim 2^{k-\ell}(\ell-k)^{p_1'/2}\lesssim1.
    \end{align*}
    The case $\ell<k$ is symmetric, thus we deduce the desired estimate.
    
    Let now $I$ be any interval in $\R$. We distinguish the following cases:
	
	\begin{enumerate}[label=\bf Case \arabic*.]
    \item There exists $m\in\Z$ with $I\subseteq[m,m+1)$. It is then immediate by \Cref{lemma:weightscounterexample} that
    \begin{align*}
        \esssup_{x\in I}\dashint_{I}|\widetilde{W}(x)^{1/q_1}\widetilde{W}(y)^{-1/q_1}|^{p_1'}\dd y\lesssim1.
    \end{align*}

    \item There does not exist $m\in\Z$ with $I\subseteq[m,m+1)$, and $|I|\geq\frac{1}{2}$. Then, there exist $m,n\in\Z$ with $m+1<n$ such that $I\subseteq[m,n)$ and $I\nsubseteq[m+1,n)$, $I\nsubseteq[m,n-1)$. It is clear that $|I|\simeq|[m,n)|$, thus
    \begin{align*}
       \esssup_{x\in I}\dashint_{I}|\widetilde{W}(x)^{1/q_1}\widetilde{W}(y)^{-1/q_1}|^{p_1'}\dd y
        \lesssim \esssup_{x\in[m,n)}\dashint_{[m,n)}|\widetilde{W}(x)^{1/q_1}\widetilde{W}(y)^{-1/q_1}|^{p_1'}\dd y\lesssim1.
    \end{align*}

    \item There does not exist $m\in\Z$ with $I\subseteq[m,m+1)$, and $|I|<\frac{1}{2}$. Then, there exists $m\in\Z$ with $I\subseteq[m,m+2)$ and $m+1\in\mathrm{Int}(I)$.

    If $m+1$ is odd, that is $m$ is even, then it is clear that
    \begin{align*}
        \widetilde{W}(x)\in \{A_1,B_1\},\quad\text{for a.e. }x\in I,
    \end{align*}
   therefore we trivially deduce
    \begin{align*}
        \esssup_{x\in I}\dashint_{I}|\widetilde{W}(x)^{1/q_1}\widetilde{W}(y)^{-1/q_1}|^{p_1'}\dd y\lesssim1.
    \end{align*}

    If $m+1$ is even, that is $m$ is odd, then it is easy to see that there exist positive integers $k,\ell$ with
    \begin{equation*}
        I\subseteq J:=\left[m+1-\frac{1}{2^{k}},m+1+\frac{1}{2^{\ell}}\right)
    \end{equation*}
    and $|I|\geq\frac{1}{2}|J|$, therefore
    \begin{align*}
        \esssup_{x\in I}\dashint_{I}|\widetilde{W}(x)^{1/q_1}\widetilde{W}(y)^{-1/q_1}|^{p_1'}\dd y
        &	\lesssim \esssup_{x\in J}\dashint_{J}|\widetilde{W}(x)^{1/q_1}\widetilde{W}(y)^{-1/q_1}|^{p_1'}\dd y	\\
        &	= \esssup_{x\in\left[-\frac{1}{2^{k}},\frac{1}{2^{\ell}}\right)} \dashint_{\left[-\frac{1}{2^{k}},\frac{1}{2^{\ell}}\right)}|\widetilde{W}(x)^{1/q_1}\widetilde{W}(y)^{-1/q_1}|^{p_1'}\dd y \lesssim 1,
    \end{align*}
    concluding the proof.
    \end{enumerate}
\end{proof}

%%%%%%%%%%%%%%%%%%%%%%%%%%%%%%%%%% section %%%%%%%%%%%%%%%%%%%%%%%%%%%%%%%%%%%%%%%
%%%%%%%
    
\section{Limited range extrapolation for matrix weights}
\label{section:extrapolation}

In this section, we establish a limited-range extrapolation theorem for matrix weights. Although the proof is a straightforward adaptation of \cite[Theorem 9.1]{Cruz_Uribe_Bownik_Extrapolation} guided by \cite[Theorem 4.9]{Auscher_Martell_I_2007}, we carefully track all details. We restrict ourselves to $\R^n$ equipped with the Euclidean metric and $n$-dimensional Lebesgue measure as underlying space. In particular, we simply write $\mathrm{d} x$ for integration under the Lebesgue measure, and $|E|$ for the $n$-dimensional Lebesgue measure of any Lebesgue-measurable subset $E$ of $\R^n$.

\subsection{Preliminaries}

Here we briefly recall some important definitions and facts from \cite{Cruz_Uribe_Bownik_Extrapolation}, where we refer for details and proofs.

\subsubsection{Integration of convex set valued maps}

We consider a complete $\sigma$-finite measure space $(\Omega,\cA,\mu)$.

\begin{dfn}
    Let $F:\Omega\to\cK_{\bcs}(\C^d)$ be an $\cA$-measurable map. We say that this map is \emph{$\mu$-integrably bounded}, if there exists a $\mu$-integrable function $k:\Omega\to[0,\infty)$ with $F(\omega)\subseteq k(\omega)\overline{\mathbf{B}}$, for all $\omega\in\Omega$.
\end{dfn}

Observe that an $\cA$-measurable map $F:\Omega\to\cK_{\bcs}(\C^d)$ is $\mu$-integrably bounded if and only if the function $|F|:\Omega\to[0,\infty)$ is $\mu$-integrable.

\begin{dfn}
    Let $F:\Omega\to\cK_{\bcs}(\C^d)$ be a $\mu$-integrably bounded map. Following \cite[Definition 3.11]{Cruz_Uribe_Bownik_Extrapolation}, we define the \emph{Aumann $\mu$-integral of $F$ over $\Omega$}, $\int_{\Omega}F(\omega)\dd\mu(\omega)$, as the following subset of $\C^d$:
    \begin{equation*}
        \int_{\Omega}F(\omega)\dd\mu(\omega) := \left\{\int_{\Omega}f(\omega)\dd\mu(\omega)\such f:\Omega\to\C^d\;\mu\text{-integrable selection function for }F\right\}.
    \end{equation*}
\end{dfn}

Let $F:\Omega\to\cK_{\bcs}(\C^d)$ be a $\mu$-integrably bounded map. In \cite[Section 3]{Cruz_Uribe_Bownik_Extrapolation} it is proved that the Aumann integral $\int_{X}F(\omega)\dd\mu(\omega)$ is a convex body in $\C^d$. For all $A\in\cA$, we denote
\begin{equation*}
    \int_{A}F(\omega)\dd\mu(\omega) := \int_{X}\1_{A}(\omega)F(\omega)\dd\mu(\omega) = \int_{A}F(\omega)\dd\mu|_{A}(\omega).
\end{equation*}

\begin{rem}
    \label{rem:Aumann_integral_nontrivial}
    Let $F:\Omega\to\cK_{\bcs}(\C^d)$ be a $\mu$-integrably bounded map. It is proved in \cite[Lemma 3.19]{Cruz_Uribe_Bownik_Extrapolation} that $\int_{X}F(\omega)\dd\mu(\omega)=\lbrace0\rbrace$ if and only if $F(\omega)=\lbrace0\rbrace$ for $\mu$-a.e. $\omega\in\Omega$.
\end{rem}

\subsubsection{Convex body valued maximal function}

\begin{dfn}
\begin{enumerate}
    
\item For any measurable function $F:\R^n\to\cK_{\mathrm{cs}}(\C^d)$ and for any $d\times d$ matrix weight $W$ on $\R^n$, we denote
\begin{equation*}
    \norm{F}_{L^{p}_{\cK}(W)}:=\left(\int_{\R^n}|W(x)^{1/p}F(x)|^{p}\dd x\right)^{1/p},\quad 0<p<\infty,
\end{equation*}
and we also denote
\begin{equation*}
    \norm{F}_{L^{\infty}_{\cK}(W)}:=\esssup_{x\in\R^n}|W(x)F(x)|.
\end{equation*}
Then, for any $0<p\leq\infty$, we denote by $L^{p}_{\cK}(W)$ the family of all (Lebesgue-a.e.~equivalence classes of) measurable functions $F:\R^n\to\cK_{\mathrm{cs}}(\C^d)$ with $\norm{F}_{L^{p}_{\cK}(W)}<\infty$. Observe that if $F\in L^{p}_{\cK}(W)$ for some $0<p\leq\infty$, then the subset $F(x)$ of $\C^d$ is bounded, for a.e.~$x\in\R^n$.

\item If $1\leq p,q\leq \infty$ and $W,V$ are $d\times d$ matrix weights on $\R^n$, then a map $T$ is said to be $L^{p}_{\cK}(W)$-$L^{q}_{\cK}(V)$ bounded, if it is a map $T:L^{p}_{\cK}(W)\to L^{q}_{\cK}(V)$ and there exists $0<C<\infty$ such that
\begin{equation*}
    \norm{Tf}_{L^{q}_{\cK}(V)}\leq C\norm{f}_{L^{p}_{\cK}(W)},
\end{equation*}
for all $f\in L^{p}_{\cK}(W)$; in this case, we denote by $\Vert T\Vert_{L^{p}_{\cK}(W)\to L^{q}_{\cK}(V)}$ the best such constant. Moreover, we denote $\Vert T\Vert_{L^{p}_{\cK}(W)}:=\Vert T\Vert_{L^{p}_{\cK}(W)\to L^{p}_{\cK}(W)}$.

\item A measurable map $F:\R^n\to\cK_{\mathrm{cs}}(\C^d)$ is said to be \emph{locally integrable} if the function $|F|:\R^n\to[0,\infty]$ is locally integrable.

\end{enumerate}
\end{dfn}

\begin{rem}
If $F\in L^{p}_{\cK}(W)$ for some $1<p<\infty$ and some $d\times d$ matrix weight $W$ such that $W^{1/(p-1)}$ is also a $d\times d$ matrix weight, then it is easy to see that $F$ is locally integrable on $\R^n$.
\end{rem}

\begin{dfn}
\begin{enumerate}
\item A map $T$ acting from a space of $\cK_{\mathrm{cs}}(\C^d)$-valued functions on $\R^n$ that is closed under the (pointwise) Minkowski sum and (pointwise) multiplication with complex numbers into a space of $\cK_{\mathrm{cs}}(\C^d)$-valued functions on $\R^n$ is said to be \emph{sublinear}, if for all $F, G$ in the domain of definition of $T$ and for all $a\in\C$ we have
\begin{equation*}
    T(F+G)(x)\subseteq TF(x)+TG(x)
\end{equation*}
and
\begin{equation*}
    T(aF)(x)=aTF(x)=|a|TF(x),
\end{equation*}
for a.e.~$x\in\R^n$.

\item A map $T$ acting from a space of $\cK_{\mathrm{cs}}(\C^d)$-valued functions on $\R^n$ into a space of $\cK_{\mathrm{cs}}(\C^d)$-valued functions on $\R^n$ is said to be \emph{monotone} if $TF(x)\subseteq TG(x)$ for a.e. $x\in\R^n$, whenever $F,G$ are in the domain of definition of $T$ such that $F(x)\subseteq G(x)$ for a.e.~$x\in\R^n$.
\end{enumerate}
\end{dfn}

\begin{dfn}
If $F:\R^n\to\cK_{\bcs}(\C^d)$ is a locally integrable function, then following \cite{Cruz_Uribe_Bownik_Extrapolation} we define its \emph{convex-set-valued maximal function} by
\begin{equation*}
    \cM^{\cK}F(x):=\mathrm{clos}\bigg(\mathrm{conv}\bigg(\bigcup_{\substack{Q\text{ cube in }\R^n\\x\in Q}}\aver{F}_{Q}\bigg)\bigg)
\end{equation*}
where $\aver{F}_{Q}:=\frac{1}{|Q|}\int_{Q}F(x)\dd x$. Observe that $\cM^{\cK}F(x)\in\cK_{
\mathrm{cs}}(\C^d)$ for a.e.~$x\in\R^n$. It is proved in \cite[Section 5]{Cruz_Uribe_Bownik_Extrapolation} that this gives rise to an operator $\cM^{\cK}$ that is sublinear and monotone moreover satisfies $F(x)\subseteq\cM^{\cK}F(x)$, for a.e.~$x\in\R^n$.
\end{dfn}

\begin{rem}
  Let $F:\R^n\to\cK_{\bcs}(\C^d)$ be a locally integrable function that is not a.e.~equal to $\{0\}$. Then, for all $x\in\R^n$, there exists a cube $Q\subseteq\R^n$ with $x\in Q$ such that $F$ is not a.e.~equal to $\{0\}$ on $Q$, therefore by \Cref{rem:Aumann_integral_nontrivial} we deduce $\aver{F}_{Q}\neq\{0\}$. It follows that $\cM^{\cK}F(x)\neq\{0\}$, for a.e.~$x\in\R^n$.
\end{rem}

We need the following weighted bound for the maximal function $\cM^{\cK}$ proved in \cite{Cruz_Uribe_Bownik_Extrapolation}.

\begin{letteredtheorem}[\protect{\cite[Theorem 6.9]{Cruz_Uribe_Bownik_Extrapolation}}]
\label{thm:boundconvexsetmaximalfunction}
    Let $1<p<\infty$ and let $W$ be a $d\times d$ matrix weight on $\R^n$. Then, $\cM^{\cK}$ acts boundedly on $L^{p}_{\cK}(W)$, $\cM^{\cK}:L^{p}_{\cK}(W)\to L^{p}_{\cK}(W)$ with
    \begin{equation*}
        \norm{\cM^{\cK}}_{L^{p}_{\cK}(W)}\lesssim_{d,n,p}[W]_{A_{p}}^{1/(p-1)}.
    \end{equation*}
\end{letteredtheorem}

\subsubsection{Rubio de Francia algorithm for convex body valued functions}

\begin{dfn}
Let $W:\R^n\to \mathrm{M}_{d}(\C)$ be a measurable function with values a.e.~in the set of positive-definite $d \times d$ complex-valued matrices. Following Bownik and Cruz-Uribe \cite{Cruz_Uribe_Bownik_Extrapolation}, we define
\begin{equation*}
 [W]_{A_1}:=\sup_{Q}\esssup_{x\in Q}\dashint|W^{-1}(x)W(y)|\dd y,
\end{equation*}
where the supremum ranges over all cubes $Q\subseteq\R^n$. If $[W]_{A_1}<\infty$, then we say that $W$ is a $d\times d$ matrix $A_1$ weight, in symbols $W\in A_1$. Observe that $[W]_{A_1}=[W^{-1}]_{A_{\infty}}$.
\end{dfn}

\begin{dfn}
Following \cite{Cruz_Uribe_Bownik_Extrapolation}, we say that a locally integrable function $F:\R^n\to\cK_{\bcs}(\C^d)$ belongs to the $A_1^{\cK}$ class, in symbols $F\in A^{\cK}_1$, if there exists $0<C<\infty$ such that
\begin{equation*}
    \cM^{\cK}F(x)\subseteq F(x),
\end{equation*}
for a.e.~$x\in\R^n$; in this case, we denote the infimum of all such constants $C$ by $[F]_{A_{1}^{\cK}}$.
\end{dfn}

In what follows, we denote by $\overline{\mathbf{B}}$ the closed unit ball in $\C^d$. Bownik and Cruz-Uribe \cite[Corollary 7.4]{Cruz_Uribe_Bownik_Extrapolation} describe precisely the relation between the $A_1^{\cK}$ class and matrix $A_1$ weights.

\begin{lemabc}[\protect{\cite[Corollary 7.4]{Cruz_Uribe_Bownik_Extrapolation}}]
    \label{thm:twoA1classes}
    If $W$ is a $d\times d$ matrix weight on $\R^n$, then $W\overline{\mathbf{B}}\in A_1^{\cK}$ if and only if $W\in A_1$.
\end{lemabc}

Bownik and Cruz-Uribe \cite[Theorem 7.6]{Cruz_Uribe_Bownik_Extrapolation} establish a Rubio de Francia-type algorithm to prove their extrapolation algorithm. This algorithm is of fundamental importance also for us, so we recall it below. We refer to \cite[Theorem 7.6]{Cruz_Uribe_Bownik_Extrapolation} for a detailed proof.

\begin{letteredtheorem}[\protect{\cite[Theorem 7.6]{Cruz_Uribe_Bownik_Extrapolation}}]
\label{thm:RubiodeFranciaalgorithm}
    Let $1\leq p\leq\infty$ and let $W$ be a $d\times d$ matrix weight on $\R^n$. Let $T:L^{p}_{\cK}(W)\to L^{p}_{\cK}(W)$ be a bounded sublinear and monotone map. Given $G\in L^{p}_{\cK}(W)$, define $SG$ by
    \begin{equation*}
        SG:=\sum_{k=0}^{\infty}\frac{1}{2^k\norm{T}_{L^{p}_{\cW}(K)}^{k}}T^{k}G,
    \end{equation*}
    where $T^0G:=G$ and $T^k$ denotes the $k$-fold composition of $T$ with itself. Then, this series converges in $L^{p}_{\cK}(W)$. Moreover, it has the following properties:
    \begin{enumerate}
        \item $G(x)\subseteq SG(x)$, for a.e.~$x\in\R^n$
        \item $\norm{SG}_{L^{p}_{\cK}(W)}\leq 2\norm{G}_{L^{p}_{\cK}(W)}$
        \item $T(SG)(x)\subseteq 2\norm{T}_{L^{p}_{\cK}(W)}SG(x)$, for a.e.~$x\in\R^n$.
    \end{enumerate}
\end{letteredtheorem}

\begin{dfn}
If $F:\R^n\to\cK_{\bcs}(\C^d)$ is a measurable function and $W:\R^n\to M_{d}(\C)$ is a measurable function, then following \cite{Cruz_Uribe_Bownik_Extrapolation} we way that $F$ is \emph{ellipsoid-valued with respect to $W$} is there exists a function $r:\R^n\to\R$ with $F(x)=r(x)W(x)$, for a.e.~$x\in\R^n$.
\end{dfn}

Cruz-Uribe and Bownik \cite[Lemma 8.5]{Cruz_Uribe_Bownik_Extrapolation} show that their Rubio de Francia algorithm in \Cref{thm:RubiodeFranciaalgorithm} behaves well for input functions that are ellipsoid-valued with respect to a given matrix function. 

\begin{lemabc}[\protect{\cite[Lemma 8.5]{Cruz_Uribe_Bownik_Extrapolation}}]
\label{thm:preservationofellipsoidvalued}
    Let $V:\R^n\to M_{d}(\C)$ be a measurable function and $T$ be an operator as in \Cref{thm:RubiodeFranciaalgorithm} such that whenever $G\in L^{p}_{\cK}(W)$ is ellipsoid-valued with respect to $V$, then $T(G)$ is also ellipsoid-valued with respect to $V$. Then, $SG$ is also ellipsoid-valued with respect to $V$, for all $G\in L^{p}_{\cK}(W)$ that are ellipsoid-valued with respect to $V$.
\end{lemabc}

Finally, we need the following reverse Jones factorization type lemma for matrix weights \cite[Corollary 8.9]{Cruz_Uribe_Bownik_Extrapolation}, which is itself a special case of a much more general result \cite[Proposition 8.8]{Cruz_Uribe_Bownik_Extrapolation}.

\begin{letteredtheorem}[\protect{\cite[Corollary 8.9]{Cruz_Uribe_Bownik_Extrapolation}}]
\label{thm:reverseJonesfactorization}
Let $1<p<\infty$, and let $W$ be a $d\times d$ matrix weight on $\R^n$.

\begin{enumerate}
   \item Let $s:\R^n\to[0,\infty)$ be a measurable function such that $W_1:=sW^{1/p}\in A_{\infty}$. Then, for all $p_0\in(p,\infty)$, we have
    \begin{equation*}
        \overline{W}:=s^{p_0-p}W^{p_0/p}\in A_{p_0},
    \end{equation*}
    with
    \begin{equation*}
        [\overline{W}]_{A_{p_0}}\lesssim_{d,n,p,p_0}[W]_{A_p}[W_1]_{A_{\infty}}^{p_0-p}.
    \end{equation*}

    \item Let $r:\R^n\to[0,\infty)$ be a measurable function such that $W_0:=rW^{1/p}\in A_{1}$. Then, for all $p_0\in(1,p)$, we have
    \begin{equation*}
        \overline{W}:=r^{p_0-p'(p_0-1)}W^{p_0/p}\in A_{p_0},
    \end{equation*}
    with
    \begin{equation*}
        [\overline{W}]_{A_{p_0}}\lesssim_{d,n,p,p_0}[W_0]_{A_{1}}^{p_0-p'(p_0-1)}[W]_{A_p}^{(p_0-1)/(p-1)}.
    \end{equation*}
    \end{enumerate}
\end{letteredtheorem}

\subsection{Statement of main result}

In the rest of this section, $\cF$ denotes a non-empty family of pairs $(F,G)$ of measurable functions $F,G:\R^n\to\cK_{\textrm{bcs}}(\C^d)$ such that neither $F$ nor $G$ is a.e.~equal to $\{0\}$. Whenever we write an inequality of the form
\begin{equation*}
    \norm{F}_{L^{p}_{\cK}(W)}\leq C\norm{G}_{L^{p}_{\cK}(W)},\quad (F,G)\in\cF
\end{equation*}
we understand that the inequality holds for all pairs $(F,G)\in\cF$ with $\norm{F}_{L^{p}_{\cK}(W)}<\infty$. The constant $C$ depends on some concrete Muckenhoupt characteristic of some concrete power of the weight $W$, but not on the particular weight $W$ itself.

We need some notation. Let $0<p_0<q_0\leq\infty$. For all $p\in[p_0,q_0)$, we set
\begin{equation*}
    t(p):=\frac{p}{p_0}\geq 1,\quad s(p):=\frac{q_0}{p}\geq 1,\quad
    r(p):=s(p)'(t(p)-1)+1\geq 1,
\end{equation*}
where the usual convention $\infty'=1$ is used; notice that $1<s(p)\leq\infty$, so $1\leq(s(p))'<\infty$. Observe that for all $p,q\in[p_0,q_0)$ with $p<q$, we have $r(p)<r(q)$.

\newtheorem*{thm:limited_range_extrapolation_matrix_weights}{\Cref{thm:limited_range_extrapolation_matrix_weights}}
\begin{thm:limited_range_extrapolation_matrix_weights}
    Let $0<p_0<q_0\leq\infty$. Assume that for some fixed $p\in[p_0,q_0]$, there exists an increasing function $K_{p}$ such that one of the following holds:
    \begin{enumerate}
        \item If $p<q_0$, then
        \begin{equation*}
        \norm{F}_{L^{p}_{\cK}(W)}\leq K_{p}([W^{s(p)'}]_{A_{r(p)}})\norm{G}_{L^{p}_{\cK}(W)},\quad (F,G)\in\cF,
        \end{equation*}
    for all $d\times d$ matrix weights $W$ on $\R^n$ with $[W^{s(p)'}]_{A_{r(p)}}<\infty$.
        \item If $p=q_0<\infty$, then
        \begin{equation*}
        \norm{F}_{L^{p}_{\cK}(W)}\leq K_{p}([W^{t(p)'/t(p)}]_{A_{\infty}})\norm{G}_{L^{p}_{\cK}(W)},\quad (F,G)\in\cF,
        \end{equation*}
    for all $d\times d$ matrix weights $W$ on $\R^n$ with $[W^{t(p)'/t(p)}]_{A_{\infty}}<\infty$.
    \item If $p=q_0=\infty$, then
        \begin{equation*}
        \norm{F}_{L^{p}_{\cK}(W)}\leq K_{p}([W^{1/p_0}]_{A_{\infty}})\norm{G}_{L^{p}_{\cK}(W)},\quad (F,G)\in\cF,
        \end{equation*}
    for all $d\times d$ matrix weights $W$ on $\R^n$ with $[W^{1/p_0}]_{A_{\infty}}<\infty$.
    \end{enumerate}
     Then, for all $q\in(p_0,q_0)$,
    \begin{equation*}
        \norm{F}_{L^{q}_{\cK}(W)}\leq K_{q}([W^{s(q)'}]_{A_{r(q)}})\norm{G}_{L^{q}_{\cK}(W)},\quad (F,G)\in\cF
    \end{equation*}
    holds, for all $d\times d$ matrix weights $W$ on $\R^n$ with $[W^{s(q)'}]_{A_{r(q)}}<\infty$, where
    \begin{equation*}
        K_{q}([W^{s(q)'}]_{A_{r(q)}})=c_1K_{p}(c_2[W^{s(q)'}]_{A_{r(q)}}^{\alpha(p,q)}),
    \end{equation*}
    for some constants $c_1,c_2>0$ depending only on $n,d,p_0,q_0,p$ and $q$, and the exponent $\alpha(p,q)$ is given by
    \begin{equation}
    \label{eq:extrapolation_exponent}
        \alpha(p,q):=
        \begin{cases}
            \frac{r(p)-1}{r(q)-1},\text{ if } q<p<q_0\\\\
            \frac{1}{r(q)-1},\text{ if }q<p=q_0\\\\
            1,\text{ if }p\leq q
        \end{cases}
        .
    \end{equation}
\end{thm:limited_range_extrapolation_matrix_weights}

\begin{rem}
    Given the comparison results proved in \Cref{section:matrix-weight}, the conditions on the weights considered in the above theorem are reasonable. In particular, in the case $d=1$ we recover Theorem \Cref{thm:original_limited_range_extrapolation}. Moreover, in the case $p_0=1$ and $q_0=\infty$ we recover \cite[Theorem 9.1]{Cruz_Uribe_Bownik_Extrapolation}.
\end{rem}

\subsection{Proof of main result}

Here we prove \Cref{thm:limited_range_extrapolation_matrix_weights} in detail. We first need the following technical lemma.

\begin{lem}
    Let $p,q\in[p_0,q_0)$. The following two identities are true:
    \begin{equation}
    \label{first identity}
    s(p)'s(q)'(p-q)=ps(p)'-qs(q)'
    \end{equation}
and
    \begin{equation}
    \label{second identity}
    \frac{r(p)}{r(q)}=\frac{ps(p)'}{qs(q)'}.
    \end{equation}
\end{lem}

\begin{proof}
If $q_0=\infty$, then $s(p)'=s(q)'=1$, so the first identity holds trivially, and the second says $\frac{t(p)}{t(q)}=\frac{p}{q}$, which also holds trivially because $t(p)=\frac{p}{p_0}$ and $t(q)=\frac{q}{p_0}$.

Now we prove the identities under the assumption that $q_0<\infty$. For the first identity we have
\begin{equation*}
    s(p)'=(q_0/p)'=\frac{q_0}{q_0-p},\quad s(q)'=\frac{q_0}{q_0-q},
\end{equation*}
so what we want to see becomes
\begin{equation*}
    \frac{q_0^2(p-q)}{(q_0-p)(q_0-q)}=\frac{pq_0}{q_0-p}-\frac{qq_0}{q_0-q}.
\end{equation*}
Computing
\begin{equation*}
    pq_0(q_0-q)-qq_0(q_0-p)=q_0^2p-pqq_0-qq_0^2+pqq_0=q_0^2(p-q),
\end{equation*}
we obtain the first identity.

The second identity can be equivalently rewritten as
\begin{align*}
    qs(q)'r(p)=ps(p)'r(q),
\end{align*}
that is
\begin{equation*}
    qs(q)'[s(p)'(t(p)-1)+1]=ps(p)'[s(q)'(t(q)-1)+1],
\end{equation*}
that is after expanding and rearranging
\begin{equation*}
    s(p)'s(q)'(qt(p)-q-pt(q)+p)=ps(p)'-qs(q)'.
\end{equation*}
We compute
\begin{equation*}
    qt(p)-q-pt(q)+p=\frac{qp}{p_0}-q-\frac{pq}{p_0}+p=p-q.
\end{equation*}
Thus, the second identity is equivalent to the first one, which concludes the proof.
\end{proof}

\begin{proof}[Proof (of \Cref{thm:limited_range_extrapolation_matrix_weights})]

We combine and adapt the proofs of \cite[Theorem 4.9]{Auscher_Martell_I_2007} and \cite[Theorem 9.1]{Cruz_Uribe_Bownik_Extrapolation}.

Let $q\in(p_0,q_0)$, $W$ be a $d\times d$ matrix weight with $[W^{s(q)'}]_{A_{r(q)}}<\infty$, and $(F,G)\in\cF$ with $0<\norm{F}_{L^{q}_{\cK}(W)}<\infty$. We may assume that $q\neq p$. We may also assume that $0<\norm{F}_{L^{q}_{\cK}(W)}<\infty$, otherwise the inequality to show is trivial. Throughout the proof, we suppress the notation dependence of implied constants on $n,d,p_0,q_0,p$, and $q$. Moreover, $c,c_1,c_2$ are positive finite constants depending only on these parameters which might change from line to line.

We now distinguish several cases.

\item[\bf Case 1.] $q<p<q_0$. We first need to define an associated Rubio de Francia operator. Thus, consider the following auxiliary operator $N_{W}$ acting on functions $H\in L^{r(q)}_{\cK}(V)$:
\begin{equation*}
    N_{W}H(x):=|V(x)^{1/r(q)}H(x)|W(x)^{-1/r(q)}\overline{\mathbf{B}},\quad\text{for a.e. }x\in\R^n,
\end{equation*}
where $V:=W^{s(q)'}$. For this operator we may compute
\begin{align*}
    \norm{N_{W}H}_{L^{r(q)}_{\cK}(W)}=\norm{H}_{L^{r(q)}_{\cK}(V)}.
\end{align*}
It is easy to see that $N_{W}$ is sublinear and monotone.

Furthermore, we claim that
\begin{equation*}
    H(x)\subseteq \Sigma(x)N_{W}H(x),\quad\text{for a.e. }x\in\R^n,
\end{equation*}
where $\Sigma:=W^{1/r(q)}\cdot V^{-1/r(q)}=W^{(1-s(q)')/r(q)}$. Indeed, this inclusion is equivalent to
\begin{equation*}
    H(x)\subseteq|V(x)^{1/r(q)}H(x)|V(x)^{-1/r(q)}\overline{\mathbf{B}},\quad\text{for a.e. }x\in\R^n,
\end{equation*}
that is
\begin{equation*}
    V(x)^{1/r(q)}H(x)\subseteq|V(x)^{1/r(q)}H(x)|\overline{\mathbf{B}},\quad\text{for a.e. }x\in\R^n,
\end{equation*}
which is obvious.

Now we define the operator $P_{W}$ acting on functions $H\in L^{r(q)}_{\cK}(W)$ by
\begin{equation*}
    P_{W}H(x):=N_{W}(\cM^{\cK}(\Sigma H))(x)=|V(x)^{1/r(q)}\cM^{\cK}(\Sigma H)(x)|W(x)^{-1/r(q)}\overline{\mathbf{B}},\quad\text{for a.e. }x\in\R^n.
\end{equation*}
Then, using \Cref{thm:boundconvexsetmaximalfunction} we can compute
\begin{align*}
    \norm{P_{W}H}_{L^{r(q)}_{\cK}(W)}=\norm{\cM^{\cK}(\Sigma H)}_{L^{r(q)}_{\cK}(V)}\lesssim[V]_{A_{r(q)}}^{1/(r(q)-1)}\norm{\Sigma H}_{L^{r(q)}_{\cK}(V)}=[V]_{A_{r(q)}}^{1/(r(q)-1)}\norm{H}_{L^{r(q)}_{\cK}(W)}.
\end{align*}
Since $\cM^{\cK}$ is itself sublinear and monotone, it is easy to see that $P_{W}$ is sublinear and monotone. (Although we do not need it, we note that it is easy to see that $H(x)\subseteq P_{W}H(x)$, for a.e.~$x\in\R^n$.) Note also that if $H$ is not a.e.~equal to $\{0\}$, then $P_{W}H(x)\neq\{0\}$, for a.e.~$x\in\R^n$.

All assumptions of \Cref{thm:RubiodeFranciaalgorithm} are fulfilled, thus we can define the following Rubio de Francia type operator acting on functions $H\in L^{r(q)}_{\cK}(W)$:
\begin{equation*}
    \cR_{W}H:=\sum_{k=0}^{\infty}\frac{1}{2^k\norm{P_{W}}_{L^{r(q)}_{\cK}(W)}^{k}}(P_{W})^{k}H.
\end{equation*}
We record the following properties of $\cR_{W}$ that we obtain from \Cref{thm:RubiodeFranciaalgorithm}:
\begin{equation*}
    H(x)\subseteq\cR_{W}H(x),
\end{equation*}
for a.e.~$x\in\R^n$,
\begin{equation*}
    \norm{\cR_{W}H}_{L^{r(q)}_{\cK}(W)}\leq 2\norm{H}_{L^{r(q)}_{\cK}(W)}
\end{equation*}
and
\begin{equation*}
    P_{W}(\cR_{W}H)(x)\subseteq\norm{P_{W}}_{L^{r(q)}_{\cK}(W)}\cR_{W}H(x)\subseteq c[V]_{A_{r(q)}}^{1/(r(q)-1)}\cR_{W}H(x),
\end{equation*}
for a.e. $x\in\R^n$. In particular, we have
\begin{equation*}
    \cM^{\cK}(\Sigma\cR_{W}H)(x)\subseteq\Sigma(x)P_{W}(\cR_{W}H)(x)\subseteq c[V]_{A_{r(q)}}^{1/(r(q)-1)}\Sigma(x)\cR_{W}H(x),
\end{equation*}
for a.e.~$x\in\R^n$. This means that if $H$ is not a.e.~equal to $\{0\}$, then the convex body valued function $\Sigma\cR_{W}H$ belongs to the $A^{1}_{\cK}$ class with $[\Sigma\cR_{W}H]_{A^1_{\cK}}\leq c[V]_{A_{r(q)}}^{1/(r(q)-1)}$.

Finally, notice that $P_{W}H$ is always ellipsoid-valued with respect to the matrix valued function $W^{-1/r(q)}$, therefore, by \Cref{thm:preservationofellipsoidvalued}, we deduce that if $H$ happens to be ellipsoid-valued with respect to $W^{-1/r(q)}$, then $\cR_{W}H$ is also ellipsoid-valued with respect to $W^{-1/r(q)}$.

We now define the scalar-valued function
\begin{equation*}
    \bar{h}(x):=\left(\frac{|W(x)^{1/q}F(x)|}{\norm{F}_{L^{q}_{\cK}(W)}}+\frac{|W(x)^{1/q}G(x)|}{\norm{G}_{L^{q}_{\cK}(W)}}\right)^{q/r(q)},\quad\text{ for a.e. }x\in\R^n.
\end{equation*}
It is obvious that
\begin{equation*}
    \norm{\bar{h}}_{L^{r(q)}}\lesssim1.
\end{equation*}
We consider the convex body valued function
\begin{equation*}
    \bar{H}(x):=\bar{h}(x)W(x)^{-1/r(q)}\overline{\mathbf{B}},\quad\text{ for a.e. }x\in\R^n.
\end{equation*}
It is clear that $\norm{H}_{L^{r(q)}_{\cK}(W)}\leq2$. Since $\bar{H}$ is ellipsoid-valued with respect to $W^{-1/r(q)}$, we deduce that $\cR_{W}\bar{H}$ is also ellipsoid-valued with respect to $W^{-1/r(q)}$, therefore we can write $\cR_{W}\bar{H}(x)=\cR_{W}\bar{h}(x)W(x)^{-1/r(q)}\overline{\mathbf{B}}$ for some nonnegative scalar-valued function $\cR_{W}\bar{h}$. Then, we have
\begin{equation*}
\bar{h}(x)W(x)^{-1/r(q)}\overline{\mathbf{B}}\subseteq\cR_{W}\bar{h}(x)W(x)^{-1/r(q)}\overline{\mathbf{B}},
\end{equation*}
which implies immediately that $\bar{h}(x)\leq\cR_{W}\bar{h}(x)$, for a.e.~$x\in\R^n$. We now define the nonnegative scalar-valued function
\begin{equation*}
    h:=(\cR_{W}\bar{h})^{r(q)/q}.
\end{equation*}
Since $\bar{H}$ is not a.e.~equal to $\{0\}$, we deduce that $h(x)\neq0$, for a.e.~$x\in\R^n$. We compute
\begin{align*}
    &\left(\int_{\R^n}|W(x)^{1/q}F(x)|^{q}\dd x\right)^{1/q}=
    \left(\int_{\R^n}|h(x)^{-1/(p/q)'}W(x)^{1/q}F(x)|^{q}\cdot h(x)^{q/(p/q)'}\dd x\right)^{1/q}\\
   &\leq\left(\int_{\R^n}|h(x)^{-1/(p/q)'}W(x)^{1/q}F(x)|^{p}\dd  x\right)^{1/p}\left(\int_{\R^n}h(x)^{q}\dd  x\right)^{1/(q(p/q)')}.
\end{align*}
It is of course the case that
\begin{equation*}
    \int_{\R^n}h(x)^{q}\dd x=\int_{\R^n}(\cR_{W}\bar{h}(x))^{r(q)}\dd  x=\int_{\R^n}|W(x)^{1/r(q)}\bar{H}(x)|^{r(q)}\dd x\lesssim1.
\end{equation*}
Now let us deal with the first integral. Set
\begin{equation*}
    T(x):=h(x)^{-p/(p/q)'}W(x)^{p/q},\quad\text{ for a.e. }x\in\R^n.
\end{equation*}
What we want to check is that
\begin{equation*}
    T^{s(p)'}\in A_{r(p)}.
\end{equation*}
By above we have
\begin{align*}
    \Sigma\cR_{W}\bar{H}\in A^1_{\cK},
\end{align*}
that means concretely
\begin{equation*}
    (\cR_{W}\bar{h})V^{-1/r(q)}\overline{\mathbf{B}}\in A^1_{\cK}.
\end{equation*}
Thus, by \Cref{thm:twoA1classes} we have $\cR_{W}\bar{h}V^{-1/r(q)}\in A_1$, therefore,
\begin{equation*}
(\cR_{W}\bar{h})^{-1}V^{1/r(q)}\in A_{\infty},
\end{equation*}
that is
\begin{equation*}
    h^{-q/r(q)}W^{s(q)'/r(q)}\in A_{\infty}.
\end{equation*}
Set
\begin{equation*}
    m(x):=h(x)^{-q/r(q)},\quad\text{ for a.e. }x\in\R^n.
\end{equation*}
Then
\begin{equation*}
    mW^{s(q)'/r(q)}\in A_{\infty}\quad\text{and}\quad W^{s(q)'}\in A_{r(q)}.
\end{equation*}
The fact that $p>q$ easily implies $r(p)>r(q)$. Thus, by \Cref{thm:reverseJonesfactorization} (1), we deduce
\begin{equation*}
    m^{r(p)-r(q)}(W^{s(q)'})^{r(p)/r(q)}\in A_{r(p)}.
\end{equation*}
Our goal now is to check that
\begin{equation*}
    m^{r(p)-r(q)}(W^{s(q)'})^{r(p)/r(q)}=T^{s(p)'},
\end{equation*}
that is
\begin{equation*}
    h^{-q(r(p)-r(q))/r(q)}W^{s(q)'r(p)/r(q)}=h^{-s(p)'p/(p/q)'}W^{s(p)'p/q}.
\end{equation*}
We first check that the exponents of $h$ agree:
\begin{align*}
    &-\frac{q(r(p)-r(q))}{r(q)}=-\frac{qr(p)}{r(q)}+q\overset{\eqref{second identity}}{=}-\frac{qps(p)'}{qs(q)'}+q=
    -\frac{ps(p)'}{s(q)'}+q=-\frac{ps(p)'-qs(q)'}{s(q)'}\\
    &\overset{\eqref{first identity}}{=}-s(p)'(p-q)=
    -\frac{s(p)'p}{\frac{p}{p-q}}=-\frac{s(p)'p}{(p/q)'}.
\end{align*}
We also check that the exponents of $W$ agree:
\begin{align*}
    \frac{s(q)'r(p)}{r(q)}\overset{\eqref{second identity}}{=}\frac{ps(p)'}{q}=\frac{s(p)'p}{q}.
\end{align*}

Thus, we can estimate
\begin{align*}
    \left(\int_{\R^n}|W(x)^{1/q}F(x)|^{q}\dd x\right)^{1/q}\leq c\left(\int_{\R^n}|T(x)^{1/p}F(x)|^{p}\right)^{1/p}.
\end{align*}

We now observe that
\begin{equation*}
    h(x)=(\cR_{W}\bar{h}(x))^{r/q}\geq\bar{h}(x)^{r/q}\geq\frac{|W(x)^{1/q}F(x)|}{\norm{F}_{L^{q}_{\cK}(W)}},\quad\text{ for a.e. }x\in\R^n,
\end{equation*}
thus
\begin{align*}
    &\int_{\R^n}|T(x)^{1/p}f(x)|^{p}\dd x=\int_{\R^n}h(x)^{-p/(p/q)'}|W(x)^{1/q}F(x)|^{p}\dd x=\int_{\R^n}h(x)^{q-p}|W(x)^{1/q}F(x)|^{p}\dd x\\
    &\leq\frac{1}{\norm{F}_{L^{q}_{\cK}(W)}^{q-p}}\int_{\R^n}|W(x)^{1/q}F(x)|^{q}\dd x=\Vert F\Vert_{L^{q}_{\cK}(W)}^{p}<\infty.
\end{align*}
Therefore, we are entitled to apply the extrapolation hypothesis, obtaining
\begin{align*}
    \left(\int_{\R^n}|W(x)^{1/q}F(x)|^{q}\dd x\right)^{1/q}	&	\leq c_1\left(\int_{\R^n}|T(x)^{1/p}F(x)|^{p}\right)^{1/p}	\\
    &	\leq c_1K_{p}([T^{s(p)'}]_{A_{r(p)}})\left(\int_{\R^n}|T(x)^{1/p}G(x)|^{p}\right)^{1/p}\\
    &\leq c_1K_{p}(c_2[V]_{A_{r(q)}}^{\alpha(p,q)})\left(\int_{\R^n}|T(x)^{1/p}G(x)|^{p}\right)^{1/p},
\end{align*}
since
\begin{equation*}
    [T^{s(p)'}]_{A_{r(p)}}\leq c[W^{s(q)'}]_{A_{r(q)}}[s(x)W^{s(q)'/r(q)}]_{A_{\infty}}^{r(p)-r(q)}
    \leq c[V]_{A_{r(q)}}[V]_{A_{r(q)}}^{(r(p)-r(q))/(r(q)-1)}=c[V]_{A_{r(q)}}^{\alpha(p,q)}.
\end{equation*}
Similarly to above, we obtain
\begin{align*}
    \left(\int_{\R^n}|T(x)^{1/p}G(x)|^{p}\right)^{1/p}\leq\norm{G}_{L_{\cK}^{q}(W)}.
\end{align*}
Thus
\begin{equation*}
    \left(\int_{\R^n}|W(x)^{1/q}F(x)|^{q}\dd x\right)^{1/q}\leq c_1K_{p}(c_2[V]_{A_{r(q)}}^{\alpha(p,q)})\norm{G}_{L_{\cK}^{q}(W)},
\end{equation*}
concluding the proof in this case.

\item[\bf Case 2.] $q<p=q_0<\infty$. Similarly to Case 1, we construct a nonnegative scalar-valued function $h$ with the following properties:
\begin{equation*}
    h(x)\geq\frac{|W(x)^{1/q}F(x)|}{\norm{F}_{L^{q}_{\cK}(W)}}+\frac{|W(x)^{1/q}G(x)|}{\norm{G}_{L^{q}_{\cK}(W)}},\quad\text{ for a.e. }x\in\R^n,
\end{equation*}
\begin{equation*}
    h(x)>0,\quad\text{for a.e. }x\in\R^n,
\end{equation*}
\begin{equation*}
    \int_{\R^n}h(x)^{q}\dd x\lesssim1,
\end{equation*}
and
\begin{equation*}
    h^{q/r(q)}W^{-s(q)'/r(q)}\in A_1.
\end{equation*}
As in case 1 we have
\begin{equation*}
    \norm{F}_{L^{q}_{\cK}(W)}\lesssim\norm{F}_{L^{p}_{\cK}(T)},\quad  \norm{F}_{L^{p}_{\cK}(T)}\leq\norm{F}_{L^{q}_{\cK}(W)}\quad\text{ and }\quad\norm{G}_{L^{p}_{\cK}(T)}\leq\norm{G}_{L^{q}_{\cK}(W)},
\end{equation*}
where
\begin{equation*}
    T(x):=h(x)^{-p/(p/q)'}W(x)^{p/q},\quad\text{ for a.e. }x\in\R^n.
\end{equation*}
Notice that
\begin{equation*}
    B:=h^{-q/r(q)}W^{s(q)'/r(q)}\in A_{\infty}.
\end{equation*}
We check now that
\begin{equation*}
    T=B^{s(p_0)-1},
\end{equation*}
that is
\begin{equation*}
    h^{-p/(p/q)'}W^{p/q}=h^{-q(s(p_0)-1)/r(q)}W^{s(q)'(s(p_0)-1)/r(q)}.
\end{equation*}
We first check that the exponents of $h$ agree. We need to show that
\begin{equation*}
    -q_0/(q_0/q)'=-\frac{q}{r(q)}\left(\frac{q_0}{p_0}-1\right),
\end{equation*}
that is that
\begin{equation*}
    q_0-q=\frac{q(q_0-p_0)}{p_0r(q)}.
\end{equation*}
We compute
\begin{align*}
    p_0r(q)=s(q)'(q-p_0)+p_0=\frac{q_0(q-p_0)}{q_0-q}+p_0=\frac{q_0(q-p_0)+p_0(q_0-q)}{q_0-q}=\frac{q(q_0-p_0)}{q_0-q},
\end{align*}
yielding the desired equality. Moreover, we check that the exponents of $W$ agree. We have
\begin{align*}
    \frac{s(q)'(s(p_0)-1)}{r(q)}=\frac{q_0(q_0-p_0)}{p_0r(q)}=\frac{q_0}{q}=\frac{p}{q}.
\end{align*}
Thus
\begin{equation*}
    T^{t(q_0)'/t(q_0)}=T^{s(p_0)'-1}=B\in A_{\infty}.
\end{equation*}
The rest of the proof proceeds exactly as in case 1.

\item[\bf Case 3.] $q<p=q_0=\infty$. Similarly to Case 1, we construct a nonnegative scalar-valued function $h$ with the following properties:
\begin{equation*}
    h(x)\geq\frac{|W(x)^{1/q}F(x)|}{\norm{F}_{L^{q}_{\cK}(W)}}+\frac{|W(x)^{1/q}G(x)|}{\norm{G}_{L^{q}_{\cK}(W)}},\quad\text{ for a.e. }x\in\R^n,
\end{equation*}
\begin{equation*}
    h(x)>0,\quad\text{for a.e. }x\in\R^n,
\end{equation*}
\begin{equation*}
    \int_{\R^n}h(x)^{q}\dd x\lesssim1,
\end{equation*}
and
\begin{equation*}
    h^{q/r(q)}W^{-s(q)'/r(q)}\in A_1.
\end{equation*}
We compute
\begin{align*}
    &\left(\int_{\R^n}|W(x)^{1/q}F(x)|^{q}\dd x\right)^{1/q}=
    \left(\int_{\R^n}|h(x)^{-1/(p/q)'}W(x)^{1/q}F(x)|^{q}\cdot h(x)^{q/(p/q)'}\dd x\right)^{1/q}\\
   &\leq\norm{F}_{L^{\infty}_{\cK}(T)}\left(\int_{\R^n}h(x)^{q}\dd  x\right)^{1/(q(p/q)')}\lesssim\norm{F}_{L^{\infty}_{\cK}(T)},
\end{align*}
where
\begin{equation*}
    T(x):=h(x)^{-1/(p/q)'}W(x)^{1/q},\quad\text{ for a.e. }x\in\R^n.
\end{equation*}
Notice that
\begin{equation*}
    B:=h^{-q/r(q)}W^{s(q)'/r(q)}\in A_{\infty}.
\end{equation*}
So, we only have to check that $T^{1/p_0}=B$. We check that the exponents of $h$ agree:
\begin{equation*}
    -q/(p_0r(q))=-q/(p_0t(q))=-q/q=-1,\quad -1/(p/q)'=-1/1=-1.
\end{equation*}
We also check that the exponents of $W$ agree:
\begin{equation*}
    s(q)'/r(q)=1/t(q)=p/q_0.
\end{equation*}
Thus $T^{1/p_0}=B\in A_{\infty}$. The rest of the proof proceeds exactly as in case 1.

\item $p_0\leq p<q$. We first need to define an associated Rubio de Francia operator. So consider first the following auxiliary operator acting on convex-body valued functions $H\in L^{r(q)'}(U)$:
\begin{equation*}
    N_{W}H(x):=|U(x)^{1/r(q)'}H(x)|W(x)^{-1/r(q)'}\overline{\mathbf{B}},\quad\text{ for a.e. }x\in\R^n,
\end{equation*}
where $U:=W^{-1/(t(q)-1)}$. For this operator, we may compute
\begin{align*}
    \norm{N_{W}H}_{L^{r(q)'}_{\cK}(W)}=\norm{H}_{L^{r(q)'}_{\cK}(U)}.
\end{align*}
It is easy to see that this map is sublinear and monotone and
\begin{equation*}
    H(x)\subseteq \Sigma(x)N_{W}H(x),\quad\text{ for a.e. }x\in\R^n,
\end{equation*}
where $\Sigma:=W^{1/r(q)'}\cdot U^{-1/r(q)'}=W^{t(q)'/r(q)'}$.

Now we define the operator acting on functions $H\in L^{r(q)'}(W)$
\begin{equation*}
    P_{W}H(x):=N_{W}(\cM^{\cK}(\Sigma H))(x)=|U(x)^{1/r(q)'}\cM^{\cK}(\Sigma H)(x)|W(x)^{-1/r(q)'}\overline{\mathbf{B}},\quad\text{ for a.e. }x\in\R^n.
\end{equation*}
Since $\cM^{\cK}$ is sublinear and monotone, one can easily observe that $P_{W}$ is sublinear and monotone. Moreover, using \Cref{thm:boundconvexsetmaximalfunction} we can compute
\begin{align*}
    \norm{P_{W}H}_{L^{r(q)'}_{\cK}(W)}=\norm{\cM^{\cK}(\Sigma H)}_{L^{r(q)'}_{\cK}(U)}\lesssim[U]_{A_{r(q)'}}^{1/(r(q)'-1)}\norm{\Sigma H}_{L^{r(q)'}_{\cK}(U)}=[U]_{A_{r(q)'}}^{1/(r(q)'-1)}\norm{H}_{L^{r(q)'}_{\cK}(W)}.
\end{align*}
(Although we do not need it, we note that it is easy to see that $H(x)\subseteq P_{W}H(x)$, for a.e.~$x\in\R^n$.) Note also that if $H$ is not a.e.~equal to $\{0\}$, then $P_{W}H(x)\neq\{0\}$, for a.e.~$x\in\R^n$.

Thus, all assumptions of \Cref{thm:RubiodeFranciaalgorithm} are fulfilled, therefore we can define the following Rubio de Francia type operator acting on functions $H\in L^{r(q)'}_{\cK}(W)$:
\begin{equation*}
    \cR_{W}H:=\sum_{k=0}^{\infty}\frac{1}{2^k\norm{P_{W}}_{L^{r(q)'}_{\cK}(W)}^{k}}(P_{W})^{k}H.
\end{equation*}
We record the following properties of $\cR_{W}$ that we obtain from \Cref{thm:RubiodeFranciaalgorithm}:
\begin{equation*}
    H(x)\subseteq\cR_{W}H(x),\quad\text{ for a.e. }x\in\R^n,
\end{equation*}
\begin{equation*}
    \norm{\cR_{W}H}_{L^{r(q)'}_{\cK}(W)}\leq 2\norm{H}_{L^{r(q)'}_{\cK}(W)}
\end{equation*}
and
\begin{equation*}
    P_{W}(\cR_{W}H)(x)\subseteq\norm{P_{W}}_{L^{r(q)'}_{\cK}(W)}\cR_{W}H(x)\subseteq c[U]_{A_{r(q)'}}^{1/(r(q)'-1)}\cR_{W}H(x),\quad\text{ for a.e. }x\in\R^n.
\end{equation*}
In particular, we have
\begin{equation*}
    \cM^{\cK}(\Sigma\cR_{W}H)(x)\subseteq\Sigma(x)P_{W}(\cR_{W}H)(x)\subseteq c[U]_{A_{r'}}^{1/(r'-1)}\Sigma(x)\cR_{W}H(x),\quad\text{ for a.e. }x\in\R^n.
\end{equation*}
This means that if $H$ is not a.e.~equal to $\{0\}$, then the convex body valued function $\Sigma\cR_{W}H$ belongs to the $A^{1}_{\cK}$ class with $[\Sigma\cR_{W}H]_{A^1_{\cK}}\leq c[U]_{A_{r(q)'}}^{1/(r(q)'-1)}$.

Finally, notice that $P_{W}H$ is always ellipsoid-valued with respect to $W^{-1/r(q)'}$, therefore by \Cref{thm:preservationofellipsoidvalued} we deduce that if $H$ happens to be ellipsoid-valued with respect to $W^{-1/r(q)'}$, then $\cR_{W}H$ is also ellipsoid-valued with respect to $W^{-1/r(q)'}$.

We now observe that the Riesz representation theorem yields that there exists a nonnegative scalar-valued function $h_1\in L^{(q/p)'}$ with $\norm{h_1}_{L^{(q/p)'}}=1$ such that
\begin{equation*}
    \int_{\R^n}|W(x)^{1/q}F(x)|^{q}\dd x=\int_{\R^n}(|W(x)^{1/q}F(x)|^{p})^{q/p}\dd x=\int_{\R^n}|W(x)^{1/q}F(x)|^{p}h_1(x)\dd x.
\end{equation*}
Now we define the nonnegative scalar-valued function
\begin{equation*}
    \bar{h}(x):=h_1(x)^{(q/p)'/r(q)'},\quad\text{ for a.e. }x\in\R^n,
\end{equation*}
and the convex body valued function
\begin{equation*}
    \bar{H}(x):=\bar{h}(x)W(x)^{-1/r(q)'}\overline{\mathbf{B}},\quad\text{ for a.e. }x\in\R^n.
\end{equation*}
It is clear that $\norm{H}_{L^{r(q)'}_{\cK}(W)}=1$. Since $\bar{H}$ is ellipsoid-valued with respect to $W^{-1/r(q)'}$, we deduce that $\cR_{W}\bar{H}$ is also ellipsoid-valued with respect to $W^{-1/r(q)'}$, therefore we can write $\cR_{W}\bar{H}(x)=\cR_{W}\bar{h}(x)W(x)^{-1/r(q)'}\overline{\mathbf{B}}$ for some nonnegative scalar-valued function $\cR_{W}\bar{h}$. Then, we have
\begin{equation*}
\bar{h}(x)W(x)^{-1/r(q)'}\overline{\mathbf{B}}\subseteq\cR_{W}\bar{h}(x)W(x)^{-1/r(q)'}\overline{\mathbf{B}},
\end{equation*}
which implies immediately that $\bar{h}(x)\leq\cR_{W}\bar{h}(x)$, for a.e.~$x\in\R^n$. We now define the nonnegative scalar-valued function
\begin{equation*}
    h:=(\cR_{W}\bar{h})^{r(q)'/(q/p)'}.
\end{equation*}
Since $\bar{H}$ is not a.e.~equal to $\{0\}$, we deduce that $h(x)\neq0$, for a.e.~$x\in\R^n$. Observe that $h\geq h_1$. Thus, we have
\begin{align*}
    \int_{\R^n}|W(x)^{1/q}F(x)|^{q}\dd x\leq\int_{\R^n}|W(x)^{1/q}F(x)|^{p}h(x)\dd x=\int_{\R^n}|h(x)^{1/p}W(x)^{1/q}F(x)|^{p}\dd x.
\end{align*}
Set
\begin{equation*}
    T(x):=h(x)W(x)^{p/q},\quad\text{ for a.e. }x\in\R^n.
\end{equation*}
Our goal is to show that
\begin{equation*}
    T^{s(p)'}\in A_{r(p)}.
\end{equation*}

By above, we have
\begin{align*}
    \Sigma\cR_{W}\bar{H}\in A^1_{\cK},
\end{align*}
that means concretely
\begin{equation*}
    (\cR_{W}\bar{h})U^{-1/r(q)'}\overline{\mathbf{B}}\in A^1_{\cK},
\end{equation*}
that is
\begin{equation*}
    (\cR_{W}\bar{h})W^{s(q)'/r(q)}\overline{\mathbf{B}}\in A^1_{\cK}.
\end{equation*}
Thus, by \Cref{thm:twoA1classes} we have $(\cR_{W}\bar{h})V^{1/r(q)}\in A_1$, that is
\begin{equation*}
    h^{(q/p)'/r(q)'}W^{s(q)'/r(q)}\in A_{1}.
\end{equation*}
Set
\begin{equation*}
    m(x):=h(x)^{(q/p)'/r(q)'},\quad\text{ for a.e. }x\in\R^n.
\end{equation*}
Then
\begin{equation*}
    mW^{s(q)'/r(q)}\in A_{1}\quad\text{and}\quad W^{s(q)'}\in A_{r(q)}.
\end{equation*}
The fact that $p<q$ easily implies $r(p)<r(q)$. Thus, by \Cref{thm:reverseJonesfactorization} (2), we deduce
\begin{equation*}
    m^{r(p)-r(q)'(r(p)-1)}(W^{s(q)'})^{r(p)/r(q)}\in A_{r(p)},
\end{equation*}
both for $p>p_0$ and $p=p_0$.

Our goal now is to check that
\begin{equation*}
    m^{r(p)-r(q)'(r(p)-1)}(W^{s(q)'})^{r(p)/r(q)}=T^{s(p)'},
\end{equation*}
that is
\begin{equation*}
    h^{(q/p)'[r(p)-r(q)'(r(p)-1)]/r(q)'}W^{s(q)'r(p)/r(q)}=h^{s(p)'}W^{s(p)'p/q}.
\end{equation*}
We first check that the exponents of $h$ agree:
\begin{align*}
    &(q/p)'\frac{r(p)-r(q)'(r(p)-1)}{r(q)'}=(q/p)'\left[\frac{r(p)}{r(q)'}-r(p)+1\right]
    =(q/p)'\left[r(p)\left(\frac{1}{r(q)'}-1\right)+1\right]\\
    &=(q/p)'\left(-\frac{r(p)}{r(q)}-1\right)
    \overset{\eqref{second identity}}{=}(q/p)'\left(-\frac{ps(p)'}{qs(q)'}-1\right)
    =(q/p)'\frac{qs(q)'-ps(p)'}{qs(q)'}\\
    &\overset{\eqref{second identity}}{=}(q/p)'\frac{s(p)'s(q)'(q-p)}{qs(q)'}
    =\frac{q}{q-p}\cdot\frac{s(p)'q-p}{q}=s(p)'.
\end{align*}
Moreover, we check that the exponents of $W$ agree:
\begin{align*}
    s(p)'\frac{r(p)}{r(q)}\overset{\eqref{second identity}}{=}s(p)'\frac{p}{q}.
\end{align*}

Therefore, we have
\begin{equation*}
    \int_{\R^n}|W(x)^{1/q}F(x)|^{q}\dd x\leq\int_{\R^n}|T(x)^{1/p}F(x)|^{p}\dd x.
\end{equation*}
We also observe that
\begin{align*}
    &\int_{\R^n}|T(x)^{1/p}F(x)|^{p}\dd x=\int_{\R^n}|W(x)^{1/q}F(x)|^{p}h(x)\dd x\\
    &\leq\left(\int_{\R^n}|W(x)^{1/q}F(x)|^{q}\dd x\right)^{p/q}\left(\int_{\R^n}h(x)^{(q/p)'}\dd x\right)^{1/(q/p)'}\\
    &=\left(\int_{\R^n}|W(x)^{1/q}F(x)|^{q}\dd x\right)^{p/q}\left(\int_{\R^n}\cR_{W}\bar{h}(x)^{r(q)'}\dd x\right)^{1/(q/p)'}.
\end{align*}
Notice that
\begin{align*}
    \int_{\R^n}\cR_{W}\bar{h}(x)^{r(q)'}\dd x=\int_{\R^n}|W(x)^{1/r(q)'}\cR_{W}\bar{H}(x)|^{r(q)'}\dd x\leq c.
\end{align*}
It follows that
\begin{align*}
    &\int_{\R^n}|T(x)^{1/p}F(x)|^{p}\dd x\leq c\left(\int_{\R^n}|W(x)^{1/q}F(x)|^{q}\dd x\right)^{p/q}<\infty.
\end{align*}
Thus, we are entitled to apply the extrapolation hypothesis, obtaining
\begin{align*}
    &\left(\int_{\R^n}|W(x)^{1/q}F(x)|^{q}\dd x\right)^{1/q}\leq c\left(\int_{\R^n}|T(x)^{1/p}F(x)|^{p}\dd x\right)^{1/p}\leq cK_{p}([T^{s(p)'}]_{A_{r(p)}})\\
    &\leq c_1K_{p}(c_2[V]_{A_{r(p)}}^{\alpha(p,q)})
    \left(\int_{\R^n}|T(x)^{1/p}G(x)|^{p}\dd x\right)^{1/p},
\end{align*}
because
\begin{align*}
    &[T^{s(p)'}]_{A_{r(p)}}\leq c[W^{s(q)'}]_{A_{r(q)}}^{(r(p)-1)/(r(q)-1)}[r(x)W(x)^{s(q)'/r(q)}]_{A_1}^{r(p)-r(q)'(r(p)-1)}\\
    &\leq c[V]_{A_{r(q)}}^{(r(p)-1)/(r(q)-1)}[U]_{A_{r(q)'}}^{[r(p)-r(q)'(r(p)-1)]/(r(q)'-1)}\\
    &\simeq c[V]_{A_{r(q)}}^{(r(p)-1)/(r(q)-1)}[V]_{A_{r(q)}}^{r(p)-r(q)'(r(p)-1)}\leq c_2[V]_{A_{r(q)}}^{\alpha(p,q)}.
\end{align*}
Similarly to the above, we obtain
\begin{equation*}
    \left(\int_{\R^n}|T(x)^{1/p}G(x)|^{p}\dd x\right)^{q}\leq c\left(\int_{\R^n}|W(x)^{1/q}G(x)|^{q}\dd x\right)^{1/q}.
\end{equation*}
Thus, finally, we deduce
\begin{align*}
    &\left(\int_{\R^n}|W(x)^{1/q}F(x)|^{q}\dd x\right)^{1/q}\leq c_1K_{p}(c_2[V]_{A_{r(p)}}^{\alpha(p,q)})
    \left(\int_{\R^n}|W(x)^{1/q}G(x)|^{q}\dd x\right)^{1/q},
\end{align*}
concluding the proof.
\end{proof}

\section{Appendix}
\label{section:appendix}

Here we prove \Cref{first lemma}:

\newtheorem*{first_lemma}{\Cref{first lemma}}
\begin{first_lemma}
Set
    \begin{equation*}
        C_n:=
        \begin{bmatrix}
            1&0\\
            0&n
        \end{bmatrix}
        ,\quad
        D_n:=
        \begin{bmatrix}
            1&\frac{1}{2n^{1/2}}\\
            \frac{1}{2n^{1/2}}&\frac{1}{n}
        \end{bmatrix}
        ,\quad n=1,2,\ldots.
    \end{equation*}
    Then, $C_n,D_n\in\mathrm{P}_{2}(\C)$ with
    \begin{equation*}
        \tr(C_nD_n),~\tr(C_n^{-1}D_n^{-1})\simeq 1,
    \end{equation*}
    for all $n=1,2,\ldots$. Moreover, for all $a>1$,
    \begin{equation*}
        \lim_{n\to\infty}\tr(C_n^{a}D_n^{a})=\infty.
    \end{equation*}
\end{first_lemma}

Before the proof, we note the identities
\begin{equation}
\label{matrix identity}
        \begin{bmatrix}
            x&z\\
            y&w
        \end{bmatrix}
        \begin{bmatrix}
            \lambda_1&0\\
            0&\lambda_2
        \end{bmatrix}
        \begin{bmatrix}
            x&y\\
            z&w
        \end{bmatrix}
    =
    \begin{bmatrix}
            \lambda_1x^2+\lambda_2z^2&\lambda_1xy+\lambda_2zw\\
            \lambda_1xy+\lambda_2zw&\lambda_1y^2+\lambda_2w^2
        \end{bmatrix}
\end{equation}
and
\begin{equation}
\label{identity}
    \tr\left(
    \begin{bmatrix}
        1&0\\
        0&\lambda
    \end{bmatrix}
    \begin{bmatrix}
            x&z\\
            y&w
        \end{bmatrix}
        \begin{bmatrix}
            \lambda_1&0\\
            0&\lambda_2
        \end{bmatrix}
        \begin{bmatrix}
            x&y\\
            z&w
        \end{bmatrix}
    \right)=\lambda_1x^2+\lambda_2z^2+\lambda\lambda_1y^2+\lambda\lambda_2w^2.
\end{equation}

\begin{proof}[Proof (of \Cref{first lemma})]
    Let $n$ be any positive integer. It is clear that $C_n,D_n\in\mathrm{P}_{2}(\C)$. We compute
    \begin{equation*}
        C_nD_n=
        \begin{bmatrix}
            1&\frac{1}{2n^{1/2}}\\
            \frac{n^{1/2}}{2}&1
        \end{bmatrix}
    \end{equation*}
    and
    \begin{equation*}
        C_n^{-1}D_n^{-1}=
        \begin{bmatrix}
            1&0\\
            0&\frac{1}{n}
        \end{bmatrix}
        \begin{bmatrix}
            \frac{4}{3}&-\frac{2n^{1/2}}{3}\\
            -\frac{2n^{1/2}}{3}&\frac{4n}{3}
        \end{bmatrix}
        =
        \begin{bmatrix}
            \frac{4}{3}&-\frac{2n^{1/2}}{3}\\
            -\frac{2}{3n^{1/2}}&\frac{4}{3}
        \end{bmatrix}
        ,
    \end{equation*}
    showing immediately that $\tr(C_nD_n),~\tr(C_n^{-1}D_n^{-1})\simeq1$.

    Next, we diagonalize $D_n$. Direct computation shows that it has the eigenvalues
    \begin{equation*}
        \lambda_1(n):=\frac{1}{2}\left(1+\frac{1}{n}+\sqrt{1-\frac{1}{n}+\frac{1}{n^2}}\right)\quad\text{and}\quad \lambda_2(n):=\frac{1}{2}\left(1+\frac{1}{n}-\sqrt{1-\frac{1}{n}+\frac{1}{n^2}}\right).
    \end{equation*}
    Notice that $\lambda_1(n)\simeq1$ and that
    \begin{equation*}
        1-\lambda_1(n)=\frac{1}{2}\left(1-\frac{1}{n}-\sqrt{1-\frac{1}{n}+\frac{1}{n^2}}\right)=
        \frac{1}{2}\cdot\frac{-\frac{1}{n}}{1-\frac{1}{n}+\sqrt{1-\frac{1}{n}+\frac{1}{n^2}}},
    \end{equation*}
    therefore
    \begin{equation*}
        |1-\lambda_1(n)|\simeq\frac{1}{n}.
    \end{equation*}
    A normalized eigenvector associated to $\lambda_1(n)$ is given by $e_1(n)=(x(n),y(n))$, where
    \begin{equation*}
        x(n):=\frac{\frac{1}{2n^{1/2}}}{\sqrt{\frac{1}{4n}+(\lambda_1(n)-1)^2}},\quad y(n):=\frac{\lambda_1(n)-1}{\sqrt{\frac{1}{4n}+(\lambda_1(n)-1)^2}}.
    \end{equation*}
    Observe that $|y(n)|\simeq\frac{1}{n^{1/2}}$. Let also $e_2(n)=(z(n),w(n))$ be a normalized eigenvector associated to $\lambda_2(n)$.
   
    Let now $a>1$ be arbitrary. We have
    \begin{equation*}
        D_n=
        \begin{bmatrix}
            x(n)&z(n)\\
            y(n)&w(n)
        \end{bmatrix}
        \begin{bmatrix}
            \lambda_1(n)&0\\
            0&\lambda_2(n)
        \end{bmatrix}
        \begin{bmatrix}
            x(n)&y(n)\\
            z(n)&w(n)
        \end{bmatrix}
        ,
    \end{equation*}
    therefore
    \begin{align*}
        D_n^{a}&=
        \begin{bmatrix}
            x(n)&z(n)\\
            y(n)&w(n)
        \end{bmatrix}
        \begin{bmatrix}
            \lambda_1(n)^{a}&0\\
            0&\lambda_2(n)^{a}
        \end{bmatrix}
        \begin{bmatrix}
            x(n)&y(n)\\
            z(n)&w(n)
        \end{bmatrix}
        .
    \end{align*}
    Thus, by \eqref{identity} it follows that
    \begin{equation*}
        \tr(C_n^{a}D_n^{a})=\lambda_1(n)^{a}x(n)^2+\lambda_2(n)^{a}z(n)^2+n^{a}\lambda_1(n)^{a}y(n)^2+n^{a}\lambda_2(n)^{a}w(n)^2\gtrsim_{a}n^{a-1}.
    \end{equation*}
    Thus, $\lim_{n\to\infty}\tr(C_n^{a}D_n^{a})=\infty$.
\end{proof}

Identities \eqref{matrix identity} and \eqref{identity} suggest yet another family of counterexamples to the converse of the Cordes inequality. Let namely
\begin{equation*}
    C_n:=
    \begin{bmatrix}
        1&0\\
        0&n(1+\log n)
    \end{bmatrix}
\end{equation*}
and
\begin{equation*}
        D_n:=
        \begin{bmatrix}
            \sqrt{1-\frac{1}{n}}&-\frac{1}{\sqrt{n}}\\
            \frac{1}{\sqrt{n}}&\sqrt{1-\frac{1}{n}}
        \end{bmatrix}
        \begin{bmatrix}
            1&0\\
            0&\frac{1}{n(1+\log n)}
        \end{bmatrix}
        \begin{bmatrix}
            \sqrt{1-\frac{1}{n}}&\frac{1}{\sqrt{n}}\\
            -\frac{1}{\sqrt{n}}&\sqrt{1-\frac{1}{n}}
        \end{bmatrix}
        ,
\end{equation*}
for all $n=1,2,\ldots$. Then
\begin{equation*}
    (D_n)_{11}\simeq 1,\quad (D_n)_{22}\simeq \frac{1}{n},
\end{equation*}
therefore
\begin{equation*}
    \tr(C_nD_n)\simeq1+\log n,
\end{equation*}
for all $n=1,2,\ldots$. Thus, $\lim_{n\to\infty}\tr(C_nD_n)=\infty$.

Let now $a\in(0,1)$ be arbitrary. Then, direct computation gives
\begin{equation*}
    (D_n^{a})_{11}\simeq_{a}1,\quad 0\leq(D_n^{a})_{12}=(D_n^{a})_{21}\lesssim_{a}\frac{1}{\sqrt{n}},\quad (D_n^{a})_{22}\simeq_{a}\frac{1}{[n(1+\log n)]^{a}},
\end{equation*}
and
\begin{equation*}
    (D_n^{-a})_{11}\simeq_{a}1,\quad |(D_n^{-a})_{12}|=|(D_n^{-a})_{21}|\lesssim_{a}\frac{[n(1+\log n)]^{a}}{\sqrt{n}},\quad (D_n^{-a})_{22}\simeq_{a}[n(1+\log n)]^{a},
\end{equation*}
for all $n=1,2,\ldots$. It follows that for all $m,n=1,2,\ldots$, we have
\begin{equation*}
    \tr(C_m^{-a}C_n^{a}),~\tr(C_m^{-a}D_n^{-a}),~\tr(D_m^{a}C_{n}^{a})\simeq_{a}\max\left(1,\frac{n(1+\log n)}{m(1+\log m)}\right)^{a}
\end{equation*}
and
\begin{align*}
    \tr(D_m^{a}D_n^{-a})&\lesssim_{a}\max\left(1,\frac{[n(1+\log n)]^{a}}{\sqrt{m}\cdot\sqrt{n}},\left(\frac{n(1+\log n)}{m(1+\log m)}\right)^{a}\right)\\
    &\lesssim_{a}\max\left(1,\left(\frac{n}{m}\right)^{1/2},\left(\frac{n(1+\log n)}{m(1+\log m)}\right)^{a}\right).
\end{align*}
Note that an elementary convexity argument gives
\begin{equation*}
    (1-\log 2)(x+y)(1+\log(x+y))\leq x(1+\log x)+y(1+\log y),
\end{equation*}
for all $x,y\in(0,\infty)$ with $x+y\geq1$.

\printbibliography																		%%	Dimitris

\end{document}